\documentclass[11pt,american,preprint]{elsarticle}
\makeatletter
\def\ps@pprintTitle{%
 \let\@oddhead\@empty
 \let\@evenhead\@empty
 \def\@oddfoot{\centerline{\thepage}}%
 \let\@evenfoot\@oddfoot}
\makeatother

\usepackage[margin=1in]{geometry}

\usepackage{xcolor}
\usepackage{lipsum}

\usepackage{changepage}
\usepackage{centernot}

\usepackage[T1]{fontenc}
\usepackage[latin9]{luainputenc}
\usepackage{amsmath}
\usepackage{amsthm}
\usepackage{amssymb}
\usepackage{stmaryrd}
\usepackage{esint}
\usepackage{nameref}
\usepackage{hyperref} 
\usepackage{cleveref} %\autoref \nameref 

\usepackage[shortlabels]{enumitem}
\usepackage{chngcntr}

\usepackage{mathrsfs}

\newsavebox{\foobox}

\usepackage{graphicx,amssymb}

\usepackage{array}
\newcolumntype{M}[1]{>{\centering\arraybackslash}m{#1}}
\usepackage{bbm} %indicatorFunction
\makeatletter
\numberwithin{equation}{section}
\theoremstyle{plain}
\newtheorem{thm}{\protect\theoremname}[section]

\theoremstyle{plain*}
\newtheorem*{thm*}{\protect\theoremname}

\theoremstyle{plain}
\newtheorem{lem}[thm]{\protect\lemmaname}
  
\theoremstyle{plain*}
\newtheorem*{lem*}{\protect\lemmaname}  
  
  \theoremstyle{plain}
  \newtheorem{prop}[thm]{\protect\propositionname}
  
    \theoremstyle{plain*}
  \newtheorem*{prop*}{\protect\propositionname}

\theoremstyle{remark}
\newtheorem{question}[thm]{Question}
\theoremstyle{remark*}
\newtheorem*{question*}{Question} 
\theoremstyle{remark}
\newtheorem{rem}[thm]{\protect\remarkname}
\theoremstyle{remark*}
\newtheorem*{rem*}{\protect\remarkname}
\theoremstyle{remark}
\newtheorem{example}[thm]{Example}
\theoremstyle{remark*}
\newtheorem*{example*}{\protect\examplename}

\theoremstyle{plain}
\newtheorem{cor}[thm]{\protect\corollaryname}
\providecommand{\corollaryname}{Corollary}
  
\theoremstyle{definition}
\newtheorem{defn}[thm]{Definition}  
\makeatother

\theoremstyle{plain}

\providecommand{\principlename}{Principle}

\theoremstyle{plain} % just in case the style had changed
\newcommand{\thistheoremname}{}
\newtheorem{genericthm}[thm]{\thistheoremname}

\newtheorem*{genericthm*}{\thistheoremname}
\newenvironment{namedthm*}[1]
  {\renewcommand{\thistheoremname}{#1}%
   \begin{genericthm*}}
  {\end{genericthm*}}

\makeatother

\usepackage{babel}
 \providecommand{\lemmaname}{Lemma}
  \providecommand{\propositionname}{Proposition}
  \providecommand{\remarkname}{Remark}
\providecommand{\theoremname}{Theorem}

\newcommand{\R}{\mathbb{R}}
\newcommand{\N}{\mathbb{N}}
\newcommand{\Q}{\mathbb{Q}}
\newcommand{\Z}{\mathbb{Z}}

\newcommand{\C}{\mathbb C}

\newcommand\precdot{\mathrel{\ooalign{$\prec$\cr
  \hidewidth\raise0ex\hbox{$\cdot\mkern0.5mu$}\cr}}}
\newcommand\preceqdot{\mathrel{\ooalign{$\preceq$\cr
  \hidewidth\raise0.225ex\hbox{$\cdot\mkern0.5mu$}\cr}}}

\usepackage{fancyhdr}
\usepackage{float}
\title{Sets of large values of polynomial multi-correlation functions}
\begin{document}
\author{V. Bergelson and R. Zelada}
\begin{abstract}
Let $p_1,...,p_\ell\in\Z[x_1,...,x_d]$ be non-constant polynomials with zero constant term. The ergodic theoretical proofs of the 
polynomial \cite{bergelsonLeibmaPolyVderWardenAndSz} and the IP-polynomial \cite{berMcCuIPPolySzemeredi} Szemer{\'e}di  theorems  as well as some of the ergodic-theoretical and combinatorial consequences of the Density Polynomial Hales-Jewett conjecture \cite{ERTaU} naturally lead to the study of \textit{sets of large returns} which are defined as 
$$
R_\epsilon^{p_1,...,p_\ell}(A):=\{\vec n\in\Z^d\,|\,\mu(A\cap T_1^{-p_1(\vec n)}A\cap\cdots\cap T_\ell^{-p_\ell(\vec n)}A)>\mu^{\ell+1}(A)-\epsilon\},
$$
where $T_1,...,T_\ell$ are commuting and invertible measure preserving transformations  on a probability space $(X,\mathcal A,\mu)$, $A\in \mathcal A$,  and $\epsilon>0$.\\
We obtain new results dealing with the fine properties of the sets of the form $R_\epsilon^{p_1,...,p_\ell}(A)$ which in turn provide new combinatorial applications.
Among other things, we show that 
%we prove that 
every set of the form $R_\epsilon^{p_1,...,p_\ell}(A)$ is syndetic if and only if the polynomials $p_1,...,p_\ell$ are linearly independent, answering a question asked by  N. Frantzikinakis and B. Kuca \cite{FraKuJoinErgForIndependentPolys}. 
Moreover, we prove that the linear independence of $p_1,...,p_\ell$ implies that every set of the form $R_\epsilon^{p_1,...,p_\ell}(A)$ has the A-IP$^*$ property (="almost" IP$^*$ property), which is  quite a bit stronger than syndeticity.\\
The following is one of the  new combinatorial results  obtained in this paper. Suppose that $p_1,...,p_\ell$ are linearly independent. For  any set $E\subseteq\Z^D$ with upper Banach density  $d^*(E)>0$, any non-zero ${\bf v_1,..., v_\ell}\in\Z^D$, and any $\epsilon>0$, the set 
\begin{equation*}\label{Abstract:DefnCombinatorialSet}
S_\epsilon^{p_1,...,p_\ell}(E):=\{\vec n\in\Z^d\,|\,d^*(E\cap (E-p_1(\vec n){\bf v_1})\cap \cdots\cap (E-p_\ell(\vec n){\bf v_\ell}))>(d^*(E))^{\ell+1}-\epsilon\}
\end{equation*}
is A-IP$^*$.  Furthermore, we prove that when $D>\ell>1$, this result is sharp in the sense that  the A-IP$^*$ property  cannot be upgraded to  IP$^*$. \\
The techniques developed in this paper lead to some additional applications. For example, we show that an amplified form of the IP-polynomial Szemer{\'e}di theorem conjectured in \cite{berMcCuIPPolySzemeredi} follows from the Density Polynomial Hales-Jewett conjecture.
%%%%%%%%%%%%%%%%%%%%%%%%%%%%%%%%%%%%
\end{abstract}

\maketitle

%\tableofcontents
\tableofcontents

%\tableofcontents
%%%%%%%%%%%%%%%%%%%%%%%%%%%%
\section{Introduction}
\subsection{New and old results dealing with single and multiple recurrence for measure preserving $\Z$-actions} 
We start our discussion with the formulation of the classical  Khintchine recurrence theorem.
\begin{thm}[Khintchine's recurrence theorem \cite{KhintchineRec1935}]\label{0.KhintchineRecThm}
For any invertible probability measure preserving system $(X,\mathcal A,\mu,T)$, any $A\in\mathcal A$, and any $\epsilon>0$, the set 
\begin{equation}
    R_\epsilon(A,T):=\{n\in\Z\,|\,\mu(A\cap T^{-n}A)\geq \mu^2(A)-\epsilon\}
\end{equation}
is syndetic, meaning that it has bounded gaps.     
\end{thm}
In  Khintchine's original paper \cite{KhintchineRec1935}, \cref{0.KhintchineRecThm}  was derived from the following variant of the ergodic theorem:
\begin{thm}\label{0.thm:khintchineLimitingFormula}
Let $(X,\mathcal A,\mu,T)$ be an invertible probability measure preserving system and let $A\in\mathcal A$. Then, 
\begin{equation}
    \lim_{N-M\rightarrow\infty}\frac{1}{N-M}\sum_{n=M+1}^N\mu(A\cap T^{-n}A)\geq \mu^2(A).
\end{equation}
\end{thm}
There is, however, an alternative approach to the proof of Khintchine's recurrence theorem which provides additional information about the sets of the form $R_\epsilon(A,T)$.
Let's call $E\subseteq \Z$  a $\Delta^*$ set if for every strictly monotone  sequence $(n_k)_{k\in\N}$ in $\Z$, the set 
$$\{n_j-n_i\,|\,j>i\}$$
has a non-trivial intersection with $E$. It is not hard to see that any $\Delta^*$ set is syndetic. On the other hand, not every syndetic set is $\Delta^*$ (consider for example the set $E=2\Z+1$). So, the following proposition forms a non-trivial amplification of \cref{0.KhintchineRecThm}. (This result is an easy corollary from the Enhanced  Principle P (EPP) formulated and proved on page 34 of \cite{MultifariousPoincare}.) 
\begin{thm}\label{0.Delta*Statement}
    Let $(X,\mathcal A,\mu,T)$ be an invertible probability measure preserving system, let $A\in\mathcal A$, and let $\epsilon>0$. The set 
    $$R_\epsilon(A,T)=\{n\in\Z\,|\,\mu(A\cap T^{-n}A)\geq \mu^2(A)-\epsilon\}$$
    is $\Delta^*$. 
    Moreover, there exists an $r:=
r(\mu(A),\epsilon)\in\N$ with the property that for any distinct $n_1,...,n_r\in\Z$, 
\begin{equation}\label{0.eq:Delta_r*property}
\{n_j-n_i\,|\,i<j,\,i,j\in\{1,...,r\}\}\cap R_\epsilon(A,T)\neq \emptyset.
\end{equation}
\end{thm}
It is worth mentioning that, by invoking Ramsey's theorem, one can show that $\Delta^*$ sets have the \textit{finite intersection property} \cite[p. 179]{Fbook}, meaning that if $E,F\subseteq \Z$ are $\Delta^*$ sets, then $E\cap F$ also is. It follows that if $(X,\mathcal A,\mu,T)$ and $(Y,\mathcal B,\nu,S)$ are invertible probability preserving systems, one has that for any $A\in\mathcal A$, $B\in\mathcal B$ with $\mu(A),\nu(B)>0$, and $\epsilon',\epsilon''>0$, the set 
\begin{equation}\label{0.eq:IntersectionOfSingleRec}
R_{\epsilon'}(A,T)\cap R_{\epsilon''}(B,S)
\end{equation}
is $\Delta^*$ (and, in particular, syndetic).\\

Next we formulate the Furstenberg-S{\'a}rk{\"o}zy theorem  \cite[Theorem 3.16]{Fbook}, \cite{sarkozy1978difference}, which can be viewed as a polynomial extension of the classical Poincar{\'e} recurrence theorem.
\begin{thm}[Furstenberg-S{\'a}rk{\"o}zy Theorem] \label{0.FSTheorem}
    Let $p\in\Z[x]$ be a non-constant polynomial with zero constant term and let $(X,\mathcal A,\mu,T)$ be an invertible probability preserving system. For any $A\in\mathcal A$ with $\mu(A)>0$, there are infinitely many $n\in\Z$ with $\mu(A\cap T^{-p(n)}A)>0$.
\end{thm}
A careful analysis of Furstenberg's proof of  \cref{0.FSTheorem} reveals that the following stronger result also holds.
\begin{thm}\label{0.FSThm.2}
     Let $p\in\Z[x]$ be a non-constant polynomial with zero constant term and let $(X,\mathcal A,\mu,T)$ be an invertible probability preserving system. For any $A\in\mathcal A$ with $\mu(A)>0$ and any $\epsilon>0$,  there is a natural number $b$ such that 
     \begin{equation}
         \lim_{N-M\rightarrow\infty}\frac{1}{N-M}\sum_{n=M+1}^N\mu(A\cap T^{-p(bn)}A)\geq \mu^2(A)-\epsilon.
     \end{equation}
\end{thm}
\begin{cor}\label{0.FSCorSyndetic}
    Let $p\in\Z[x]$ be a non-constant polynomial with zero constant term and let $(X,\mathcal A,\mu,T)$ be an invertible probability preserving system. For any $A\in\mathcal A$ and any $\epsilon>0$, the set
\begin{equation}\label{0.eq:PolySingleRecSET}
R_\epsilon^{p}(A,T):=\{n\in\Z\,|\,\mu(A\cap T^{-p(n)}A)>\mu^2(A)-\epsilon\}
\end{equation}
is syndetic.
\end{cor}
In light of \cref{0.FSThm.2} it is natural to inquire if, in analogy with \cref{0.Delta*Statement}, 
the sets of the form \eqref{0.eq:PolySingleRecSET}
are $\Delta^*$.
It turns out  that $R_\epsilon^p(A,T)$ is not necessarily a $\Delta^*$ set. For example, one can show that for any irrational number $\alpha\in\R$, the Lebesgue preserving transformation $T_\alpha:[0,1)\rightarrow [0,1)$ defined by $T_\alpha x=x+\alpha\mod 1$ has the property that if $A$ is an  interval of  length strictly smaller than $1/2$ and $p\in\Z[x]$ satisfies $\deg(p)>1$, then the set $R^p_\epsilon(A,T_\alpha)$ is not $\Delta^*$ (for the case $p(x)=x^2$, see \cite[p. 177]{Fbook}; see also \cite[Remark 5.4]{BerZel_JCTA_iteratedDifferences2021}). On the other hand, one can show  that sets of the form   $R_\epsilon^p(A,T)$ possess the IP$^*$ property which is  strictly stronger   than that of  syndetic (but strictly weaker than that of $\Delta^*$). The notion of IP$^*$ (in the generality of $\Z^d$) plays a prominent role in our paper. To introduce it,  we start by  recalling the definition  of IP set.\\
For any $d\in\N$, a set $E\subseteq \Z^d$ is called an IP set, if  there exists a  sequence $(\vec n_k)_{k\in\N}$ in $\Z^d$ with $\lim_{k\rightarrow\infty}|\vec n_k|=\infty$ such that  
$$E:=\{\vec n_{k_1}+\cdots+\vec n_{k_t}\,|\,k_1<\cdots<k_t,\,t\in\N\}.$$
The following variant of a classical theorem of Hindman \cite{HIPPartitionRegular}  is needed for the sequel.
\begin{thm}[Cf. Lemma 2.1 in \cite{BerHi1993AdditiveAndMultiplicative}]\label{0.thm:Hindman}
    Let $d\in\N$ and let $E,F\subseteq \Z^d$ be such that $E\cup F$ contains an IP set. Then, at least one of $E$ and $F$  contains an IP set as well. 
\end{thm}
We say that a set $E\subseteq \Z^d$ is IP$^*$ if it  has a non-trivial intersection with every IP set in $\Z^d$. The family   of IP$^*$ sets in $\Z^d$ has the following two important properties: (i) each of its elements is syndetic (a set $E\subseteq \Z^d$ is syndetic if there is a finite  $F\subseteq\Z^d$ with $\bigcup_{\vec a\in F}(E+\vec a)=\Z^d$) and (ii) as a consequence of \cref{0.thm:Hindman}, it has the finite intersection property. We remark in passing that  not every syndetic set is IP$^*$ (consider, for example, the set $2\Z+1$).
\begin{thm}[Cf. Corollary 2.1 in \cite{BFM}]\label{0.thm:BFM}
    Let $d\in\N$ and let $p\in\Z[x_1,...,x_d]$ be a non-constant polynomial with zero constant term. For any invertible probability preserving system $(X,\mathcal A,\mu,T)$, any $A\in\mathcal A$, and any $\epsilon>0$, the set 
    \begin{equation}
    R_\epsilon^p(A,T):=\{\vec n\in\Z^d\,|\,\mu(A\cap T^{-p(\vec n)}A)>\mu^2(
    A)-\epsilon\}
    \end{equation}
    is \rm{IP$^*$}. 
\end{thm}
The discussed above results demonstrate that statements pertaining to \textit{ single recurrence} (such as the Poincar{\'e} recurrence theorem or \cref{0.FSTheorem})  admit amplified forms dealing with  \textit{large intersections} (see Theorems \ref{0.KhintchineRecThm}, \ref{0.Delta*Statement}, \ref{0.FSThm.2}, and \ref{0.thm:BFM}).\\
It is of interest to find out when \textit{multiple polynomial recurrence} results have an amplified form which deals with large intersections. For example, given  any family of polynomials $p_1,...,p_\ell\in\Z[x_1,...,x_d]$ with zero constant term, the IP-polynomial Szemer{\'e}di theorem \cite{berMcCuIPPolySzemeredi} implies  that every set of the form
\begin{equation}\label{0.PositiveIntersection}
R_{>0}^{p_1,...,p_\ell}(A,T):=\{\vec n\in\Z^d\,|\,\mu(A\cap T^{-p_1(\vec n)}A\cap \cdots\cap T^{-p_\ell(\vec n)}A)>0\},
\end{equation}
is IP$^*$ (and, in particular, syndetic).\\
This makes it natural to inquire under which conditions are the sets of the form 
\begin{equation}\label{0.eq:LargeVectorIntersections}
R_\epsilon^{p_1,...,p_\ell}(A,T)=\{\vec n\in\Z^d\,|\,\mu(A\cap T^{-p_1(\vec n)}A\cap \cdots\cap T^{-p_\ell(\vec n)}A)>\mu^{\ell+1}(A)-\epsilon\}
\end{equation}
"large" (e.g. syndetic or, perhaps, IP$^*$).\\
Actually, the situation with the sets $R_\epsilon^{p_1,p_2,...,p_\ell}(A,T)$ is quite intricate.
For example,  it was shown in \cite{BHKNilSystems2005} that there are non-ergodic systems  $(X,\mathcal A,\mu,T)$ such that for some  $\epsilon>0$, and some $A\in\mathcal A$, the set $R_\epsilon^{n,2n}(A,T)$ equals $\{0\}$. On the other hand, it was also shown in \cite{BHKNilSystems2005}, that when  $(X,\mathcal A,\mu,T)$ is ergodic,  the sets of the  form $R_\epsilon^{n,2n}(A,T)$, $R_\epsilon^{n,2n,3n}(A,T)$  are syndetic but this is no longer the case for $R_\epsilon^{n,2n,3n,4n}(A,T)$. As a matter of fact, and in contrast with \eqref{0.eq:IntersectionOfSingleRec}, the proof of  \cite[Proposition 1.9] {DonosoLeSunMoreiraLowerBounds} reveals that there are ergodic $(X,\mathcal A,\mu,T)$, $(Y,\mathcal B,\nu,S)$, $A\in\mathcal A$, $B\in\mathcal B$, and $\epsilon>0$ for which  the sets $R_\epsilon^{n,2n}(A,T)$, $R_\epsilon^{n^2}(B,S)$ are syndetic but  
\begin{equation}\label{0.eq:NonSyndetic}
R_\epsilon^{n,2n}(A,T)\cap R_\epsilon^{n^2}(B,S)=\{0\}.
\end{equation}
(Formula \eqref{0.eq:NonSyndetic} should be juxtaposed with the fact  that, in light of the IP-polynomial Szemer{\'e}di theorem, $R_{>0}^{n,2n}(A,T)\cap R_{>0}^{n^2}(B,S)$ is not only syndetic but is actually  IP$^*$.) \\
We remark that,  as a consequence of \cite[Corollary 11.6]{ABBLargeIntersection2021}, for any $\ell>1$, any distinct and non-zero integers $a_1,...,a_\ell$, and any non-constant $p\in\Z[x]$ with zero constant term, there is an invertible probability preserving system $(X,\mathcal A,\mu,T)$, a set $A\in\mathcal A$ with $\mu(A)>0$, and an $\epsilon>0$ for which  $R_\epsilon^{a_1p,...,a_\ell p}(A,T)=\{0\}$ (see \cite{fra2008ThreePolynomials}, \cite{DonosoLeSunMoreiraLowerBounds} for related results).
On the other hand, as shown in  \cite{WMPet}, every set of the form \eqref{0.eq:LargeVectorIntersections} is syndetic under the additional assumption that $T$ is weak mixing. 
(Recall that  a measure preserving system $(X,\mathcal A,\mu,T)$ is called  weak mixing if 
$\lim_{N-M\rightarrow\infty}\frac{1}{N-M}\sum_{n=M+1}^N|\mu(A\cap T^{-n}B)-\mu(A)\mu(B)|=0$
for any $A,B\in\mathcal A$). \\
The above discussion indicates that, unless the system $(X,\mathcal A,\mu,T)$ has special properties, say,  is weak mixing (or, slightly more generally, a weakly mixing extension of a Kronecker system, see \cite[Theorem 6.1]{HKAComplexityOfNilsystems},
 \cite[Theorem 4.3]{AlmostIPBerLeib}), one should consider special families of polynomials in order for  sets of the form $R_\epsilon^{p_1,...,p_\ell}(A,T)$ to be large.
 For example, in light of the classical results of uniform distribution of polynomial tuples in multi-dimensional tori \cite{weyl1916Mod1} (see, for example, \cite[Theorem 1.6.4]{kuipers2012uniform}), it is natural to assume that the polynomials that we deal with are linearly independent over $\Q$. Under this independence assumption, various results about the largeness of the sets of the form $R_\epsilon^{p_1,...,p_\ell}(A,T)$ were established  
in \cite{AlmostIPBerLeib}, \cite{fra2008ThreePolynomials},  \cite{FraKra2006IndependentPolys}, \cite{FraKuJoinErgForIndependentPolys}.
The following theorem due to Frantzikinakis and Kra was the first positive result in this direction of inquiry. 
\begin{thm}[Theorem 1.3 in \cite{FraKra2006IndependentPolys}]\label{0.thm:FraKraTheorem1.3} Let $(X,\mathcal A,\mu,T)$ be an invertible probability measure preserving system and let $p_1,...,p_\ell\in\Z[x]$ be non-constant polynomials with zero constant term. Suppose that $p_1,...,p_\ell$ are $\Q$-linearly independent. Then, for any $A\in\mathcal A$ and every $\epsilon>0$, the set 
\begin{equation}\label{0.eq:DefnR_eSingleT}
R_\epsilon^{p_1,...,p_\ell}(A,T):=\{ n\in\Z\,|\,\mu(A\cap T^{-p_1(n)}A\cap \cdots \cap T^{-p_\ell(n)}A)>\mu^{\ell+1}(A)-\epsilon \}
\end{equation}
is syndetic. 
\end{thm}
In light of \cref{0.thm:BFM}, which is an IP$^*$-amplification of \cref{0.FSCorSyndetic}, it is reasonable to ask if \cref{0.thm:FraKraTheorem1.3} allows for a generalization 
which replaces syndeticity with the IP$^*$ property. 
\\
 The following theorem, which we prove in Section \ref{Sec8} and which  generalizes natural variants of Corollaries 1.14 and 1.15 in \cite{zelada2025coexistence}, provides a negative answer to this question hereby indicating that  any possible strengthening of \cref{0.thm:FraKraTheorem1.3} must deal with a notion of largeness which is either weaker than or different  from IP$^*$. The polynomials $p_1,...,p_\ell\in \Z[x_1,...,x_d]$ are called essentially distinct if for any different $i,j\in\{1,...,\ell\}$, $p_j-p_i$ is non-constant. 
\begin{thm}\label{0.thm:IPFailure}
Let $d,\ell\in\N$ be such that $\ell>1$ and let $p_1,...,p_\ell\in\Z[x_1,...,x_d]$ be essentially distinct, non-constant polynomials with zero constant term. There exists an  invertible probability preserving system $(X,\mathcal A,\mu,T)$, a set $A\in\mathcal A$ with $\mu(A)>0$, and an $\epsilon>0$ for which the set 
    \begin{equation}
        \{\vec n\in\Z^d\,|\,\mu(A\cap T^{-p_1(\vec n)}A\cap \cdots\cap T^{-p_\ell(\vec n)}A)\leq\mu^{\ell+1}(A)-\epsilon\}
    \end{equation}
    contains an IP set. So, for such an $A$, $T$, and $\epsilon$, the set 
    $$
R_{\epsilon}^{p_1,...,p_\ell}(A,T)=\{\vec n\in\Z^d\,|\,\mu(A\cap T^{-p_1(\vec n)}A\cap \cdots\cap T^{-p_\ell(\vec n)}A)>\mu^{\ell+1}(A)-\epsilon\},
    $$
 cannot be {\rm IP$^*$}. 
\end{thm}
\begin{question}
    We do not know if one can take the system $(X,\mathcal A,\mu,T)$ in \Cref{0.thm:IPFailure} to be ergodic. 
\end{question}
It turns out that a slight relaxation of the notion of IP$^*$, which was introduced in \cite{AlmostIPBerLeib}, is well-fitted  to obtain the sought after amplification of \cref{0.thm:FraKraTheorem1.3}. Given a set $F\subseteq \Z^d$, we define $d^*(F)$, the \textit{upper Banach density of $F$}, by 
$$d^*(F):=\limsup_{\substack{N_j-M_j\rightarrow\infty\\ j\in\{1,...,d\}}}\frac{|F\cap\prod_{j=1}^d\{M_j+1,...,N_j\}|}{\prod_{j=1}^d(N_j-M_j)}.$$
A set $E\subseteq \Z^d$ is called an \textit{almost} IP$^*$ set (denoted A-IP$^*$) if there exists a   set $F\subseteq \Z^d$ with $d^*(F)=0$ for which $E\cup F$ is  IP$^*$. In other words, up to a potential modification on a set with upper Banach density zero, $E$ is an  IP$^*$ set. One can show that the notion of A-IP$^*$ lies strictly between IP$^*$ and syndetic \cite[p. 505]{AlmostIPBerLeib} (see, for example, item  (3)  in the explanations to Figure \ref{Fig:VenDiagram} below). Also, A-IP$^*$ sets have the finite intersection property \cite[p. 502]{AlmostIPBerLeib}. Indeed,  let $C,D\subseteq \Z^d$ be A-IP$^*$.  Then we can find $E,F\subseteq \Z^d$ with $d^*(E)=0=d^*(F)$  such that $C\cup E$ and $D\cup F$ are IP$^*$. It follows that $(C\cup E)\cap (D\cup F)$ is also an  IP$^*$. So, by noting that 
$$(C\cup E)\cap (D\cup F)\subseteq (C\cap D)\cup (E\cup F)$$
and utilizing the fact that supersets of IP$^*$ sets are IP$^*$, we conclude that $(C\cap D)\cup(E\cup F)$ is IP$^*$. Thus, since $d^*(E\cup F)=0$, $C\cap D$ is A-IP$^*$.\\
As a matter of fact, a  stronger notion than that of A-IP$^*$ that can be viewed as a  "finitary" version thereof  also has the finite intersection property and is good for our purposes.\\
 Following Furstenberg and Katznelson \cite{FKIPSzemerediLong}, given any $r\in\N$ and any  distinct $\vec n_1,...,\vec n_r\in\Z^d$, we call the set 
 $$\{\vec n_{j_1}+\cdots+\vec n_{j_t}\,|\,1\leq j_1<\cdots<j_t\leq r\}$$
 an IP$_r$ set.  We say that  a set $E\subseteq \Z^d$ is an IP$_0^*$ set if it is an IP$_r^*$ set for some $r\in\N$ (i.e. for every IP$_r$ set $D\subseteq \Z^d$, $E\cap D\neq\emptyset$). A set $E\subseteq \Z^d$ is called an A-IP$_0^*$ set (read "almost" IP$_0^*$ set) if there exists  a set $F\subseteq \Z^d$ with $d^*(F)=0$
such that $E\cup F$ is IP$_0^*$. It is worth pointing out that IP$_r$ sets are finitary versions of IP sets and that the classes of IP$_0^*$ sets  and A-IP$_0^*$ sets   possess the finite intersection property (in the case of IP$_0^*$ sets, such property can be obtained with the help of \cite[Proposition 2.5]{BDonaldRobertsonIP_r}. The finite intersection property of A-IP$_0^*$ sets, which was first observed in \cite[p. 502] {AlmostIPBerLeib}, can be deduced from that of IP$_0^*$ sets in a way similar to that in which the corresponding property of A-IP$^*$ sets is deduced from that of IP$^*$). We remark that the notion of A-IP$_0^*$ is strictly stronger than that of A-IP$^*$ (see \cite[p. 505]{AlmostIPBerLeib} or item (4) in the explanations to Figure \ref{Fig:VenDiagram} below).\\ 
The following result which is a special case of a more general result to be proved in Section \ref{Sec4} is a strengthening of \cref{0.thm:FraKraTheorem1.3} along the lines discussed above. 
\begin{thm}[Cf. Theorem 4.2 in \cite{AlmostIPBerLeib}]\label{0.thm:BerLeibTheorem4.2} Let $(X,\mathcal A,\mu,T)$ be an invertible  probability measure preserving system and let $p_1,...,p_\ell\in\Z[x_1,...,x_d]$ be non-constant polynomials with zero constant term. Suppose that $p_1,...,p_\ell$ are $\Q$-linearly independent. Then, for any $A\in\mathcal A$ and every $\epsilon>0$, the set 
\begin{equation}\label{0.eq:LArgeBerLeibmanIntersection}
R_\epsilon^{p_1,...,p_\ell}(A,T)=\{\vec n\in\Z^d\,|\,\mu(A\cap T^{-p_1(\vec n)}A\cap \cdots \cap T^{-p_\ell(\vec n)}A)>\mu^{\ell+1}(A)-\epsilon \}
\end{equation}
is \rm{A-IP$_0^*$}.
\end{thm}
A few words regarding the case  $\ell=1$ of \cref{0.thm:BerLeibTheorem4.2} are now in place. First, let us note that when $p_1\in\Z[x_1,...,x_d]$ is linear, the conclusions of both \cref{0.thm:BFM} and \cref{0.thm:BerLeibTheorem4.2} follow from a natural analogue of \cref{0.Delta*Statement} which states that for any $\vec a=(a_1,...,a_d)\in \Z^d$, any invertible probability preserving system $(X,\mathcal A,\mu,T)$, any $A\in\mathcal A$, and any $\epsilon>0$, the set 
\begin{equation}\label{0.eq:LinearCaseT}
R^{\vec a\cdot \vec n}_\epsilon(A,T):=\{(n_1,...,n_d)\in \Z^d\,|\,\mu(A\cap T^{-(a_1n_1+\cdots+a_dn_d)}A)>\mu^2(A)-\epsilon\}
\end{equation}
is IP$_0^*$.\\
It is somewhat surprising that for $\ell=1$,
a similar common refinement of  Theorems \ref{0.thm:BFM} and \cref{0.thm:BerLeibTheorem4.2}  does not hold when $p_1\in\Z[x_1,...,x_d]$ is a non-linear polynomial. Indeed, we have the following generalization  of \cite[Theorem A]{ZelIP0Khintchine2023} to polynomials $p\in \Z[x_1,...,x_d]$ (see Section \ref{Sec7}).
\begin{thm}[Cf. Theorem A in \cite{ZelIP0Khintchine2023}]\label{0.thm:IPoFailure}
    Let $d\in\N$ and let $p\in\Z[x_1,...,x_d]$ be a non-constant, non-linear polynomial with zero constant term.  
    Then, there exist an invertible, weakly mixing probability preserving system $(X,\mathcal A,\mu,T)$, a set $A\in\mathcal A$ with $\mu(A)>0$, and an $\epsilon>0$, for which the set $R_\epsilon^{p}(A,T)$ is not {\rm IP$_0^*$}.
\end{thm}
\subsection{ Relations between the various notions of largeness}
 The following  diagram provides a visual representation of relations between various notions of largeness which were introduced above.
The purpose of various symbols, such  as  $R_\epsilon^{\vec a\cdot \vec n}(A,T)$ and  $R_\epsilon^{p_1,...,p_{\ell+1}}(A,T)$,  which appear in the diagram below,  is to indicate the non-emptiness of the corresponding regions (see remarks below  Figure \ref{Fig:VenDiagram}).
The sets $R_U(x_0)$ and $R_\epsilon^{(\vec a\cdot\vec n),2(\vec a\cdot\vec n)}(A,T)$ are defined as follows:
\begin{enumerate}
\item [-] Given a topological dynamical system $(X,T)$, a point $x_0\in X$, and an open neighborhood $U$ of $x_0$, we let $R_U(x_0):=\{\vec n\in \Z^d\,|\, T^{\vec n}x_0\in U\}$.
\item [-] For any probability preserving system $(X,\mathcal A,\mu,T)$, any $\epsilon>0$, any $A\in\mathcal A$, and any $\vec a=(a_1,...,a_d)\in\Z^d$, we let 
$$
R_\epsilon^{(\vec a\cdot\vec n),2(\vec a\cdot\vec n)}(A,T)=\{(n_1,...,n_d)\in\Z^d\,|\,\mu(A\cap T^{-\sum_{j=1}^d a_jn_j}A\cap T^{-2\sum_{j=1}^da_jn_j}A)>\mu^3(A)-\epsilon\}.$$
\end{enumerate}
\begin{figure}[H]
\centering
\resizebox{!}{6.5cm} {\includegraphics{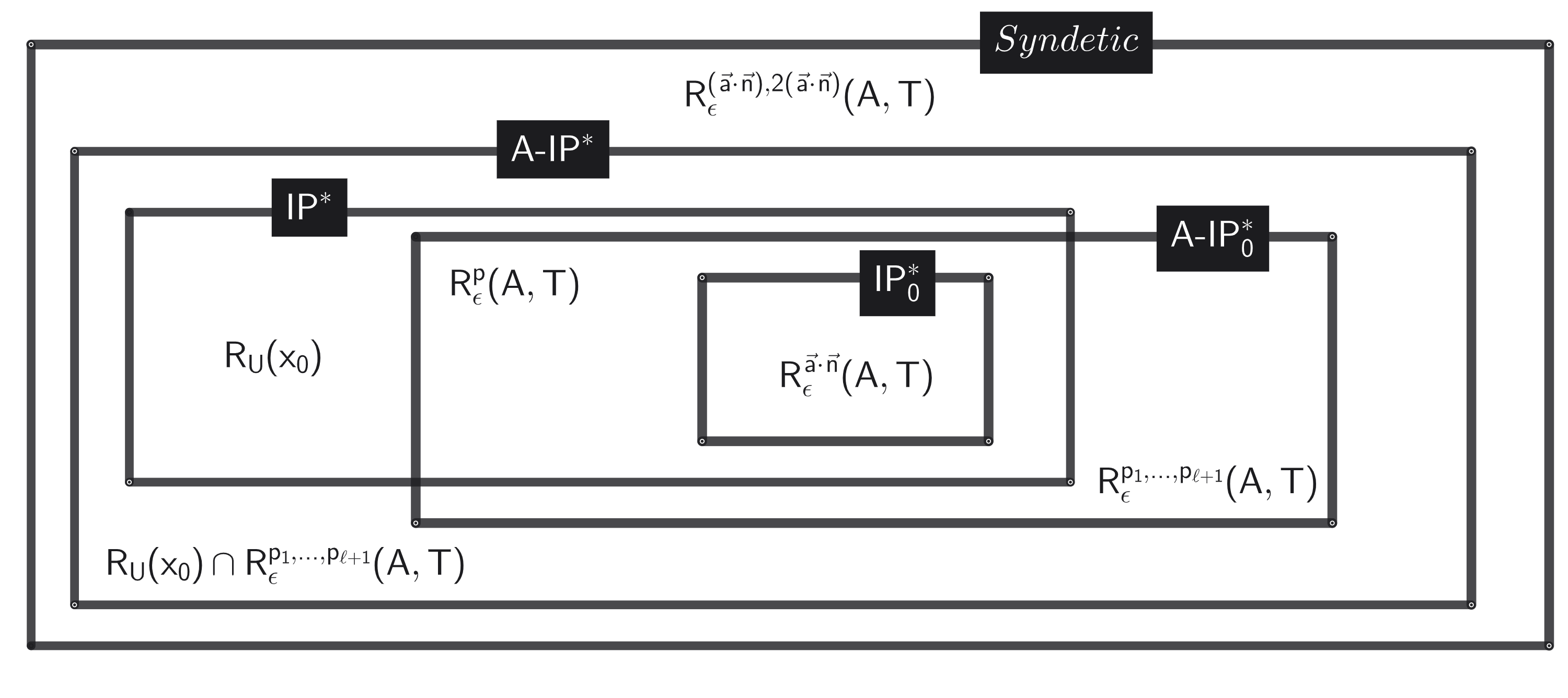}}
\caption{\small 
}
\label{Fig:VenDiagram}
\end{figure}
Here are the explanatory remarks concerning the families of sets appearing  in Figure \ref{Fig:VenDiagram}:  
\begin{enumerate}[(1)]
\item The fact that  every  set of the form $R_\epsilon^{\vec a\cdot \vec n}(A,T)$ is IP$_0^*$ (see the rectangle at the center of the diagram above) follows from \eqref{0.eq:LinearCaseT}. 
\item Let $p\in\Z[x_1,...,x_d]$ be a non-constant, non-linear polynomial with zero constant term. The fact that there are sets of the form $R_\epsilon^p(A,T)$  in the region specified in Figure 1, follows from Theorems  \ref{0.thm:BFM},  \ref{0.thm:BerLeibTheorem4.2}, and \ref{0.thm:IPoFailure}.
%%%%%%%%%%%%%%%%%%%%%%%%%
\item The fact that there are sets of the form $R_\epsilon^{p_1,...,p_{\ell+1}}(A,T)$, $\ell\in\N$,  in the region specified in Figure 1, follows from Theorems  \ref{0.thm:IPFailure} and \ref{0.thm:BerLeibTheorem4.2}.
\item The fact that there is a set of the form $R_U(x_0)$  in the region specified in Figure 1, follows from \cite[p. 505]{AlmostIPBerLeib}.  There, the authors show that there exists a distal dynamical system $(X,T)$, an $x_0\in X$, and an open neighborhood $U$ of $x_0$ for which the set 
$$R_U(x_0):=\{\vec n\in \Z^d\,|\,T^{\vec n}x_0\in U\}$$
is IP$^*$ but not A-IP$_0^*$. 
%%%%%%%%%%%%%%%%%%%%%%%%%%%%%%%%%%
\item Let $R_\epsilon^{p_1,...,p_{\ell+1}}(A,T)$ and $R_U(x_0)$ be as in (3) and (4). Then $R:=R_\epsilon^{p_1,...,p_{\ell+1}}(A,T)\cap R_U(x_0)$ is A-IP$^*$. However, since $R_\epsilon^{p_1,...,p_{\ell+1}}(A,T)$ is not IP$^*$, $R$ is not IP$^*$ and since  $R_U(x_0)$ is not A-IP$_0^*$, $R$ neither is A-IP$_0^*$.
%$%%%%%%%%%%%%%%%%%%%%%%%%%%%%%%%%%%%%%%%
\item Let $d\in\N$. We claim that for every non-zero $\vec a=(a_1,...,a_d)\in \Z^d$ there is a set of the form 
$R^{(\vec a\cdot\vec n),2(\vec a\cdot\vec n)}_\epsilon(A,T)$  in the region specified in Figure 1. Let $\alpha$ be irrational and consider the classical skew product transformation on the 2-torus $\mathbb T^2$ given by $T(x,y)=(x+\alpha,y+x)$. Denote the Lebesgue measure on $\mathbb T^2$ by $\mu$ and let  $\mathcal A$ denote the $\sigma$-algebra of Borel sets in $\mathbb T^2$. 
As noted in \cite[Subsection 8.1]{BHFebruary2003}, there is  an $A\in \mathcal A$ and an $\epsilon>0$ for which the set 
$\Z\setminus R_\epsilon ^{n,2n}(A,T)$ is IP$^*$  (that such an $A\in\mathcal A$ and $\epsilon>0$ exist can also be verified by  a slight modification of the  proof of Theorem  B.1  in  \cite{Shalomkra2025}). Let $\varphi:\Z^d\rightarrow \Z$ be given by $\varphi(\vec n)=\vec a\cdot \vec n$  and note that $\varphi^{-1}(R_\epsilon^{n,2n}(A,T))$ is a set of the form $R_\epsilon^{(\vec a\cdot\vec n),2(\vec a\cdot\vec n)}(A,T)$. It is now routine to check that since $R_\epsilon^{n,2n}(A,T)$ is syndetic, $\varphi^{-1}(R_\epsilon^{n,2n}(A,T))$ also is.\\
To prove that $\varphi^{-1}(R_\epsilon^{n,2n}(A,T))$  is not A-IP$^*$, we will show that 
$$E:=\Big(\Z^d\setminus\varphi^{-1}(R_\epsilon^{n,2n}(A,T))\Big)\cup \varphi^{-1}(\{0\})$$ 
is IP$^*$ (the claim then follows by recalling that A-IP$^*$ sets have the finite intersection property and that, since $\varphi^{-1}(\{0\})$ is contained in a proper subspace of $\R^d$,  $d^*(\varphi^{-1}(\{0\}))=0$). Suppose, for the sake of contradiction, that there is an IP set $F$ with $F\subseteq (\Z^d\setminus E)$. Since $\Z^d\setminus E=\varphi^{-1}(R_\epsilon^{n,2n}(A,T)\setminus\{0\})$, our assumption implies that $\varphi(F)$ contains an IP set (in $\Z$) and $\varphi(F)\subseteq R_\epsilon^{n,2n}(A,T)$ contrary to the fact that $\Z\setminus R_\epsilon^{n,2n}(A,T)$ is IP$^*$. 
\end{enumerate}
\subsection{Sets of large returns for commuting polynomial actions ---the  main result}
It is desirable to obtain generalizations of the above results to measure preserving systems involving commuting measure preserving transformations. For starters, one would like to know if \cref{0.thm:FraKraTheorem1.3} extends to commuting operators. 
The following  recent result due to N. Frantzikinakis and B. Kuca provides an affirmative answer to this question. It  follows from \cite[Corollary 2.11]{FraKuJoinErgForIndependentPolys}, the Remark on page 633 of \cite{FraKuJoinErgForIndependentPolys}, and the comments in Subsection 2.7 of \cite{FraKuJoinErgForIndependentPolys}.
The polynomials $p_1,...,p_\ell\in\Z[x_1,...,x_d]$ are called \textit{jointly intersective} if for every $m\in\N$ there is an $\vec n\in \Z^d$ such that $p_j(\vec n)\equiv 0\mod m$ for each $j\in\{1,...,\ell\}$.
 %%%%%%%
\begin{thm}\label{0.thm:FraKu} Let $(X,\mathcal A,\mu)$ be a probability space, let $T_1,...,T_\ell$ be invertible and commuting $\mu$-preserving transformations, and   let $p_1,...,p_\ell\in\Z[x_1,...,x_d]$ be non-constant polynomials. Suppose that $p_1,...,p_\ell$ are $\Q$-linearly independent and jointly intersective. Then, for any $A\in\mathcal A$ and every $\epsilon>0$, the set 
\begin{equation}%\label{0.eq:DefnR_eSingleT'}
R_\epsilon^{p_1,...,p_\ell}(A):=\{\vec n\in\Z^d\,|\,\mu(A\cap T_1^{-p_1(\vec n)}A\cap \cdots \cap T
_\ell^{-p_\ell(\vec n)}A)>\mu^{\ell+1}(A)-\epsilon \}
\end{equation}
is syndetic.
\end{thm}
The following theorem strengthens \cref{0.thm:FraKu} in a manner analogous to how \cref{0.thm:BerLeibTheorem4.2} refines \cref{0.thm:FraKraTheorem1.3}. Furthermore, it answers the following question posed by N. Frantzikinakis and B. Kuca (see second bullet of the Remark on page 633 of \cite{FraKuJoinErgForIndependentPolys}). 
\begin{question}\label{0.Q:FraKu}
    Fix $d,\ell\in\N$. What conditions on a family of polynomials $p_1,...,p_\ell\in \Z[x_1,...,x_d]$ (if any) is equivalent to the fact that every set of the form 
    \begin{equation}
    R_\epsilon^{p_1,...,p_\ell}(A)=\{\vec n\in \Z^d\,|\,\mu(A\cap T_1^{-p_1(\vec n)}A\cap\cdots\cap T_\ell^{-p_\ell(\vec n)}A)>\mu^{\ell+1}(A)-\epsilon\}
    \end{equation}
    is syndetic? Notice that here we allow $p_j(\vec 0)\neq 0$ for $j\in\{1,...,\ell\}$.  
\end{question}
We remark that the special case
of \Cref{0.Q:FraKu} dealing with $\ell=1$ can be obtained with the help of \cref{0.FSThm.2} and routine considerations about the connection between intersective polynomials and sets of measurable recurrence. The case $\ell=2$ was answered in \cite{FraKuJoinErgForIndependentPolys}. 
\begin{thm}\label{0.MainResultLebesgueSpaces} 
    Let $d,\ell\in\N$ and let $p_1,...,p_\ell\in\Z[x_1,...,x_d]$ be non-constant polynomials. The following statements are equivalent. 
    \begin{enumerate}[(i)]
    \item The polynomials $p_1,...,p_\ell$ are $\Q$-linearly independent and jointly intersective (i.e. for every $m\in\N$ there is an $\vec n\in \Z^d$ such that $p_j(\vec n)\equiv0\mod m$ for each $j\in\{1,...,\ell\}$). 
    \item For any probability space $(X,\mathcal A,\mu)$, any invertible and commuting $\mu$-preserving transformations $T_1,...,T_\ell$, any $A\in\mathcal A$, and any $\epsilon>0$, there is a $\Z^d$-translate of the set
    \begin{equation}
        R_\epsilon^{p_1,...,p_\ell}(A)=\{\vec n\in\Z^d\,|\,\mu(A\cap T_1^{-p_1(\vec n)}A\cap \cdots\cap T_\ell^{-p_\ell(\vec n)}A)>\mu^{\ell+1}(A)-\epsilon\}
    \end{equation}
    which is {\rm A-IP$_0^*$} (i.e. $\exists {\vec n_0}\in\Z^d$ with $R_\epsilon^{p_1,...,p_\ell}(A)-{\vec n_0}$ an {\rm A-IP$_0^*$} set). Furthermore, if $p_j(\vec  0)=0$ for each $j\in\{1,...,\ell\}$, the set $R_\epsilon^{p_1,...,p_\ell}(A)$ is itself an {\rm A-IP$_0^*$} set.
        \item For any probability space $(X,\mathcal A,\mu)$,
    any invertible and commuting $\mu$-preserving transformations $T_1,...,T_\ell$, any $A\in\mathcal A$, any $\epsilon>0$, and any $j\in\{1,...,\ell\}$, there exists an $\vec n\in R_\epsilon^{p_1,...,p_\ell}(A)$ for which $p_j(\vec n)\neq 0$.  
\end{enumerate}
Furthermore, in the special case that $d=1$, items (i)-(iii) are also equivalent to the following statement:
\begin{enumerate}[(iv)]
    \item  For any invertible and commuting $\mu$-preserving transformations $T_1,...,T_\ell$, any $A\in\mathcal A$, and any $\epsilon>0$, the set
    \begin{equation}
        R_\epsilon^{p_1,...,p_\ell}(A)=\{n\in\Z\,|\,\mu(A\cap T_1^{-p_1(n)}A\cap \cdots\cap T_\ell^{-p_\ell(n)}A)>\mu^{\ell+1}(A)-\epsilon\}
    \end{equation}
    contains more than $\max_{1\leq j\leq \ell}\deg(p_j)$ elements.
    \end{enumerate}
\end{thm}
\begin{rem}
    We remark that items (i)-(iii) in \cref{0.MainResultLebesgueSpaces} are also equivalent to the following ostensibly  stronger variant of (ii) dealing with the A-VIP$_0^*$ property which was introduced in \cite{AlmostIPBerLeib} and which implies the A-IP$_0^*$ property (see Subsection \ref{Sec3.1} for the definition and basic properties):
       \begin{enumerate}
        \item [(ii$^{\prime}$)] For any probability space $(X,\mathcal A,\mu)$, any invertible and commuting $\mu$-preserving transformations $T_1,...,T_\ell$, any $A\in\mathcal A$, and any $\epsilon>0$, there is a $\Z^d$-translate of the set
    \begin{equation}
        R_\epsilon^{p_1,...,p_\ell}(A)=\{\vec n\in\Z^d\,|\,\mu(A\cap T_1^{-p_1(\vec n)}A\cap \cdots\cap T_\ell^{-p_\ell(\vec n)}A)>\mu^{\ell+1}(A)-\epsilon\}
    \end{equation}
    which is {\rm A-VIP$_0^*$}. Furthermore, if $p_j(\vec  0)=0$ for each $j\in\{1,...,\ell\}$, the set $R_\epsilon^{p_1,...,p_\ell}(A)$ is itself an {\rm A-VIP$_0^*$} set.
    \end{enumerate} 
As we explain in Section \ref{Sec4}, where we prove (i)$\implies$(ii$^\prime$),  item (ii$^\prime$) above will allow us to obtain   a non-trivial amplification of  \cite[Theorem 4.2]{AlmostIPBerLeib} (See \cref{RemarkAboutTheorem4.2} for further discussion). 
\end{rem}
The following result, which is a corollary of the proof of \cref{0.MainResultLebesgueSpaces}, demonstrates the sharpness of the implication  (iii)$\implies$(i) in \cref{0.MainResultLebesgueSpaces}. 
\begin{cor}\label{0.thm:Sharp(iii)->(i)}
Let $d,\ell\in\N$ with $\ell>1$ and let $p_1,...,p_\ell\in\Z[x_1,...,x_d]$ be non-constant, linearly dependent polynomials. Consider a non-empty $\alpha\subseteq \{1,...,\ell\}$ with the property that for some non-zero integers $a_j$, $j\in\alpha$,
$$\sum_{j\in\alpha}a_jp_j=0.$$ 
Then, there exist a probability space $(X,\mathcal A,\mu)$ and invertible and commuting $\mu$-preserving transformations $T_1,...,T_\ell$ with the property that for every $r\in\N$ there is an  $A_r\in\mathcal A$ with $\mu(A_r)>0$ and  the property that 
\begin{equation}\label{0.eq:EqualityWithLongerExpression}
    \{\vec n\in\Z^d\,|\,\mu(A_r\cap T_1^{-p_1(\vec n)}A_r\cap \cdots\cap T_\ell^{-p_\ell (\vec n)}A_r)>\frac{\mu^r(A_r)}{2}\}
    =\bigcap_{j\in \alpha} \{\vec n\in\Z^d\,|\,p_j(\vec n)=0\}.
\end{equation}
In other words, the $\vec n\in\Z^d$ for which 
\begin{equation}
\mu(A_r\cap T_1^{-p_1(\vec n)}A_r\cap \cdots\cap T_\ell^{-p_\ell (\vec n)}A_r)>\frac{\mu^r(A_r)}{2}
\end{equation}
are exactly those $\vec n$ for which $p_j(\vec n)=0$ for each $j\in\alpha$.
\end{cor}
We will now formulate  a combinatorial corollary of \cref{0.MainResultLebesgueSpaces} which can be  obtained by applying Furstenberg's correspondence principle  as presented  in \cite[Theorem 6.4.17]{bergelson2000DiophantineSurvey}. This combinatorial  result can be interpreted as a strong form of a special case of the polynomial Szemer{\'e}di theorem \cite{bergelsonLeibmaPolyVderWardenAndSz} dealing with families of jointly intersective, linearly independent polynomials.  
A sequence $(\Phi_N)_{N\in\N}$ of non-empty, finite subsets of $\Z^D$ is called a F{\o}lner sequence if for every $\vec n\in\Z^D$ one has 
$$
\limsup_{N\rightarrow\infty} \frac{|\Phi_N\triangle(\Phi_N-\vec n)|}{|\Phi_N|}=0.
$$
For any $E\subseteq \Z^D$, we write
$$
\overline d_{(\Phi_N)}(E)=\limsup_{N\rightarrow\infty} \frac{|E\cap \Phi_N|}{|\Phi_N|}.
$$
\begin{cor}\label{0.CombinatorialApplication}
    Let $d,D,\ell\in\N$ and let $p_1,...,p_\ell\in\Z[x_1,...,x_d]$ be non-constant polynomials. Suppose that $p_1,...,p_\ell$ are jointly intersective and $\Q$-linearly independent. Then, for any $E\subseteq \Z^D$, any F{\o}lner sequence $(\Phi_N)_{N\in\N}$ in $\Z^D$, any non-zero ${\bf v_1},...,{\bf v_\ell}\in\Z^D$, and any $\epsilon>0$, the set 
    $$
S_\epsilon^{p_1,...,p_\ell}(E):=\{\vec n\in\Z^d\,|\,\overline d_{(\Phi_N)}(E\cap (E-p_1(\vec n){\bf v_1})\cap \cdots\cap (E-p_\ell(\vec n){\bf v_\ell}))>(\overline d_{(\Phi_N)}(E))^{\ell+1}-\epsilon\}
    $$
is a translation of an A-\rm{VIP$_0^*$} set. Furthermore, if $p_j(\vec 0)=0$ for each $j\in\{1,...,\ell\}$, then $S_\epsilon^{p_1,...,p_\ell}(E)$ is an A-\rm{VIP$_0^*$} set. 
\end{cor}
We remark that 
\cref{0.CombinatorialApplication} provides an amplification of \cite[Corollary 2.12]{FraKuJoinErgForIndependentPolys} which established the syndeticity of the sets of the form $S_\epsilon^{p_1,...,p_\ell}(E)$.\\
We conclude this introduction with the following result  complementing \cref{0.CombinatorialApplication}. We say that a  set $E\subseteq \Z^D$ has uniform density $\delta\geq 0$  and write $d_{{\rm U}}(E)=\delta$ if for every F{\o}lner sequence $(\Phi_N)_{N\in\N}$ in $\Z^D$ one has
 $$\delta=d_{(\Phi_N)}(E):=\lim_{N\rightarrow\infty}\frac{|E\cap \Phi_N|}{|\Phi_N|}.$$
\begin{cor}\label{0.Cor:CombinatorialFailureOfIP}
    Let $d,D,\ell\in\N$  be such that $D>\ell>1$. Let $p_1,...,p_\ell\in\Z[x_1,...,x_d]$ be  non-constant polynomials with zero constant term and 
    let ${\bf v_1},...,{\bf v_\ell}\in\Z^D$ be non-zero vectors for which  the  maps $\vec n\mapsto p_j(\vec n){\bf v_j}$, $j\in\{1,...,\ell\}$, are pairwise distinct.
    There exist $\epsilon>0$ and a set $E\subseteq \Z^D$ with $d_{{\rm U}}(E)>0$ such that the set 
    \begin{equation}\label{0.eq:InverseFurstenberg}
\{\vec n\in\Z^d\,|\,d_{{\rm U}}(E\cap (E-p_1(\vec n){\bf v_1})\cap\cdots\cap (E-p_\ell(\vec n){\bf v_\ell}))>(d_{{\rm U}}(E))^{\ell+1}-\epsilon\}
    \end{equation}
    is not {\rm IP$^*$}.  
\end{cor}
We now state a special case of Question 2 on page 78 in \cite{berMcCuIPPolySzemeredi}, which served as one of the primary motivations for our paper.
\begin{question}\label{0.question:IP*2}
    Let $p_1,...,p_\ell\in \Z[x]$ be non-constant polynomials with zero constant term, let $(X,\mathcal A,\mu)$ be an atomless probability space, let  $T_1,...,T_\ell$ be invertible and commuting $\mu$-preserving transformations,  and let $A\in\mathcal A$ be such that $\mu(A)>0$. Is there a constant $c>0$ for which the set 
    \begin{equation}\label{0.eq:StrongRec}  
        \{ n\in\Z\,|\,\mu(A\cap T_1^{-p_1(n)}A\cap \cdots\cap T_\ell^{-p_\ell(n)}A) > c\}
    \end{equation}
    is IP$^*$?
\end{question}
While we cannot resolve \Cref{0.question:IP*2}, it will be demonstrated in Appendix \ref{Sec:A} that an affirmative answer to \Cref{0.question:IP*2}
follows from the Density polynomial Hales-Jewett conjecture (DPHJ). (See Appendix \ref{A.sec.DPHJConjecture} for the statement of DPHJ and \cref{8.thm:ConditionalSzemeredi} for its consequence providing a positive answer to \Cref{0.question:IP*2}).\\

The structure of this paper is as follows: 
In Section \ref{Sec:2}, we review the necessary background in ergodic theory. 
In Section \ref{Sec:3}, we define VIP$_0^*$ and A-VIP$_0^*$ sets and prove \cref{3.thm:DiophantineRec}, a Diophantine result involving jointly intersective polynomials, which is needed for the proof of \cref{0.MainResultLebesgueSpaces} and is of independent interest.
In Section \ref{Sec4}, we prove the implication (i)$\implies$(ii$^\prime$) in \cref{0.MainResultLebesgueSpaces}.   In Section \ref{Sec:5}, we prove a special case of the implication (iii)$\implies$(i) in \cref{0.MainResultLebesgueSpaces} and derive \cref{0.thm:Sharp(iii)->(i)}.
In Section \ref{Sec:6}, we complete the proof of \cref{0.MainResultLebesgueSpaces}. 
In Section \ref{Sec7}, we prove \cref{0.thm:IPoFailure}. 
In Section \ref{Sec8}, we prove \cref{0.thm:IPFailure}. 
In Section \ref{Sec:9}, we establish an approximate "uniform" version of the inverse Furstenberg correspondence principle and use it to prove \cref{0.Cor:CombinatorialFailureOfIP}.
In Appendix \ref{Sec:A}, we derive some non-trivial ergodic-theoretical consequences of the 
 Density Polynomial Hales-Jewett conjecture.\\

\textit{Acknowledgments.} Rigoberto Zelada is supported by EPSRC through Joel Moreira's Frontier Research Guarantee grant, ref. EP/Y014030/1. 
%%%%%%%%%%%%%%%%%%%%%%%%%%%%%%%%%%%%%%%%%%%%%%%%%%%%%%%%%%%%%%%%%%%%%%%%%
%%%%%%%%%%%%%%%%%%%%%%%%%%%%%%%%%%%%%%%%%%%%%%%%%%%%%%%%%%
%%%%%%%%%%%%%%%%%%%%%%%%%%%%%%%%%%%%%%%%%%%%%%%%%%%%%%%%%%%%%%%%
\section{Background on ergodic theory}\label{Sec:2}
In this section we formulate the basic definitions and results which will be needed in the subsequent sections.\\
Let $(X,\mathcal A,\mu)$ be a probability space and let $T:X\rightarrow X$ be an invertible $\mu$-preserving transformation. 
As is customary, we assume that the $\sigma$-algebra $\mathcal A$ is separable. 
An invertible probability measure preserving system ${\bf Y}=(Y,\mathcal B,\nu,S)$ is a factor of $
{\bf X}=(X,\mathcal A,\mu,T)$ if there is a surjective measure preserving map $\pi:X\rightarrow Y$ with the property that $S\circ \pi=\pi\circ T$. An alternative way of introducing a factor of $(X,\mathcal A,\mu,T)$
is to indicate the corresponding $T$-invariant sub-$\sigma$-algebra of $\mathcal A$. The Kronecker factor of $(X,\mathcal A,\mu,T)$ is the factor whose associated $L^2$-space is generated by the bounded, $\mathcal A$-measurable eigenfunctions of $T$. It is convenient to identify the Kronecker factor with the quadruple $(X,\mathcal K_T,\mu,T)$, where $\mathcal K_T$ is the smallest $\sigma$-algebra with respect to which every  bounded, $\mathcal A$-measurable eigenfunction of $T$ is measurable.\\
The following result  is an immediate consequence of the comments at the end of \cite[Subsection 2.7]{FraKuJoinErgForIndependentPolys}.
\begin{thm}[Cf. Theorem 2.10 and equation (8) in
\cite{FraKuJoinErgForIndependentPolys}]
\label{FraKuTheorem}
    Let $(X,\mathcal A,\mu)$ be a probability space. For any $d,\ell\in\N$, any commuting invertible $\mu$-preserving transformations  $T_1,...,T_\ell$, any  $\Q$-linearly independent polynomials $p_1,...,p_\ell\in \Z[x_1,...,x_d]$ with zero constant term, and any $f_1,...,f_\ell\in L^\infty(\mu)$, one has 
\begin{multline}
    \lim_{N\rightarrow\infty}\left\|\frac{1}{|\Phi_N|}\sum_{\vec n\in\Phi_N}T_1^{p_1(\vec n)}f_1\cdots T_\ell^{p_\ell(\vec n)}f_\ell\right.\\
    \left.-\frac{1}{|\Phi_N|}\sum_{\vec n\in\Phi_N}T_1^{p_1(\vec n)}\mathbb E(f_1\,|\,\mathcal K_{T_1})\cdots T_\ell^{p_\ell(\vec n)}\mathbb E(f_\ell\,|\,\mathcal K_{T_\ell})\right\|_{L^2}=0
\end{multline}
for any F{\o}lner sequence $(\Phi_N)_{N\in\N}$ in $\Z^d$.
\end{thm}
The following corollary to  \cref{FraKuTheorem} plays an essential role in the proof of (i)$\implies$(ii$^\prime$) of \cref{0.MainResultLebesgueSpaces} (one of the main results of this paper).
\begin{cor}[Cf. Lemma 4.3 in \cite{FraKra2006IndependentPolys}]\label{2.Cor:DensityFraKu}
    Let $(X,\mathcal A,\mu)$ be a probability space. For any $d,\ell\in\N$, any commuting invertible $\mu$-preserving transformations  $T_1,...,T_\ell$, any  $\Q$-linearly independent polynomials $p_1,...,p_\ell\in \Z[x_1,...,x_d]$ with zero constant term, and any $f_0,f_1,...,f_\ell\in L^\infty(\mu)$, one has 
\begin{multline}\label{eq:FraKuCor}
    \lim_{N\rightarrow\infty}\frac{1}{|\Phi_N|}\sum_{\vec n\in\Phi_N}\left|\int_Xf_0T_1^{p_1(\vec n)}f_1\cdots T_\ell^{p_\ell(\vec n)}f_\ell\text{d}\mu\right.\\
    \left.-\int_X f_0T_1^{p_1(\vec n)}\mathbb E(f_1\,|\,\mathcal K_{T_1})\cdots T_\ell^{p_\ell(\vec n)}\mathbb E(f_\ell\,|\,\mathcal K_{T_\ell})\text{d}\mu\right|=0
\end{multline}
for any F{\o}lner sequence $(\Phi_N)_{N\in\N}$ in $\Z^d$.
\end{cor}
\begin{proof}
The  proof is similar to that of  Lemma 4.3 in \cite{FraKra2006IndependentPolys}.
    Fix a F{\o}lner sequence $(\Phi_N)_{N\in\N}$ in $\Z^d$. For any complex numbers $a_0,...,a_\ell$ and $b_0,...,b_\ell$ one has that
\begin{equation}\label{eq:IntegralZeroIdentity}
        \prod_{j=0}^\ell a_j-\prod_{j=0}^\ell b_j
    =(a_0-b_0)\prod_{j=1}^\ell a_j+b_0(a_1-b_1)\prod_{j=2}^\ell a_j+
    \cdots+\prod_{j=0}^{\ell-1}b_j(a_\ell-b_\ell).
    \end{equation}
    Thus, in order to prove \eqref{eq:FraKuCor}, it is enough to show that 
\begin{equation}\label{eq:StrongConvergence1}
    \lim_{N\rightarrow\infty}\frac{1}{|\Phi_N|}\sum_{\vec n\in\Phi_N}|\int_Xf_0T_1^{p_1(\vec n)}f_1\cdots T_\ell^{p_\ell(\vec n)}f_\ell\text{d}\mu|^2=0,
    \end{equation}
    provided that for some $j\in\{1,...,\ell\}$, $\mathbb E(f_j\,|\,\mathcal K_{T_j})=0$. Without loss of generality, we assume $\mathbb E(f_\ell\,|\,\mathcal K_{T_\ell})=0$. Note that \eqref{eq:StrongConvergence1} holds whenever
    \begin{equation}\label{eq:StrongConvergence2}
    \lim_{N\rightarrow\infty}\frac{1}{|\Phi_N|}\sum_{\vec n\in\Phi_N}\int_{X\times X}f_0\otimes \overline {f_0}(T_1\times T_1)^{p_1(\vec n)}f_1\otimes \overline {f_1}\cdots(T_\ell\times T_\ell)^{p_\ell(\vec n)}f_\ell\otimes \overline {f_\ell}\text{d}(\mu\otimes\mu)=0.
    \end{equation}
    (Here $\overline f_j$ denotes the complex conjugate of $f_j$ for $j\in \{0,...,\ell\}$). By a classical fact (see for example  \cite[Lemma 4.18]{Fbook}), one has that $\mathcal K_{T_\ell\times T_\ell}=\mathcal K_{T_\ell}\otimes \mathcal K_{T_\ell}$. Thus, because $\mathbb E(f_\ell\,|\,\mathcal K_{T_\ell})=0$, we also have that $\mathbb E(f_\ell \otimes \overline{f_\ell}\,|\,\mathcal K_{T_\ell\times T_\ell})=0$. Invoking \cref{FraKuTheorem}, we obtain that \eqref{eq:StrongConvergence2} holds and, so, \eqref{eq:FraKuCor} also holds. 
\end{proof}
In the sequel we will also need the following result.
\begin{lem}[Lemma 1.6 in \cite{ChuLowerBoundForCommutingTs}]\label{2.lem:ProdExpectationsLowerBound}
    Let $(X,\mathcal A,\mu)$ be a probability space, let $\ell\geq 1$, and let $\mathcal A_1,...,\mathcal A_\ell$ be sub-$\sigma$-algebras of $\mathcal A$. For any non-negative $f\in L^\infty(\mu)$ one has
    $$\int_X f\mathbb E(f\,|\,\mathcal A_1)\cdots\mathbb E(f\,|\,\mathcal A_\ell)\text{d}\mu\geq \left( \int_X f\text{d}\mu\right)^{\ell+1}.$$
\end{lem}
%%%%%%%%%%%%%%%%%%%%%%%%%%%%%%%%%%%%%%%%%%%%%%%%%%
%%%%%%%%%%%%%%%%%%%%%%%%%%%%%%%%%%%%%%%%%%%%%%%%%%%%%%%%%%%%%%%%%%%%%%%%%
%%%%%%%%%%%%%%%%%%%%%%%%%%%%%%%%%%%%%%%%%%%%%%%%%%%%%%%%%%%%%
%%%%%%%%%%%%%%%%%%%%%%%%%%%%%%%%%%%%%%%%%%%%%
\section{VIP$_0^*$ sets, A-VIP$_0^*$ sets, and a Diophantine result}\label{Sec:3}
Our goal in this section is to introduce the notions of VIP$_0^*$ and A-VIP$_0^*$ sets mentioned in the introduction and formulate and prove a useful Diophantine result which will be used in the sequel.
%%%%%%%%%%%%%%%%%%%%%%%%%%%%%%%%%%%%%%%%%%%%%%%%%%%%%%%%%%%%%%%%%%%%%%%%%%%%%%%%%%%%%%%%%%%
\subsection{{\rm VIP$_0^*$} and {\rm A-VIP$_0^*$} sets}\label{Sec3.1}
The presentation of the material in this subsection closely  follows  \cite[pp. 645-646]{BerLeibNilsystemsIPr}.\\
For any  set $A$, let $\mathcal F(A)$ denote the class of all non-empty finite subsets of $A$ and put $\mathcal F_\emptyset(A):=\mathcal F(A)\cup\{\emptyset\}$. Let $H$ be an additive abelian group and let   $\varphi:\mathcal F_\emptyset (A)\rightarrow H$ be a map. For any $\beta\in\mathcal F_\emptyset (A)$, we define the $\beta$-derivative of $\varphi$ to be the map $D_\beta\varphi:\mathcal F_\emptyset(A\setminus \beta)\rightarrow H$ given by 
$$D_\beta\varphi(\alpha)=\varphi(\alpha\cup \beta)-\varphi(\alpha).$$
Given $r\in\N$, we say that $\varphi:\mathcal F_\emptyset(\{1,...,r\})\rightarrow H$ is a \textit{VIP$_r$-polynomial} of degree at most $k$, if for any pairwise disjoint $\beta_0,...,\beta_k\in\mathcal F_\emptyset (\{1,...,r\})$, $D_{\beta_0}D_{\beta_1}\cdots D_{\beta_k}\varphi=0$.
\begin{example}
    Let $r,d\in\N$ and let $\{\vec n_1,...,\vec n_r\}$ be an $r$-element subset of $\Z^d$. For any $\alpha\in \mathcal F_\emptyset(\{1,...,r\})$, let 
    $$\vec n_\alpha=\sum_{k\in\alpha}\vec n_k,$$
    where, by convention, we take $\sum_{k\in\emptyset}\vec n_k=0$. For any polynomial $p\in \Z[x_1,...,x_d]$, the map $\varphi:\mathcal F_\emptyset(\{1,...,r\})\rightarrow \Z$ defined by $\varphi(\alpha)=p(\vec n_\alpha)$ is a VIP$_r$-polynomial of degree at most $k$, where $k=\deg_{x_1,...,x_d} p$ (i.e. $k$ is the least non-negative integer $K$ for which 
 $$
 p(x_1,...,x_d)=\sum_{s=0}^K\sum_{\substack{t_1+\cdots+t_d=s\\t_1,...,t_d\in\N\cup\{0\}}}c_s^{(t_1,...,t_d)}x_1^{t_1}\cdots  x_d^{t_d}.)$$
\end{example}
\begin{defn}\label{3.defnOFVIP*}
For any given $k,d,r\in\N$, a set $E\subseteq \Z^d$ is called a VIP$_{k,r}$ set, if there is a VIP$_r$-polynomial mapping $\varphi:\mathcal F_\emptyset(\{1,...,r\})\rightarrow \Z^d$ of degree at most $k$ satisfying $\varphi(\emptyset)=0$ and such that 
$$\{\varphi(\alpha)\,|\,\alpha\subseteq \{1,...,r\},\,\alpha\neq\emptyset\}\subseteq E.$$
A set $E\subseteq \Z^d$ is called a VIP$_{k,r}^*$ set if it has a non-trivial intersection with every VIP$_{k,r}$ set in $\Z^d$ and say that it is a VIP$_0^
*$ set if for each $k\in\N$ there is an $r\in\N$ for which $E$ is VIP$_{k,r}^*$.
We say that a set $E\subseteq \Z^d$ is A-VIP$_0^*$ if and only if there exists a set $F\subseteq \Z^d$ with $d^*(F)=0$ such that $E\cup F$ is a VIP$_0^*$ set.\\
Notice that, by definition, every VIP$_0^*$ set is an IP$_0^*$ set and every A-VIP$_0^*$ set is an A-IP$_0^*$ set. 
\end{defn}
The following proposition establishes the finite intersection property of   VIP$_0^*$  and A-VIP$_0^*$ sets.
\begin{prop}\label{4.rem:IntersectionProperty}
Let $d\in\N$ and let $E,F\subseteq \Z^d$. If both $E$ and $F$ are {\rm VIP$_0^*$} sets, then $E\cap F$ also is. If both $E$ and $F$ are {\rm A-VIP$_0^*$} sets, then $E\cap F$ also is.
\end{prop}
\begin{proof}
That A-VIP$_0^*$ sets have the finite intersection property can be derived from the fact that VIP$_0^*$ sets also possess this property in a way similar to that in which the corresponding property of A-IP$^*$ sets is deduced from that of IP$^*$ (see the Introduction). Thus, we only need to show that VIP$_0^*$ sets have the finite intersection property. \\
 Let $E,F\subseteq \Z^d$ be VIP$_0^*$ sets and fix $k\in\N$. It suffices to show that $E\cap F$ is a VIP$_{k,R}^*$ set for some $R\in\N$.
 Note that by the definition of VIP$_0^*$ sets,  there exists an $r_0\in\N$  with the property that for any $r\geq r_0$, $E$ and $F$ are  both VIP$_{k,r}^*$ sets.  By \cite[Theorem 3.13]{graham1991RamseyTheory} (which is a finitary version of Hindman's theorem),  there is an $R\in\N$ with the property that for any VIP$_{k,R}$ set $D\subseteq \Z^d$, at least one of  $D\cap E$ and $D\cap (\Z^d\setminus E)$ is a VIP$_{k,r_0}$ set. Since $E$ is a VIP$_{k,r_0}^*$ set,  $D\cap (\Z^d\setminus E)$ cannot be a VIP$_{k,r_0}$ set. It follows that $D\cap E$ is a VIP$_{k,r_0}$ set and, so, $(D\cap E)\cap F\neq\emptyset$. Thus, $E\cap F$ is a VIP$_{k,R}^*$ set. We are done. 
\end{proof}
%%%%%%%%%%%%%%%%%%%%%%%%%%%%%%%%%%%%%%%%%%%%%
\subsection{A Diophantine result involving polynomials and VIP$_0^*$ sets}
In this section we prove a Diophantine result, \cref{3.thm:DiophantineRec}, which is needed for the proof of the implication (i)$\implies$(ii$^\prime$) in \cref{0.MainResultLebesgueSpaces} in Section \ref{Sec4}. We remark that  some of the consequences of \cref{3.thm:DiophantineRec} were implicitly used in the proofs of \cite[Theorem 1.4]{bergelson2008intersective} and \cite[Theorem 4.10]{AlmostIPBerLeib}. 
\begin{thm}\label{3.thm:DiophantineRec}
    Let $d,\ell,t\in\N$, let $p_1,...,p_\ell\in\Z[x_1,...,x_d]$ be a family of non-constant, jointly intersective polynomials, and let $\alpha_1,...,\alpha_t\in\R$. For every $\epsilon>0$ there is a ${\vec u}\in \Z^d$ for which the set  
    \begin{equation}\label{3.eq:DiophantineReturns}
\left(\bigcap_{j=1}^\ell\bigcap_{s=1}^t\{\vec n\in\Z^d\,|\,\|p_j(\vec n)\alpha_s\|<\epsilon\}\right)-{\vec u}
    \end{equation}
    is a {\rm VIP$_0^*$} set. Furthermore, if $p_j(\vec 0)=0$ for each $j\in\{1,...,\ell\}$, one can take ${\vec u}=0$. 
\end{thm}
The proof of \cref{3.thm:DiophantineRec} makes use of the following two lemmas. The first of these lemmas is a consequence of Theorem 0.5 in \cite{BerLeibNilsystemsIPr} and can be viewed as a simultaneous refinement of both Theorem 2.19 in \cite{Fbook} and Theorem 7.7 in \cite{BerUltraAcrossMath}. The second lemma is a simple observation about the linear independence of a family of jointly intersective polynomials.
\begin{lem}\label{3.lem:DiophantineRec}
    Let $d\in\N$ and let $p\in\Z[x_1,...,x_d]$ be such that $p(\vec 0)=0$. For any $\alpha\in\R$ and any $\epsilon>0$, the set 
    $$
\{(n_1,...,n_d)\in\Z^d\,|\,\|p(n_1,...,n_d)\alpha\|<\epsilon\}
    $$
    is a {\rm VIP$_0^*$} set. 
\end{lem}
\begin{lem}[Cf. Corollary 3.5 in \cite{bergelson2008intersective}] \label{3.lem:LinearIndForJointlyIntersective}
    Let $p_1,...,p_\ell\in\Z[x_1,...,x_d]$ be $\Q$-linearly independent and jointly intersective  polynomials.  Let ${\vec u}\in\Z^d$. For each $j\in\{1,...,\ell\}$, set $q_j=p_j-p_j({\vec u})$. Then, the polynomials $q_1,...,q_\ell$ are $\Q$-linearly independent. 
\end{lem}
\begin{proof}[Proof of \cref{3.lem:LinearIndForJointlyIntersective}]
    Let $a_1,...,a_\ell\in\Z$ be such that $a_1q_1+\cdots+a_\ell q_\ell=0$. Then, $\sum_{j=1}^\ell a_j p_j=\sum_{j=1}^\ell a_jp_j(\vec  u)$. Furthermore, since $p_1,...,p_\ell$ are jointly intersective, one has that for every $M\in\N$, there is an $\vec n_M\in\Z^d$ with 
    $$\sum_{j=1}^\ell a_jp_j({\vec u})\equiv \sum_{j=1}^\ell a_jp_j(\vec n_M)\equiv 0\mod M.$$
    It follows that $\sum_{j=1}^\ell a_jp_j({\vec u})=0$. So, by the linear independence of $p_1,...,p_\ell$ and because $\sum_{j=1}^\ell a_jp_j=\sum_{j=1}^\ell a_jp_j({\vec u})$, we obtain that $a_1=\cdots=a_\ell=0$. We are done.
\end{proof}
\begin{proof}[Proof of \cref{3.thm:DiophantineRec}]
 Fix $\epsilon>0$. It suffices to show that there is a ${\vec u}\in\Z^d$ with 
\begin{equation}\label{3.eq:SingleGoodTime}
\|p_j({\vec u})\alpha_s\|<\frac{\epsilon}{2},\,(j,s)\in \{1,...,\ell\}\times\{1,...,t\}.
\end{equation}
Indeed, suppose that such a ${\vec u}$ exists and for each $j\in\{1,...,\ell\}$, set $q_j(\vec n)=p_j(\vec n+\vec u)-p_j(\vec u)$. For any polynomials $v_1,...,v_\ell\in \Z[x_1,...,x_d]$ and any $\delta>0$, set 
$$
R_\delta(v_1,...,v_\ell):=\bigcap_{j=1}^\ell\bigcap_{s=1}^t\{\vec n\in\Z^d\,|\,\|v_j(\vec n)\alpha_s\|<\delta\}.
$$
Observe that by our assumption on $\vec u$,
\begin{equation}\label{3.eq:NiceInclussion}
R_{\epsilon/2}(q_1,...,q_\ell)\subseteq R_\epsilon(p_1,...,p_\ell)-{\vec u}.
\end{equation}
Thus, since  $q_j(\vec 0)=0$ for each $j$, \cref{3.lem:DiophantineRec} and  \cref{4.rem:IntersectionProperty} imply that the set 
$R_{\epsilon/2}(q_1,...,q_\ell)$ is VIP$_0^*$. So, by \eqref{3.eq:NiceInclussion}, $R_\epsilon(p_1,...,p_\ell)-{\vec u}$ is a VIP$_0^*$ set, as claimed.\\

%\noindent $\diamond$ \textit{ Finding ${\vec u}$.}
  Pick $M\in\N$ with the property that for every $s\in\{1,...,t\}$ with $\alpha_s\in \Q$, one has that $M\alpha_s\in\Z$ and let $\vec u_0\in\Z^d$ be such that for every $j\in\{1,...,\ell\}$, $M|p_j(\vec u_0)$. In order to  find $\vec u\in\Z^d$ for which \eqref{3.eq:SingleGoodTime} holds, we only need to check that the polynomials
  $$
Q_j(\vec x)=p_j(M\vec x+\vec u_0),\,j\in\{1,...,\ell\},
  $$
satisfy the following two conditions:
 \begin{enumerate}
     \item [(C.1)] For every $j\in\{1,...,\ell\}$ and every $\vec n\in\Z^d$, $M|Q_j(\vec n)$
     \item [(C.2)] There is an $\vec n_0\in \Z^d$ with the property that for every $(j,s)\in\{1,...,\ell\}\times\{1,...,t\}$, if $\alpha_s\not\in\Q$, then $\|Q_j(\vec n_0)\alpha_s\|<\epsilon/2$.
 \end{enumerate}
 (If both (C.1) and (C.2) hold, $\vec u=M\vec n_0+\vec u_0$ satisfies \eqref{3.eq:SingleGoodTime}.)\\
 Notice that in order to prove that (C.1) and (C.2) hold, we can assume  without loss of generality,  that $p_1,...,p_\ell$ and $\alpha_1,....,\alpha_t$ are $\Q$-linearly independent. As a matter of fact, we can further assume that for each $s\in\{1,...,t\}$, $\alpha_s\in\Q$ if and only if $s=1$.\\ 

To check that condition (C.1) holds for each $j\in\{1,...,\ell\}$, let $q_j(\vec x)=p_j(\vec x+\vec u_0)-p_j(\vec u_0)$ and note that $q_j(\vec 0)=0$. So, for every $\vec n\in\Z^d$, $M|q_j(M\vec n)$. Thus,
$$
M|q_j(M\vec n)+p_j(\vec u_0)=Q_j(\vec n)
$$
for every $\vec n\in\Z^d$, as claimed.\\

To check that (C.2) holds, we use our additional assumptions that $p_1,...,p_\ell$ and $\alpha_1,...,\alpha_t$ are $\Q$-linearly independent and that for $s\neq 1$, $\alpha_s$ is irrational. By \cref{3.lem:LinearIndForJointlyIntersective}, we have that the polynomials $p_j(\vec x)-p_j({\vec u}_0)$, $j\in\{1,...,\ell\}$, are $\Q$-linearly independent and, so, the polynomials $Q_j(\vec x)-Q_j(\vec 0)$, $j\in\{1,...,\ell\}$, are also $\Q$-linearly independent. Thus, by Weyl's equidistribution theorem \cite{weyl1916Mod1}, the  sequence  
$$
(Q_1(\vec n)\alpha_s,....,Q_\ell(\vec n)\alpha_s)_{s=2}^t\mod 1,\,\vec n\in\Z^d,
$$
is equidistributed on $[0,1)^{\ell (t-1)}$. It follows that there exists an ${\vec n_0}\in\Z^d$ with the property that for every $j\in\{1,...,\ell\}$ and every $s\in\{2,...,t\}$,
$$
\| Q_j({\vec n_0})\alpha_s\|<\frac{\epsilon}{2}
$$
as claimed.\\
To complete the proof, notice that if $p_j(\vec 0)=0$ for each $j\in\{1,...,\ell\}$, we can take $\vec u_0=\vec n_0=\vec 0$ in the previous argument. We are done. 
\end{proof}
The following consequences of \cref{3.thm:DiophantineRec} are of interest on their own right. We omit the proofs.
\begin{cor}
    Let $p_1,...,p_\ell\in\Z[x_1,...,x_d]$ be jointly intersective polynomials and let $\vec \alpha:=(\alpha_1,...,\alpha_\ell)\in \mathbb T^\ell$. Then, for any open neighborhood $U$ of $\vec 0\in \mathbb T^\ell$, the set \
    $$
\{\vec n\in\Z^d\,|\,(p_1(\vec n)\alpha_1,...,p_\ell(\vec n)\alpha_\ell)\in U\}
    $$
    is a $\Z^d$-translate of a {\rm VIP$_0^*$} set.
\end{cor}
\begin{cor}[Cf. Theorem 0.5 in \cite{BerLeibNilsystemsIPr} ]
    For any $p\in\R[x_1,...,x_d]$ with $p(\vec 0)=0$, the set 
    $$
\{\vec n\in\Z^d\,|\,\|p(\vec n)\|<\epsilon \}
    $$
    is {\rm VIP$_0^*$}. 
\end{cor}
%%%%%%%%%%%%%%%%%%%%%%%%%%%%%%%%%%%%%%%%%%%%%%%%
\section{The proof of the implication (i)$\implies$(ii$^\prime$) in \cref{0.MainResultLebesgueSpaces}}\label{Sec4}
In this section we prove  the implication (i)$\implies$(ii$^\prime$) in \cref{0.MainResultLebesgueSpaces}. We remark that this result generalizes Theorem 4.2 in \cite{AlmostIPBerLeib} (see \Cref{RemarkAboutTheorem4.2} below for the details).
\begin{thm}\label{4.GenBerLeib}
    Let $p_1,...,p_\ell\in\Z[x_1,...,x_d]$ be $\Q$-linearly independent and jointly intersective  polynomials. Let  $(X,\mathcal A,\mu)$ be a probability space and let $T_1,...,T_\ell$ be invertible and commuting $\mu$-preserving transformations. For every $\epsilon>0$ and every $A\in\mathcal A$,   the set
    \begin{equation}
        R_\epsilon^{p_1,...,p_\ell}(A)=\{\vec n\in\Z^d\,|\,\mu(A\cap T_1^{-p_1(\vec n)}A\cap \cdots\cap T_\ell^{-p_\ell(\vec n)}A)>\mu^{\ell+1}(A)-\epsilon\}
    \end{equation}
    has the property that for some ${\vec u}\in \Z^d$, $ R_\epsilon^{p_1,...,p_\ell}(A)-{\vec u}$ is {\rm A-VIP$_0^*$}. Furthermore, if $p_j(\vec 0)=0$ for each $j\in\{1,...,\ell\}$, then one can take ${\vec u}=0$. 
\end{thm}
\begin{rem}\label{RemarkAboutTheorem4.2}
It is worth noting that Theorem 4.2 in \cite{AlmostIPBerLeib} states that when $p_1,...,p_\ell\in\Z[x_1,...,x_d]$ are linearly independent polynomials with zero constant term and $T$ is ergodic, every set of the form $R_\epsilon^{p_1,...,p_\ell}(A,T)$ is A-VIP$_0^*$. Thus, 
 \cref{4.GenBerLeib} generalizes \cite[Theorem 4.2]{AlmostIPBerLeib} in two ways. First, instead of dealing with a single, ergodic measure preserving transformation, \cref{4.GenBerLeib} deals with several commuting, not necessarily ergodic, measure preserving transformations. Second,  \cref{4.GenBerLeib} deals with  linearly independent polynomials  in $\Z[x_1,...,x_d]$ which, in principle, not necessarily have  zero constant term. 
\end{rem}
\begin{proof}
Let $\epsilon>0$. By applying 
\cref{2.lem:ProdExpectationsLowerBound} with $\mathcal A_j=\mathcal K_{T_j}$, $j\in\{1,...,\ell\}$, we see that if $\vec m\in\Z^d$ satisfies 
\begin{equation}\label{5.LowerBound}
|\mu(A\cap T_1^{-p_1(\vec m)}A\cap \cdots\cap T_\ell^{-p_\ell(\vec m)}A)-\int_X \mathbbm 1_A\mathbb E(\mathbbm 1_A|\mathcal K_{T_1})\cdots\mathbb E(\mathbbm 1_A|\mathcal K_{T_\ell})\text{d}\mu|<\epsilon
\end{equation}
then 
$$\mu(A\cap T_1^{-p_1(\vec m)}A\cap \cdots\cap T_\ell^{-p_\ell(\vec m)}A)>\int_X \mathbbm 1_A\mathbb E(\mathbbm 1_A|\mathcal K_{T_1})\cdots\mathbb E(\mathbbm 1_A|\mathcal K_{T_\ell})\text{d}\mu-\epsilon\geq \mu^{\ell+1}(A)-\epsilon$$
and, so, 
$$\vec m\in R^{(p_1,...,p_\ell)}_\epsilon(A)=\{\vec n\in\Z^d\,|\,\mu(A\cap T_1^{-p_1(\vec n)}A\cap \cdots\cap T_\ell^{-p_\ell(\vec n)}A)>\mu^{\ell+1}(A)-\epsilon\}.$$ 
Thus, in order to prove that $R^{(p_1,...,p_\ell)}_\epsilon(A)$ is a $\Z^d$-translate of an A-VIP$_0^*$ set, it suffices to show that the following two conditions hold:
\begin{enumerate}
    \item [(C.1)] There exists an $E_\epsilon\subseteq \Z^d$ with $d^*(E_\epsilon)=0$ and the property that for every $\vec n\in \Z^d\setminus E_\epsilon$,
    \begin{equation}\label{4.eq:CondtionC1}
    |\mu(A\cap T_1^{-p_1(\vec n)}A\cap \cdots\cap T_\ell^{-p_\ell(\vec n)}A)
    -\int_X \mathbbm 1_{A} T_1^{p_1(\vec n)}\mathbb E(\mathbbm 1_{A}|\mathcal K_{T_1})\cdots T_\ell^{p_
\ell(\vec n)}\mathbb E(\mathbbm 1_{A}|\mathcal K_{T_\ell})\text{d}\mu|<\frac{\epsilon}{2}.
\end{equation}
\item [(C.2)] The set $\mathcal D$ formed by all $\vec n\in\Z^d$ for which 
\begin{multline}\label{5.eq:CloseReturn}
|\int_X \mathbbm 1_A T_1^{p_1(\vec n)}\mathbb E(\mathbbm 1_A|\mathcal K_{T_1})\cdots T_\ell^{p_
\ell(\vec n)}\mathbb E(\mathbbm 1_A|\mathcal K_{T_\ell})\text{d}\mu\\
-\int_X \mathbbm 1_A\mathbb E(\mathbbm 1_A|\mathcal K_{T_1})\cdots\mathbb E(\mathbbm 1_A|\mathcal K_{T_\ell})\text{d}\mu|<\frac{\epsilon}{2}.
\end{multline}
is a $\Z^d$-translate of a  VIP$_0^*$ set.
\end{enumerate}
(The fact that one can take ${\vec u}=\vec 0$ in the statement of \cref{4.GenBerLeib} whenever $p_j(\vec 0)=0$ for each $j\in\{1,...,\ell\}$, is an immediate consequence of our proof of \cref{4.GenBerLeib}. See the parenthetical comment  at the end of the proof.)\\ 

To prove that condition (C.1) holds,
set $p_0\in \Z[x_1,...,x_d]$ to be the zero polynomial and for each $j\in\{0,1,...,\ell\}$, put $A_j=T^{-p_j(\vec 0)}A$ and $q_j=p_j-p_j(\vec 0)$. By \cref{3.lem:LinearIndForJointlyIntersective}, $q_1,...,q_\ell$ are $\Q$-linearly independent. Thus, by \cref{2.Cor:DensityFraKu}, there is an  $E_\epsilon\subseteq \Z^d$ with $d^*(E_\epsilon)=0$ and the property that 
for any $\vec n\in \Z^d\setminus E_\epsilon$,
\begin{equation}\label{4.eq:ConditionC1'}
|\mu(A_0\cap T_1^{-q_1(\vec n)}A_1\cap \cdots\cap T_\ell^{-q_\ell(\vec n)}A_\ell)-\int_X \mathbbm 1_{A_0} T_1^{q_1(\vec n)}\mathbb E(\mathbbm 1_{A_1}|\mathcal K_{T_1})\cdots T_\ell^{q_
\ell(\vec n)}\mathbb E(\mathbbm 1_{A_\ell}|\mathcal K_{T_\ell})\text{d}\mu|<\frac{\epsilon}{2}.
\end{equation}
Noting that \eqref{4.eq:ConditionC1'} is equivalent to \eqref{4.eq:CondtionC1}, we see that condition (C.1) holds.\\

We now check condition (C.2).
Utilizing formula  \eqref{eq:IntegralZeroIdentity} and the fact that the functions 
$$|\mathbbm 1_A|,|\mathbb E(\mathbbm 1_A|\mathcal K_{T_1})|,...,|\mathbb E(\mathbbm 1_A|\mathcal K_{T_\ell})|$$ are bounded by one  $\mu$-almost everywhere, we see that $\vec n\in\mathcal D$ whenever
\begin{multline}\label{5.eq:BoundIsSumOfBounds}
|\int_X \mathbbm 1_A T_1^{p_1(\vec n)}\mathbb E(\mathbbm 1_A|\mathcal K_{T_1})\cdots T_\ell^{p_
\ell(\vec n)}\mathbb E(\mathbbm 1_A|\mathcal K_{T_\ell})\text{d}\mu
-\int_X \mathbbm 1_A\mathbb E(\mathbbm 1_A|\mathcal K_{T_1})\cdots\mathbb E(\mathbbm 1_A|\mathcal K_{T_\ell})\text{d}\mu|\\
  \leq   \int_X|\mathbbm 1_A-\mathbbm 1_A|\text{d}\mu+\sum_{j=1}^\ell \int_X|T_j^{p_j(\vec n)}\mathbb E(\mathbbm 1_A|\mathcal K_{T_j})-\mathbb E(\mathbbm 1_A|\mathcal K_{T_j})|\text{d}\mu\\
    =\sum_{j=1}^\ell \int_X|T_j^{p_j(\vec n)}\mathbb E(\mathbbm 1_A|\mathcal K_{T_j})-\mathbb E(\mathbbm 1_A|\mathcal K_{T_j})|\text{d}\mu\\
    \leq \sum_{j=1}^\ell \|T_j^{p_j(\vec n)}\mathbb E(\mathbbm 1_A|\mathcal K_{T_j})-\mathbb E(\mathbbm 1_A|\mathcal K_{T_j})\|_{L^2(\mu)}<\frac{\epsilon}{2}.
\end{multline}
    For each $j\in\{1,...,\ell\}$, let $(f_k^{(j)})_{k\in\N}$ be a sequence in $L^2(\mu)$ consisting of  eigenfunctions of $T_j$ with $\|f_k\|_{L^2}=1$, $k\in\N$, and the property that for some sequence $(c_k^{(j)})_{k\in\N}$ in $\C$,
    $$\mathbb E(\mathbbm 1_A|\mathcal K_{T_j})=\sum_{k=1}^\infty c_k^{(j)}f_k^{(j)}.$$
    (Note that one may have infinitely many  of the coefficients $c_k^{(j)}$ equal to zero.  For example, when $\mathbb E(\mathbbm 1_A|\mathcal K_{T_1})=\mu(A)$, we can take $f^{(1)}_k=\mathbbm 1_X$ for each $k\in\N$, $c_1^{(1)}=\mu(A)$, and $c_k^{(1)}=0$ for $k>1$. Also note that, without loss of generality, we can assume that for each $j\in\{1,...,\ell\}$, the set $\{f_k^{(j)}\,|\,\exists k\in\N,\,c_k^{(j)}\neq 0\}$ is a set of orthonormal eigenfunctions.) \\
Let $N\in\N$ be large enough to ensure that for every $j\in\{1,...,\ell\}$,
\begin{equation}\label{5.eq:GoodApproxTruncation}
\|\mathbb E(\mathbbm 1_A|\mathcal K_{T_j})-\sum_{k=1}^N c_k^{(j)}f_k^{(j)}\|_{L^2(\mu)}<
\frac{\epsilon}{6\ell}
\end{equation}
For each $j\in\{1,...,\ell\}$ and each $k\in\{1,...,N\}$, let $\alpha_{j,k}\in \R$ be such that $T_jf_k^{(j)}=e^{2\pi i\alpha_{j,k}}f_k^{(j)}$. We have 
\begin{multline}\label{5.eq:UpperboundOfCompactApprox}
\|T_j^{p_j(\vec n)}\left(\sum_{k=1}^Nc_k^{(j)}f_k^{(j)}\right)-\sum_{k=1}^Nc_k^{(j)}f_k^{(j)}\|_{L^2}\\
=\|\sum_{k=1}^N(e^{2\pi ip_j(\vec n)\alpha_{j,k}}-1)c_k^{(j)}f_k^{(j)}\|_{L^2}\leq \sum_{k=1}^N|c_k^{(j)}||e^{2\pi ip_j(\vec n)\alpha_{j,k}}-1|,
\end{multline}
for every $j\in\{1,...,\ell\}$.\\
We now claim that if $\vec n\in\Z^d$ satisfies 
\begin{equation}\label{5.eq:DiophExpression}
\sum_{k=1}^N|c_k^{(j)}||e^{2\pi ip_j(\vec n)\alpha_{j,k}}-1|<\frac{\epsilon}{6\ell}
\end{equation}
for  each $j\in\{1,...,\ell\}$, then $\vec n\in\mathcal D$.\\ 
Indeed, \eqref{5.eq:UpperboundOfCompactApprox} implies that if $\vec n\in\Z^d$ satisfies \eqref{5.eq:DiophExpression}, then 
\begin{equation}\label{5.eq:GoodApproxAtDiophantineTime}
\|T_j^{p_j(\vec n)}\left(\sum_{k=1}^Nc_k^{(j)}f_k^{(j)}\right)-\sum_{k=1}^Nc_k^{(j)}f_k^{(j)}\|_{L^2}<\frac{\epsilon}{6\ell}.
\end{equation}
holds for every $j\in\{1,...,\ell\}$. Combining  \eqref{5.eq:GoodApproxTruncation} and \eqref{5.eq:GoodApproxAtDiophantineTime}, we obtain that for every $j\in\{1,...,\ell\}$,
\begin{multline*}
\|T_j^{p_j(\vec n)}\mathbb E(\mathbbm 1_A|\mathcal K_{T_j})-\mathbb E(\mathbbm 1_A|\mathcal K_{T_j})\|_{L^2(\mu)}\\
\leq \|T_j^{p_j(\vec n)}\mathbb E(\mathbbm 1_A|\mathcal K_{T_j})-T_j^{p_j(\vec n)}\left(\sum_{k=1}^N c_k^{(j)}f_k^{(j)}\right)\|_{L^2}+\|T_j^{p_j(\vec n)}\left(\sum_{k=1}^N c_k^{(j)}f_k^{(j)}\right)-\sum_{k=1}^N c_k^{(j)}f_k^{(j)}\|_{L^2}\\
+\|\sum_{k=1}^N c_k^{(j)}f_k^{(j)}-\mathbb E(\mathbbm 1_A|\mathcal K_{T_j})\|_{L^2}\\
=2\|\sum_{k=1}^N c_k^{(j)}f_k^{(j)}-\mathbb E(\mathbbm 1_A|\mathcal K_{T_j})\|_{L^2}+\|T_j^{p_j(\vec n)}\left(\sum_{k=1}^N c_k^{(j)}f_k^{(j)}\right)-\sum_{k=1}^N c_k^{(j)}f_k^{(j)}\|_{L^2}<\frac{\epsilon}{2\ell}.
\end{multline*}
So, by  \eqref{5.eq:BoundIsSumOfBounds}, we see that $\vec n\in \mathcal D$ as claimed.\\
It now follows from \eqref{5.eq:DiophExpression} that there are real numbers $\alpha_1,...,\alpha_t$ and a $\delta>0$ with the property that the set 
$$
\mathcal C:=\bigcap_{j=1}^\ell \left(\bigcap_{s=1}^t\{\vec n\in \Z^d\,|\,
        \|p_j(\vec n)\alpha_s\|<\delta\}\right)
$$
is a subset of $\mathcal D$. By \cref{3.thm:DiophantineRec} there is a ${\vec u}\in\Z^d$ for which  $\mathcal C-{\vec u}$ is a  VIP$_0^*$ set, and, so, $\mathcal D-\vec u$ is also a VIP$_0^*$ set, proving that condition (C.2) holds. (Notice that when $p_j(\vec 0)=0$ for each $j\in\{1,...,\ell\}$, \cref{3.thm:DiophantineRec} states that one can take $\vec u=\vec 0$).
\end{proof}
%%%%%%%%%%%%%%%%%%%%%%%%%%%%%%%%%%%%%%%%%%%%%%%%%%%
\section{The proof of  a special case of the implication (iii)$\implies$(i) in \cref{0.MainResultLebesgueSpaces} and derivation of \Cref{0.thm:Sharp(iii)->(i)}}\label{Sec:5}
In this section we prove the implication (iii)$\implies$(i) in \cref{0.MainResultLebesgueSpaces} under the additional assumption that the polynomials $p_1,...,p_\ell$ are jointly intersective and then use it to derive \Cref{0.thm:Sharp(iii)->(i)}. As a matter of fact, we will prove a more general result (see \cref{5.thm:SmallIntersections} below). The proof of  (iii)$\implies$(i) in its full generality will be presented in the next section. \\ 
Throughout this section we  set $X=[0,1)^2$, $\mathcal A=\text{Borel}([0,1)^2)$, and let $\mu$ denote the Lebesgue measure on $[0,1)$. Let $T:X\rightarrow X$ be defined by 
$$
T(x,y)=(x,y+x\mod 1).
$$
Note that $T$ is an invertible, $(\mu\otimes\mu)$-preserving transformation.
\begin{thm}[Cf. Theorem 2.1 in \cite{BHKNilSystems2005}]\label{5.thm:SmallIntersections}
    For any integer $\ell>1$ there exists a sequence $A_r^{(\ell)}$, $r\in\N$,  of  $\mathcal A$-measurable sets having positive $\mu\otimes\mu$-measure with the following property:\\
    \begin{adjustwidth}{0.5em}{0.5em}
    For any $d\in\N$ and any non-constant and linearly dependent polynomials $p_1,...,p_\ell\in\Z[x_1,...,x_d]$, there exist $a_1,...,a_\ell\in\Z\setminus\{0\}$ and a non-empty $\alpha\subseteq \{1,...,\ell\}$ for which the transformations $T_j=T^{a_j}$, $j=1,...,\ell$, satisfy 
    \begin{multline}\label{LongExpressionsWithFewElements}
\{\vec n\in\Z^d\,|\,\mu\otimes\mu(A_r^{(\ell)}\cap T_1^{-p_1(\vec n)}A_r^{(\ell)}\cap\cdots\cap T_\ell^{- p_\ell(\vec n)}A_r^{(\ell)})
>\frac{(\mu\otimes\mu)^r(A_r^{(\ell)})}{2}\}\\
\subseteq \bigcap_{j\in\alpha}\{\vec n\in\Z^d\,|\,p_j(\vec n)=0\}, 
    \end{multline}
for each $r\in\N$. 
Furthermore, for any non-zero $b_1,...,b_\ell\in\Z$  such that 
$$b_1p_1(x_1,...,x_d)+\cdots+b_\ell p_\ell(x_1,...,x_d)=0,$$
one has that the maps $S_j=T^{b_j}$, $j\in\{1,...,\ell\}$, satisfy
\begin{multline}\label{ShortExpressionWithFewElements}
\{\vec n\in\Z^d\,|\,\mu\otimes\mu(A_r^{(\ell)}\cap S_1^{-p_1(\vec n)}A_r^{(\ell)}\cap\cdots\cap S_\ell^{- p_\ell(\vec n)}A_r^{(\ell)})>\frac{(\mu\otimes\mu)^r(A_r^{(\ell)})}{2}\}\\
=\bigcap_{j=1}^\ell \{\vec n\in\Z^d\,|\,p_j(\vec n)=0\}.
\end{multline}
\end{adjustwidth}
\end{thm}
Before proceeding to the proof of \cref{5.thm:SmallIntersections}, we review the necessary background from additive combinatorics. 
\subsection{A variant of Behrend's result}
The proof of \cref{5.thm:SmallIntersections} utilizes the following  generalization of a classical result due to  Behrend \cite{BehrendSetsWithNo3-APs}. We remark that it is a refinement of Theorem 2.3 in \cite{Ruzsa93LinearEquations} which was formulated without proof.
\begin{lem}\label{10'.lem:Rusza(d-1)}
Let $b\in\N$, $b>1$. There exists a $\beta>0$ with the property that for every  large enough $N\in\N$ there exists a set 
$\Lambda_N\subseteq \{1,...,N\}$ 
which satisfies 
\begin{equation}\label{eq:Lambda_NLowerBound}
    |\Lambda_N|>Ne^{-\beta \sqrt{\log(N)}},
\end{equation}
and such  that for any $t\in\N$, any $a_1,...,a_t\in\N$ with $\sum_{j=1}^ta_j\leq b$, and any $x_1,...,x_{t+1}\in \Lambda_N$, one has 
\begin{equation}\label{eq:equationsOfanyLength}
a_1x_1+\cdots+a_tx_t=(\sum_{j=1}^ta_j)x_{t+1},
\end{equation}
if and only if $x_1=x_2=\cdots=x_{t+1}$. 
\end{lem}
\begin{proof}
    Our proof of \cref{10'.lem:Rusza(d-1)} is a modification of the argument presented in \cite{BehrendSetsWithNo3-APs}. For any natural numbers $d,k,n\in\N$ with $d,n\geq 2$, $k\leq n(d-1)^2$, consider the set 
    $$
\Lambda_{d,k,n}:=\{\sum_{i=0}^{n-1} \gamma_i( b d-1)^i\,|\,\forall i\in\{0,...,n-1\},\,\gamma_i\in\{0,...,d-1\}\text{ and }\sum_{i=0}^{n-1}\gamma_i^2=k\}.
    $$
Pick $t\in\{1,...,b\}$,  let $a_1,...,a_t\in\N$ be such that $a_{t+1}:=\sum_{j=1}^ta_j
\leq b$, and let   $x_1,...,x_{t+1}\in\Lambda_{d,k,n}$ satisfy 
$$a_1x_1+\cdots+a_tx_t=a_{t+1}x_{t+1}.$$
We claim that  
\begin{equation}\label{eq:TrivialSolution}
x_1=x_2=\cdots=x_{t+1}.
\end{equation}
To prove \eqref{eq:TrivialSolution}, for each $j\in \{1,...,t+1\}$, let $\gamma^{(j)}_0,...,\gamma^{(j)}_{n-1}$ be the unique elements of $\{0,...,d-1\}$ with the property that 
\begin{equation}\label{defnOfx_j}
x_j=\sum_{i=0}^{n-1}\gamma_i^{(j)}(bd-1)^i.
\end{equation}
Let $\gamma^{(j)}\in \R^n$ be the vector given by $\gamma^{(j)}=(\gamma^{(j)}_0,...,\gamma^{(j)}_{n-1})$. Note that by the definition of $\Lambda_{d,k,n}$, we have the following equality for the $L^2$-norm of $\gamma^{(j)}$,
$$\|\gamma^{(j)}\|=\sqrt{\sum_{i=0}^{n-1}(\gamma^{(j)}_i)^2}=\sqrt k.$$ 
We now show that $a_{t+1}\gamma^{(t+1)}=\sum_{j=1}^ta_j\gamma^{(j)}$.  To check that this is the case, we first note that by the definition of $\Lambda_{d,k,n}$, for every $i\in\{0,...,n-1\}$,
\begin{equation}\label{eq:NotWraparround}
\sum_{j=1}^ta_j \gamma^{(j)}_i\leq a_{t+1}(d-1)<bd-1.
\end{equation}
and, so, by \eqref{defnOfx_j} and \eqref{eq:NotWraparround}, 
$$
a_{t+1}\sum_{i=0}^{n-1}\gamma_i^{(t+1)}(bd-1)^i=a_{t+1}x_{t+1}=\sum_{j=1}^ta_jx_j= \sum_{i=0}^{n-1}\left(\sum_{j=1}^ta_j \gamma^{(j)}_i\right)(bd-1)^i,
$$
which proves that $a_{t+1}\gamma^{(t+1)}=\sum_{j=1}^ta_j\gamma^{(j)}$.\\
Thus, by the triangle inequality for the $L^2$-norm,
$$a_{t+1}\sqrt{k}=\|a_{t+1}\gamma^{(t+1)}\|= 
\|\sum_{j=1}^ta_j\gamma^{(j)}\|\leq \sum_{j=1}^ta_j\|\gamma^{(j)}\|=a_{t+1}\sqrt k.$$
So, $\|a_{t+1}\gamma^{(t+1)}\|= \sum_{j=1}^ta_j\|\gamma^{(j)}\|$ which together with the fact that 
 $\|\gamma^{(j)}\|=\sqrt k$ for all $j\in\{1,...,t+1\}$, imply $\gamma^{(1)}=\gamma^{(2)}=\cdots=\gamma^{(t+1)}$ (which in turn implies that \eqref{eq:TrivialSolution} holds).\\
Note now that 
$$
|\bigcup_{k=1}^{n(d-1)^2}\Lambda_{d,k,n}|=d^n-1.
$$
So, we  can find a $k_{d,n}\in \{1,...,n(d-1)^2\}$ with 
\begin{equation}\label{eq:LowerBoundForLambda}
|\Lambda_{d,k_{d,n},n}|\geq \frac{d^n-1}{n(d-1)^2}>\frac{d^{n-2}}{n}\geq \frac{d^{n-b}}{n}.
\end{equation}
Consider $N\in\N$ with $N>b^{4b}$. Let 
\begin{equation}\label{eq:DefnOfn}
    n:=\left\lfloor \sqrt{\frac{b\log N}{\log b}}\right\rfloor 
\end{equation}
and pick $d\in\N$ such that
\begin{equation}\label{eq:DefnOfd}
(db)^n\leq N<((d+1)b)^n. 
\end{equation}
(Note that a sufficient condition for the existence of $d$ with $d>b-1$ is that $N>b^{4b}$. Also note that since $N>b^{4b}$, $n\geq 2b$.)\\
Set $\Lambda_N=\Lambda_{d,k_{d,n},n}$. Note that by \eqref{eq:DefnOfd}, $\Lambda_N\subseteq \{1,...,N\}$. Also note  that by  \eqref{eq:LowerBoundForLambda},
\begin{equation*}
    |\Lambda_N|>\frac{d^{n-b}}{n}>\frac{(N^{1/n}-b)^{n-b}}{nb^{n-b}}
    =\frac{N^{1-b/n}}{nb^{n-b}}(1-bN^{-1/n})^{n-b}.
\end{equation*}
To prove that $\Lambda_N$ satisfies \eqref{eq:Lambda_NLowerBound}, it is enough to note that for any $\epsilon,c>0$ and  sufficiently large $N$, 
\begin{equation*}
    \frac{N^{1-b/n}}{nb^{n-b}}=N^{1-\frac{b}{n}-\frac{\log (n)}{\log(N)}-(n-b)\frac{\log(b)}{\log(N)}}>N^{1-2\sqrt{\frac{b\log(b)}{\log(N)}}-\frac{\epsilon}{\sqrt{\log(N)}}}=Ne^{-(2\sqrt{b\log(b)}+\epsilon)\sqrt{\log(N)}}
\end{equation*}
and 
\begin{equation*}
\frac{(n-b)\log(1-bN^{-1/n})}{\sqrt{\log(N)}}>-c
\end{equation*}
We are done.  
\end{proof}
%%%%%%%%%%%%%%%%%%%%%%%%%%%%%%
\subsection{The proof of \cref{5.thm:SmallIntersections}}
\begin{proof}[Proof of \cref{5.thm:SmallIntersections}]
We use a construction similar to the one used to prove Theorem 2.1 in \cite{BHKNilSystems2005}.
Fix $r\in\N$. We will first define the set $A_r^{(\ell)}$ and then show that both formulas  \eqref{LongExpressionsWithFewElements} and \eqref{ShortExpressionWithFewElements} hold.\\

 For  every sufficiently large  $N\in\N$, let $\Lambda_N$ denote the set guaranteed to exist by \cref{10'.lem:Rusza(d-1)} with $b=\ell$. For each $j\in\Lambda_N$, we set
$$
\mathcal I_N(j):= \frac{j}{2N(\ell+1)}+[0,\frac{1}{8N(\ell+1)^2})
$$
and set $\mathcal I_N=\bigcup_{j\in\Lambda_N}\mathcal I_N(j)$. Suppose that there exist a $t\in\{1,...,\ell\}$ and $x_1,...,x_{t+1}\in \mathcal I_N$ with 
\begin{equation}\label{eq:SumInSet}
x_1+\cdots+x_t=tx_{t+1}\mod 1.
\end{equation}
Then, we can find $j_1,...,j_{t+1}\in\Lambda_N$  such that 
for every $s\in\{1,...,t+1\}$, $x_s\in\mathcal I_N(j_s)$.
We claim that 
$$
j_1+\cdots+j_t=tj_{t+1}.
$$
Indeed, noting that  
$|\sum_{j=1}^tx_j|,|tx_{t+1}|<5/8$, we see  that $x_1+\cdots+x_t=tx_{t+1}$ and, so, 
\begin{multline}
    \frac{1}{2N(\ell+1)}|\sum_{s=1}^tj_s-tj_{t+1}|=|\sum_{s=1}^t \left(\frac{j_s}{2N(\ell+1)}-x_s\right)-t\left(\frac{j_{t+1}}{2N(\ell+1)}-x_{t+1}\right)|\\
    <\frac{2t}{8N(\ell+1)^2}\leq \frac{1}{4N(\ell+1)}.
\end{multline}
Thus,
$
\frac{1}{2}>|\sum_{s=1}^t j_s-tj_{t+1}|
$, which proves the claim. It now follows that whenever \eqref{eq:SumInSet} holds with $x_1,...,x_{t+1}\in \mathcal I_N$, one can find a $j\in\Lambda_N$ with $x_1,...,x_{t+1}\in \mathcal I_N(j)$.\\
We pick $M\in\N$ large enough to ensure that 
\begin{equation}\label{eq:ChoiceOfM}
2\leq \frac{Me^{-(r-1)\beta\sqrt{\log(M)}}}{(8(\ell+1)^2)^{r-2}}
\end{equation}
and set  $A_r^{(\ell)}=[0,1)\times \mathcal I_M$.\\

We now show that \eqref{LongExpressionsWithFewElements} holds. For this, consider  $d\in\N$ and  the  non-constant, linearly dependent polynomials $p_1,...,p_\ell\in\Z[x_1,...,x_d]$. By re-indexing $p_1,...,p_\ell$, if needed, we can assume that without loss of generality there exist a $t>1$ and non-zero $a_1,...,a_\ell\in\Z\setminus\{0\}$ with
$$
\sum_{j=1}^ta_jp_j(x_1,...,x_d)=0.
$$
Setting $T_j=T^{a_j}$ for each $j\in\{1,...,\ell\}$, we obtain that for every pair $(x,y)\in [0,1)^2$ and any $\vec n\in\Z^d$, one always has that
\begin{equation}
(tx,ty)=\sum_{j=1}^t(x,y+a_jp_j(\vec n)x)=\sum_{j=1}^t T_j^{p_j(\vec n)}(x,y).      
\end{equation}
So, if 
\begin{equation}
(x,y), T_1^{p_1(\vec n)}(x,y),...,T_t^{p_t(\vec n)}(x,y)\in [0,1)\times \mathcal I_M,       
\end{equation}
then one can find a $j\in\Lambda_M$ with 
\begin{equation}
(x,y), T_1^{p_1(\vec n)}(x,y),...,T_t^{p_t(\vec n)}(x,y)\in [0,1)\times \mathcal I_M(j).    
\end{equation}
It is now easy to see that when $p_i(\vec n)\neq 0$ for some $i\in\{1,...,t\}$,
\begin{multline}\label{eq:KeyComputationForLDpolys}
    \mu\otimes\mu(A_r^{(\ell)}\cap T_1^{-p_1(\vec n)}A_r^{(\ell)}\cap \cdots\cap T_\ell^{-p_\ell(\vec n)}A_r^{(\ell)}) \\
    \leq \mu\otimes\mu(A_r^{(\ell)}\cap T_1^{-p_1(\vec n)}A_r^{(\ell)}\cap \cdots\cap T_t^{-p_t(\vec n)}A_r^{(\ell)})\\
    =\int_{[0,1)}\int_{[0,1)}\mathbbm  1_{A_r^{(\ell)}}(x,y)\prod_{j=1}^t\mathbbm 1_{A_r^{(\ell)}}(x,y+a_jp_j(\vec n)x)\text{d}\mu(y)\text{d}\mu(x)\\
    =\sum_{j\in\Lambda_M}\int_{[0,1)}\int_{\mathcal I_M(j)}\prod_{j=1}^t\mathbbm 1_{\mathcal I_M(j)}(y+a_jp_j(\vec n)x)\text{d}\mu(y)\text{d}\mu(x)\\
    \leq \sum_{j\in\Lambda_M}\int_{\mathcal I_M(j)}\int_{[0,1)}\mathbbm 1_{\mathcal I_M(j)}(y+a_ip_i(\vec n)x)\text{d}\mu(x)\text{d}\mu(y)\\
    =\sum_{j\in\Lambda_M}\frac{1}{[8M(\ell+1)^2]^2}=\frac{|\Lambda_M|}{[8M(\ell+1)^2]^2}. 
\end{multline}
By \eqref{eq:ChoiceOfM}, we have 
$$
\frac{|\Lambda_M|}{[8M(\ell+1)^2]^2}\leq \frac{1}{2}\frac{(Me^{-\beta\sqrt{\log(M)}})^{r-1}}{[8M(\ell+1)^2]^{r-2}}\frac{|\Lambda_M|}{[8M(\ell+1)^2]^2}\leq \frac{1}{2}\left(\frac{|\Lambda_M|}{[8M(\ell+1)^2]}\right)^r=\frac{1}{2}(\mu\otimes\mu)^r(A_r^{(\ell)}).
$$
It follows that 
\begin{multline}\label{eq:BoundedByDeg}
\{\vec n\in\Z^d\,|\,\mu\otimes\mu(A_r^{(\ell)}\cap T_1^{-p_1(\vec n)}A_r^{(\ell)}\cap\cdots\cap T_\ell^{- p_\ell(\vec n)}A_r^{(\ell)})> \frac{(\mu\otimes\mu)^r(A_r^{(\ell)})}{2}\}\\
\subseteq \bigcap_{j=1}^t\{\vec n\in\Z^d\,|\,p_j(\vec n)=0\}.
\end{multline}
To prove that  \eqref{ShortExpressionWithFewElements} holds, suppose that there exist $b_1,...,b_\ell\in\Z\setminus\{0\}$ such that $\sum_{j=1}^\ell b_jp_j(\vec x)=0$ and for each $j\in\{1,...,\ell\}$, set $S_j=T^{b_j}$. By the previous argument applied with $t=\ell$, we have  
\begin{multline*}
\{\vec n\in\Z\,|\,\mu\otimes\mu(A_r^{(\ell)}\cap S_1^{-p_1(\vec n)}A_r^{(\ell)}\cap\cdots\cap S_\ell^{- p_\ell(\vec n)}A_r^{(\ell)})>\frac{(\mu\otimes\mu)^r(A_r^{(\ell)})}{2}\}\\
\subseteq \bigcap_{j=1}^\ell \{\vec n\in\Z^d\,|\,p_j(\vec n)=0\}.
\end{multline*}
On the other hand,  for any $\vec n\in\Z^d$ with $p_j(\vec n)=0$ for each $j\in\{1,...,\ell\}$, 
$$
\mu\otimes\mu(A_r^{(\ell)}\cap S_1^{-p_1(\vec n)}A_r^{(\ell)}\cap\cdots\cap S_\ell^{- p_\ell(\vec n)}A_r^{(\ell)})=\mu\otimes\mu(A_r^{(\ell)}).
$$
We are done.  
\end{proof}
%%%%%%%%%%%%%%%%%%%%%%%%%%%%%%%%%%%%%
\subsection{The proof of \cref{0.thm:Sharp(iii)->(i)}}
\begin{proof}[Proof of \cref{0.thm:Sharp(iii)->(i)}]
Let $d,\ell\in\N$, $p_1,...,p_\ell\in\Z[x_1,...,x_d]$, $\alpha\subseteq \{1,...,\ell\}$ and $a_j$, $j\in\alpha$,  be as in the statement of \cref{0.thm:Sharp(iii)->(i)}. Set $t=|\alpha|$ and suppose without loss of generality that $\alpha=\{1,...,t\}$.  Applying formula \eqref{ShortExpressionWithFewElements} in  \cref{5.thm:SmallIntersections}  to the polynomials $p_1,...,p_t$, we get that the $\mu\otimes\mu$-preserving and commuting transformations $T_j=T^{a_j}$, $j\in\{1,...,t\}$, have the property that for every $r\in\N$, the set  $A_r:=A_{r}^{(t)}$ satisfies 
$$
\{\vec n\in\Z^d\,|\,\mu\otimes\mu(A_{r}\cap T_1^{-p_1(\vec n)}A_{r}\cap\cdots\cap T_t^{- p_t(\vec n)}A_{r})>\frac{(\mu\otimes\mu)^r(A_{r})}{2}\}\\
=\bigcap_{j=1}^t \{\vec n\in\Z^d\,|\,p_j(\vec n)=0\}.
$$
Equation  \ref{0.eq:EqualityWithLongerExpression} now follows by letting $T_j=\text{Id}$ for each $j\in\{1,...,\ell\}\setminus\{1,...,t\}$. 
\end{proof}
%%%%%%%%%%%%%%%%%%%%%%%%%%
%%%%%%%%%%%%%%%%%%%%%%%%%%%%%%%%%%%%%%%%%%%%%%%%%%%%%%%%%%%%%%%
\section{Completion of the proof of \cref{0.MainResultLebesgueSpaces}}\label{Sec:6}
In this section we wrap-up the proof of \cref{0.MainResultLebesgueSpaces}
by employing the results obtained in Sections 3--5.
\begin{proof}[Proof of \cref{0.MainResultLebesgueSpaces}]
(i)$\implies$(\rm{ii$^\prime$}): This implication follows from \cref{4.GenBerLeib}.\\
(\rm{ii$^\prime$})$\implies$(ii): As was mentioned in the introduction and in Subsection 3.1, every A-VIP$_0^*$ set is an A-IP$_0^*$ set.\\
(ii)$\implies$(iii): For any non-constant polynomial $p\in\Z[x_1,...,x_d]$, the set $E:=\{\vec n\in\Z^d\,|\,p(\vec n)=0\}$ satisfies $d^*(E)=0$.\footnote{Here is one of the many ways of deriving this fact. Let $\alpha$ be an irrational number and let $$t_\epsilon=d^*(\{\vec n\in\Z^d\,|\,\|p(\vec n)\alpha\|<\epsilon\}),\,\epsilon>0.$$
By Weyl's equidistribution theorem,  $\lim_{\epsilon\rightarrow 0^+}t_\epsilon=0$. Noting that $E\subseteq \{\vec n\in\Z^d\,|\,\|p(\vec n)\alpha\|<\epsilon\}$ for every positive value of $\epsilon$, we deduce that $d^*(E)=0$.} Thus, because every A-IP$_0^*$ set $D\subseteq \Z^d$ is syndetic (and, so, satisfies $d^*(D)>0$), we obtain that (iii) holds. \\
(iii)$\implies$(i): We will prove the contrapositive. So assume that the non-constant polynomials $p_1,...,p_\ell\in\Z[x_1,...,x_d]$ fail to be either linearly independent or jointly intersective. If $p_1,...,p_\ell$ are not linearly independent, then the result follows from \cref{5.thm:SmallIntersections}. On the other hand, if $p_1,...,p_\ell$ fail to be  jointly intersective but are linearly independent, one can easily find a counterexample in systems of finite cardinality as follows. Because $p_1,...,p_\ell$ are not jointly intersective, there is   an $M\in\N$ with the property that for each $\vec n\in\Z^d$ there is a $j\in\{1,...,\ell\}$ with $p_j(\vec n)\not\equiv  0\mod M$. Let $T_M:\Z/M\Z\rightarrow \Z/M\Z$ be the map given by $T_Mx\equiv x+1\mod M$. Letting $A=\{0\}\subseteq\Z/M\Z$, we obtain
$$
\{\vec n\in\Z^d\,|\,\frac{|A\cap T_M^{-p_1(\vec n)}A\cdots \cap T_M^{-p_\ell(\vec n)}A|}{M}=0\}=\Z^d.
$$
Thus, taking $\epsilon=\frac{1}{2M^{\ell+1}}$, we obtain 
$$R_\epsilon^{p_1,...,p_\ell}(A)=\{\vec n\in\Z^d\,|\,\frac{|A\cap T_M^{-p_1(\vec n)}A\cdots \cap T_M^{-p_\ell(\vec n)}A|}{M}>\frac{1}{M^{\ell+1}}-\frac{1}{2M^{\ell+1}}\}=\emptyset$$
which implies that for every $\vec n\in R_\epsilon^{p_1,...,p_\ell}(A)$ and every  $j\in\{1,...,\ell\}$, $p_j(\vec n)= 0$, contrary to (iii).\\
(ii)$\implies$(iv) when $d=1$: The argument presented in the proof of (ii)$\implies$(iii) also shows that, when (ii) holds, every set of the form $R_\epsilon^{p_1,...,p_\ell}(A)$ is infinite.\\
(iv)$\implies$(iii) when $d=1$: This implication is a consequence of the pigeon hole principle. All that one has to note is that if $|R_\epsilon^{p_1,...,p_\ell}(A)|>\deg(p_j)$, then there is an $n\in R_\epsilon^{p_1,...,p_\ell}(A)$ with $p_j(n)\neq 0$.   
\end{proof}
\begin{rem}
     The counterexample used above to prove  the contrapositive of  (iii)$\implies$(i)  when $p_1,...,p_\ell$ are linearly independent but not jointly intersective can be modified to have its underlying  probability space be an atomless Lebesgue space (simply consider the set $[0,1)\times \Z/M\Z$ with its Borel $\sigma$-algebra and the product measure induced by the Lebesgue measure and the normalized counting measure). 
     Furthermore, under the additional assumption that the polynomials $p_j-p_j(\vec 0)$, $j\in\{1,...,\ell\}$, are $\Q$-linearly independent, one has that 
     any such counterexample must involve the existence of non-trivial eigenfunctions. Indeed, \cite{FraKuJoinErgForIndependentPolys} implies that for any probability space $(X,\mathcal A,\mu)$, any $A\in\mathcal A$,  and any invertible and commuting $\mu$-preserving transformations $T_1,...,T_\ell$, if $T_1,...,T_\ell$ possess no non-trivial rational eigenfunctions, then 
    $$
\lim_{N\rightarrow\infty}\frac{1}{|F_N|}\sum_{\vec n\in F_N}\mu(A\cap T_1^{-p_1(\vec n)}A\cap \cdots\cap T_\ell^{-p_\ell(\vec n)}A)=\mu^{\ell+1}(A).
    $$
    for any F{\o}lner sequence $(F_N)_{N\in\N}$ in $\Z^d$.
\end{rem}
%%%%%%%%%%%%%%%%%%%%%%%%%%%%%%%%%%%%%%%%%%%%%%%%%
%%%%%%%%%%%%%%
\section{The proof of \cref{0.thm:IPoFailure}}\label{Sec7}
Our goal in this section is to prove the following result which implies \cref{0.thm:IPoFailure}. For any $d\in\N$ and any non-zero polynomial $p\in\Z[x_1,...,x_d]$, the degree of $p$, denoted $\deg_{x_1,...,x_d}(p)$, is the maximum sum of the exponents of the variables in any single monomial term forming the polynomial.
\begin{thm}\label{6.thm:KhinthcinFailure}
Let $p\in\Z[x_1,...,x_d]$ be a polynomial with zero constant term and $\deg_{x_1,...,x_d}(p)>1$. 
    There exist an invertible weakly mixing system $(X,\mathcal A,\mu,T)$, a set $A\in\mathcal A$ with $\mu(A)=\frac{1}{2}$, and an $\epsilon>0$ for which the set 
    $$
\{\vec n\in\Z^d\,|\,\mu(A\cap T^{-p(\vec n)}A)> \mu^2(A)-\epsilon\} 
    $$
    is not \rm{IP$_0^*$}.
\end{thm}
As we show in Subsection \ref{sec:7.1} below, \cref{6.thm:KhinthcinFailure} can be derived from \cite[Theorem A]{ZelIP0Khintchine2023}, which we now state.
\begin{thm}[Theorem A in \cite{ZelIP0Khintchine2023}]\label{6.thm:TheoremA}
Let $p\in\Z[x]$ be a polynomial with zero constant term and $\deg_{x}(p)>1$. 
    There exist an invertible weakly mixing system $(X,\mathcal A,\mu,T)$, a set $A\in\mathcal A$ with $\mu(A)=\frac{1}{2}$, and an $\epsilon>0$ for which the set 
    $$
\{n\in\Z\,|\,\mu(A\cap T^{-p(n)}A)> \mu^2(A)-\epsilon\} 
    $$
    is not \rm{IP$_0^*$}.
\end{thm}
%%%%%%%%%%%%%%%%%%%%%%%%%%%%%%%%%%%%%%%%%%%%%%%%%%%%%
\subsection{Deducing \cref{6.thm:KhinthcinFailure} from Theorem A}\label{sec:7.1}
In order to  derive \cref{6.thm:KhinthcinFailure}   from  \cite[Theorem A]{ZelIP0Khintchine2023} we will need the following algebraic observation. We remark that this result will also be employed to prove \cref{0.thm:IPFailure} in the next section. A set $E\subseteq \Z^d$ is said to have \textit{uniform density one} if  for every F{\o}lner sequence $(\Phi_N)_{N\in\N}$ in $\Z^d$, $\overline d_{(\Phi_N)}(E)=1$.  
\begin{lem}\label{6.lem:DegreeOfLinearRestriction}
    Let $d\in\N$ and let $p\in \Z[x_1,...,x_d]$ be a non-constant polynomial. There is a set $E\subseteq \Z^d$ with uniform density one such that for any $\vec a=(a_1,...,a_d)\in E$ the polynomial 
    $$q_{\vec a}(x)=p(a_1x,...,a_dx)$$
    satisfies $\deg_x(q_{\vec a})=\deg_{x_1,...,x_d}(p)$.
\end{lem}
\begin{proof}[Proof of \cref{6.lem:DegreeOfLinearRestriction}] Let $D=\deg_{x_1,...,x_d}(p)$ and   write
$$
p(x_1,...,x_d)=\sum_{s=0}^D\sum_{\substack{j_1+\cdots+j_d=s\\j_1,...,j_d\in\N\cup\{0\}}}c_{j_1,...,j_d}x_1^{j_1}\cdots x_d^{j_d}.
$$
Consider the polynomial 
$$
u(x_1,...,x_d):= \sum_{\substack{j_1+\cdots+j_d=D\\j_1,...,j_d\in\N\cup\{0\}}}c_{j_1,...,j_d}x_1^{j_1}\cdots x_d^{j_d}
$$
and note that, like $p$, $u$ is a non-constant polynomial. Thus,  the set $F:=\{\vec a\in\Z^d\,|\,u(\vec a)=0 \}$ satisfies $d^*(F)=0$. Letting $E=\Z^d\setminus F$ it now follows that for every $\vec a=(a_1,...,a_d)\in E$, the polynomial  
$$
q_{\vec a}(x)=\sum_{s=0}^D\left(\sum_{\substack{j_1+\cdots+j_d=s\\j_1,...,j_d\in\N\cup\{0\}}}c_{j_1,...,j_d}a_1^{j_1}\cdots a_d^{j_d}\right)x^s.
$$
satisfies $\deg_x(q_{\vec a})=D$.
\end{proof}
\begin{proof}[Proof of \cref{6.thm:KhinthcinFailure}]
Let $p\in \Z[x_1,...,x_d]$ be a polynomial with zero constant term and  $\deg_{x_1,..,x_d}(p)>1$. By \cref{6.lem:DegreeOfLinearRestriction}, there is a non-zero $\vec a=(a_1,...,a_d)\in \Z^d$ for which the polynomial 
$$q(x)=p(a_1x,...,a_dx)$$
satisfies $\deg(q)>1$ and $q(0)=0$. By \cref{6.thm:TheoremA}, there exist an invertible weakly mixing system $(X,\mathcal A,\mu,T)$, a set $A\in\mathcal A$ with $\mu(A)=\frac{1}{2}$, an $\epsilon>0$, and an IP$_0$ set $\Gamma'\subseteq \Z$ (i.e. for each $r\in\N$, $\Gamma'$ contains an IP$_r$ set) with the property that for each $n\in\Gamma'$, 
\begin{equation}\label{6.eq:qConsequenceOfThmA}
\mu(A\cap T^{-q(n)}A)\leq\mu^2(A)-\epsilon.
\end{equation}
Let $\Gamma:=\{(a_1n,...,a_dn)\,|\,n\in \Gamma'\}$. The result now follows by noting that $\Gamma$ is an IP$_0$ set in $\Z^d$ and that for any $\vec n\in \Gamma$ there is an $m\in\Gamma'$ with 
$$
\mu(A\cap T^{-p(\vec n)}A)=\mu(A\cap T^{-q(m)}A)\leq\mu^2(A)-\epsilon.
$$
We are done. 
\end{proof}
%%%%%%%%%%%%%%%%%%%%%%%%%%%%%%%%%%
%%%%%%%%%%%%%%%%%%%%%%%%%%%%%%%%%%%%%%%%%%%%%%%%%%%%%%%%%%%%%%%%%%%%%%%%%%%%%%%%%%%%%
\section{The proof of \cref{0.thm:IPFailure}}\label{Sec8}
Our goal in this section is to prove the following result which implies   \cref{0.thm:IPFailure}. The  polynomials $p_1,...,p_\ell\in\Z[x_1,...,x_d]$ are called \textit{essentially distinct} if for any $i\neq j$, $p_j-p_i$ is a non-constant polynomial. 
\begin{thm}[Cf. Remark 1.16 in \cite{zelada2025coexistence}]\label{6.thm:FailureOfIP*Part2}
Let $d,\ell\in\N$, $\ell>1$, and let $p_1,...,p_\ell\in\Z[x_1,...,x_d]$ be essentially distinct, non-constant polynomials with zero constant term. For each $r\in\N$ there exists an invertible probability  measure preserving system $(X_r,\mathcal A_r,\mu_r,T_r)$  with the property that for some $A_r\in\mathcal A_r$ with $\mu_r(A_r)>0$, the set 
$$
\{\vec n\in\Z^d\,|\,\mu_r(A_r\cap T^{-p_1(\vec n)}A_r\cap \cdots\cap T^{-p_\ell(\vec n)}A_r)>\frac{\mu_r^r(A_r)}{2}\}
$$
is not \rm{IP$^*$}. 
\end{thm}
Before proving   \cref{6.thm:FailureOfIP*Part2} we need to introduce some notation and technical results.
 %%%%%%%%%%%%%%%%%%%%%%%%%
 \subsection{Background in spectral theory}
The proof of \cref{6.thm:FailureOfIP*Part2} makes use  Corollary E  in \cite{BKLUltrafilterPoly}. Before stating Corollary E  we need to introduce some notation.\\

Denote the set of all non-empty, finite subsets of $\N$ by $\mathcal F$. For any sequence $(n_k)_{k\in\N}$ in $\Z$ and any $\alpha\in \mathcal F$, let
$$
n_\alpha=\sum_{j\in\alpha}n_j.
$$
For any compact metric space $(X,d)$ and any $\mathcal F$-indexed  sequence $(x_\alpha)_{\alpha\in\mathcal F}$ in $X$, we
write 
$$
\mathop{\text{IP-lim}}_{\alpha\in\mathcal F}\,x_\alpha=x
$$
if there is an  $x\in X$ such that for any $\epsilon>0$, there exists a $k_\epsilon\in\N$ for which every $\alpha\in\mathcal F$ with $\min \alpha>k_{\epsilon}$ satisfies $d(x,x_\alpha)<\epsilon$.
\begin{cor}[Corollary E in \cite{BKLUltrafilterPoly}.]\label{6.thm:CorollaryE}
    Let $N\in\mathbb N$ and denote by $\mathcal P_{\leq N}$ the set of all polynomials with integer coefficients, zero constant term,  and degree not exceeding  $N$. Let $G$ be a subgroup of the additive group $\mathcal P_{\leq N}$ such that $\max_{p\in G} \deg(p)=N$. Then, the group $G$ has finite index in $\mathcal P_{\leq N}$ if and only if there exists  a Borel probability measure $\sigma$ on $\mathbb T$ and an increasing sequence $(n_k)_{k\in\N}$ in $\N$ such that for any $p\in \mathcal P_{\leq N}$ 
    $$
\mathop{\text{IP-lim}}_{\alpha\in\mathcal F}\int_\mathbb T e^{2\pi i p(n_\alpha)x}\text{d}\sigma(x)=\begin{cases}
1\text{ if }p\in G,\\
0\text{ if }p\not\in G. 
    \end{cases}
    $$
\end{cor}
We now utilize \cref{6.thm:CorollaryE} to deduce the following result which is needed for the proof of \cref{6.thm:FailureOfIP*Part2}.
\begin{thm}[Cf. Theorem 1.8 item (iv) in \cite{zelada2025coexistence}.]\label{6.thm:DistributionResult}
    Let $\ell\in\N$ and let $p_1,...,p_\ell\in\Z[x]$ be linearly independent polynomials with zero constant term. For every $M\in\N$ there exists a Borel probability measure $\sigma$ on $\mathbb T=[0,1)$ and an increasing sequence $(n_k)_{k\in\N}$ in $\N$ with the property that for every continuous function $f:\mathbb T^\ell\rightarrow \R$,
\begin{equation}\label{6.eq:IPDistribution}
\mathop{\text{IP-lim}}_{\alpha\in\mathcal F}\int_\mathbb T f(p_1(n_\alpha)x,...,p_\ell(n_\alpha)x)\text{d}\sigma(x)=\int_{\mathbb T} f(y,2y,...,\ell y)\text{d}\lambda_{M}(y),
\end{equation}
where $\lambda_{M}$ denotes the normalized counting measure on the set $\{\frac{0}{M},...,\frac{M-1}{M}\}$.
\end{thm}
\begin{rem}\label{6.rem:FourierCoefficients}
Let $\ell\in\N$. Recall that any Borel probability measure $\lambda$ on $\mathbb T^\ell$ is determined by its Fourier coefficients 
$$
\int_{\mathbb T^\ell} e^{2\pi i \sum_{j=1}^\ell a_jy_j}\text{d}\lambda(y_1,...,y_\ell),\,a_1,...,a_\ell\in\Z.
$$
When $\ell=1$ and $\lambda=\lambda_M$ is as defined in \cref{6.thm:DistributionResult}, one has that for any $a\in\Z$,
$$
\int_{\mathbb T} e^{2\pi i ay}\text{d}\lambda_M(y)=\begin{cases}
    1,\text{ if }a\in M\Z,\\
    0,\text{ otherwise.}.
\end{cases}
$$
\end{rem}
\begin{proof}
We  consider separately the cases (i) $\ell=\max_{1\leq j\leq \ell}\deg p_j$ and (ii)  $\ell\neq\max_{1\leq j\leq \ell}\deg p_j$. \\

 Suppose first that $\ell=\max_{1\leq j\leq \ell}\deg p_j$ and let 
$$G=\{Ma_1p_1+\sum_{j\in\{1,...,\ell\}\setminus\{1\}}^\ell a_j(jp_1-p_j)\,|\,a_1,...,a_\ell\in \Z\}.$$
Let $b_1,...,b_\ell\in\Z$ be such that $\sum_{j=1}^\ell b_jp_j\in G$. By writing 
$$
\sum_{j=1}^\ell b_jp_j =\left(Ma_1+\sum_{j=2}^\ell ja_j\right)p_1-\sum_{j=2}^\ell a_jp_j
$$
and comparing coefficients, we see that  
$b_1=Ma-\sum_{j=2}^\ell jb_j$
for some $a\in\Z$. Thus,  for any given $b_1,...,b_\ell\in\Z$, 
\begin{equation}\label{6.eq:CharG}
\sum_{j=1}^\ell b_jp_j\in G\text{ if and only if }\sum_{j=1}^\ell jb_j\in M\Z.
\end{equation}
Because $p_1,...,p_\ell$ are linearly independent, we have that for any $t\in\{1,...,\ell\}$, there is  a $K\in\N$ with $Kx^t=\sum_{j=1}^\ell b_jp_j$ for some $b_1,...,b_\ell\in\Z$. By \eqref{6.eq:CharG},  $MKx^t\in G$. So, $G$ has finite index in $\mathcal P_{\leq \ell}$ and   $\max_{p\in G}\deg(p)=\ell$. Invoking \cref{6.thm:CorollaryE} (and formula \eqref{6.eq:CharG}, we obtain that there exist an increasing sequence $(n_k)_{k\in\N}$ in $\N$ and a Borel probability measure $\sigma$ on $\mathbb T$ such that for any $a_1,...,a_\ell\in\Z$,
\begin{equation}\label{6.eq:ConvergenceIndistribution}
\mathop{\text{IP-lim}}_{\alpha\in\mathcal F}\int_\mathbb T e^{2\pi i \left(\sum_{j=1}^\ell a_jp_j(n_\alpha)\right)x}\text{d}\sigma(x)=\begin{cases}
1\text{ if }\sum_{j=1}^\ell ja_j\in M\Z,\\
0\text{ otherwise. } 
    \end{cases}
\end{equation}
In light of \Cref{6.rem:FourierCoefficients}, formula  \eqref{6.eq:ConvergenceIndistribution}  implies that  for every $a\in\Z$,
$$
\mathop{\text{IP-lim}}_{\alpha\in\mathcal F}\int_\mathbb T e^{2\pi iap_1(n_\alpha)x}\text{d}\sigma(x)=\int_\mathbb T e^{2\pi i ay}\text{d}\lambda_M(y).
$$
Also, by formula \eqref{6.eq:ConvergenceIndistribution} and because for each $j\in\{1,...,\ell\}$, $(p_j-jp_1)\in G$, we have that for any $a_1,...,a_\ell\in\Z$,
\begin{multline*}
\mathop{\text{IP-lim}}_{\alpha\in\mathcal F}\left|\int_\mathbb T e^{2\pi i \left(\sum_{j=1}^\ell a_jp_j(n_\alpha)\right)x}\text{d}\sigma(x)-\int_\mathbb T e^{2\pi i \left(\sum_{j=1}^\ell ja_j\right)p_1(n_\alpha)x}\text{d}\sigma(x)\right|\\
\leq \mathop{\text{IP-lim}}_{\alpha\in\mathcal F}\int_\mathbb T \left|e^{2\pi i \sum_{j=1}^\ell a_j\Big(p_j(n_\alpha)-jp_1(n_\alpha)\Big)x}-1\right|\text{d}\sigma(x)=0.
\end{multline*}
Thus, 
\begin{multline*}
    \mathop{\text{IP-lim}}_{\alpha\in\mathcal F}\int_\mathbb T e^{2\pi i \left(\sum_{j=1}^\ell a_jp_j(n_\alpha)\right)x}\text{d}\sigma(x)=\mathop{\text{IP-lim}}_{\alpha\in\mathcal F}\int_\mathbb T e^{2\pi i \left(\sum_{j=1}^\ell ja_j\right)p_1(n_\alpha)x}\text{d}\sigma(x)\\
= \int_\mathbb T e^{2\pi i (a_1y+2a_2y+\cdots+\ell a_\ell y)}\text{d}\lambda_M(y).
\end{multline*}

which implies that formula \eqref{6.eq:IPDistribution} holds.\\

 Suppose now that $N:= \max_{1\leq j\leq \ell}\deg(p_j)$ satisfies $N\neq \ell$.  By the linear independence of $p_1,...,p_\ell$, we have $\ell< N$. Thus, we can find polynomials $p_{\ell+1},...,p_N\in \mathcal P_{\leq N}$ such that $p_1,...,p_N$ are linearly independent. The result now follows by applying the previous case and noting that every continuous function $f:\mathbb T^\ell\rightarrow \R$ can be viewed as a continuous function $f:\mathbb T^N\rightarrow \R$. We are done.  
\end{proof}
%%%%%%%%%%%%%%%%%%%%%%%%%%%%%%%%%%%%%%%%%%%%%%%%%%%%%%%%%%%%%%%%%%%%%%%%%%%%%%%%
\subsection{The proof of \cref{6.thm:FailureOfIP*Part2}}
In order to prove that \cref{6.thm:FailureOfIP*Part2} holds for any $d\in\N$, it suffices to show that it holds in the special case that $d=1$. Indeed, let $d>1$ and let $p_1,...,p_\ell\in\Z[x_1,...,x_d]$ be  essentially distinct, non-constant  polynomials with zero constant term. By \cref{6.lem:DegreeOfLinearRestriction} we have that for every $j\in\{1,...,\ell\}$ there exists a set  $E_j\subseteq \Z^d$  with uniform density one and such that for every $(b_1,...,b_d)\in E_j$,
    $$\deg_x(p_j(b_1x,...,b_dx))=\deg_{x_1,...,x_d}(p_j(x_1,...,x_d)).$$
    Similarly, for any $i\neq j$, there exists a set $E_{i,j}\subseteq \Z^d$ with uniform density one and such that for every $(b_1,...,b_d)\in E_{i,j}$
    $$\deg_x((p_j-p_i)(b_1x,...,b_dx))=\deg_{x_1,...,x_d}((p_j-p_i)(x_1,...x_d)).$$
    Thus, we can find an $\vec a=(a_1,...,a_d)\in\Z^d$ for which the polynomials
    \begin{equation}
        q_j(x):=p_j(a_1x,...,a_dx),\,j\in\{1,...,\ell\},
    \end{equation}
    are essentially distinct, non-constant polynomials with zero constant term in $\Z[x]$. It now follows that in order to prove that a set of the form 
    $$
    \tilde R:=\{\vec n\in\Z^d\,|\,\mu_r(A_r\cap T^{-p_1(\vec n)}A_r\cap \cdots\cap T^{-p_\ell(\vec n)}A_r)>\frac{\mu_r^r(A_r)}{2}\}
    $$
    is not IP$^*$ in  $\Z^d$, it is enough to establish that 
    \begin{equation}\label{7.eq:WantToshow}
R:=\{ n\in\Z\,|\,\mu_r(A_r\cap T^{-q_1(n)}A_r\cap \cdots\cap T^{-q_\ell(n)}A_r)>\frac{\mu_r^r(A_r)}{2}\}
\end{equation}
is not an IP$^*$ set in $\Z$, as claimed. (If there is an IP set  $\Gamma\subseteq \Z$ with $\Gamma\cap R=\emptyset$, then the set 
$$\tilde\Gamma:=\{(a_1n,...,a_dn)\,|\,n\in\Gamma\}$$
is  an IP set in $\Z^d$ and, by the definition of $q_1,...,q_\ell$, it satisfies $\tilde\Gamma\cap \tilde R=\emptyset$.)
\\
Fix now any $\ell$ non-constant and essentially distinct polynomials $p_1,...,p_\ell\in\Z[x]$ having zero constant term. We start by noting that in order to prove \cref{6.thm:FailureOfIP*Part2} (in the case $d=1$) it suffices to show that  for each $r\in\N$, there exist an atomless, standard Lebesgue space $(X_r,\mathcal A_r,\mu_r)$, an invertible $\mu_r$-preserving transformation $T_r$, a set $A_r\in \mathcal A_r$ with $\mu_r(A_r)>0$, and an IP set $\Gamma_r\subseteq \Z$
    for which the set 
    \begin{equation}\label{6.eq:LargeIntersectionSet}
        \{n\in\Z\,|\, \mu_r(A_r\cap T_r^{-p_1(n)}A_r\cap \cdots\cap T_r^{-p_\ell(n)}A_r)> \frac{\mu_r^r(A_r)}{2}\}
    \end{equation}
    has an empty intersection with $\Gamma_r$ and, so, it is not IP$^*$. 

We now describe the general strategy that we will utilize to prove \eqref{6.eq:LargeIntersectionSet}. We will take $X_r=\mathbb T^2$ (which we identify with $[0,1)^2$), $\mathcal A_r=\text{Borel}(\mathbb T^2)$, and $\mu_r=\sigma_r\otimes\mu$, where $\sigma_r$ is an \textit{ad hoc} Borel probability measure on $\mathbb T$ which we will  pick only \textit{after} having defined $A_r$ and $\mu$ denotes the normalized Lebesgue measure on $\mathbb T$. The transformation $T_r:X_r\rightarrow X_r$ is given by $T_r(x,y)=(x,y+x)$. 
  We will define $A_r$ similarly to the definition of $A^{(\ell)}_r$  in the proof of \cref{5.thm:SmallIntersections}. The key feature of $A_r$ is that for some distinct natural numbers $a,b$ (which are picked according to the linear independence properties of $p_1,...,p_\ell$),  one has that 
  $$
(\mu\otimes\mu)(A_r\cap T_r^{-a}A_r\cap T_r^{-b}A_r)\leq \frac{(\mu\otimes\mu)^{r}(A_r)}{8}.
  $$
We will then pick the probability measure $\sigma_r$ together with an IP set $\Gamma_r$ to have the property that for a specific pair of  distinct $i,j\in\{1,...,\ell\}$, the quantity 
\begin{equation*}\label{8.ErrorInstrategy}
|(\sigma_r\otimes\mu)(A_r\cap T_r^{-p_i(n)}A_r\cap T_r^{-p_j(n)}A_r)-(\mu\otimes\mu)(A_r\cap T_r^{-a}A_r\cap T_r^{-b}A_r)|
\end{equation*}
is sufficiently small for every $n\in \Gamma_r$. The result then follows via some rather natural inequalities. \\ 
We will now implement this plan in a sequence of steps.\\
    
$\bullet$ \textit{ Step 1: Picking the set $A_r$ for which 
the set described in \eqref{6.eq:LargeIntersectionSet} is not {\rm IP$^*$}.}   
%We now explain how to pick a set $A_r\in\mathcal A$ with convenient statistical properties which will later allow us to define the measure $\mu_r$.\\  
We begin by picking  two  natural numbers $a$ and $b$ depending on the linear independence properties of the polynomials $p_1,...,p_\ell$. In the case that at least two of $p_1,...,p_\ell$ are linearly independent, we will let $a=1$, $b=2$, and assume that, without loss of generality, $p_1$ and $p_2$ are linearly independent. 
In the case that no two of $p_1,...,p_\ell$ are linearly independent, we will let $a,b\in\Z$ be such that $b p_1=a p_2$. Because $\deg(p_1-p_2)>0$, we have $a\neq b$. Because the conclusion of \cref{6.thm:FailureOfIP*Part2} holds for the family of polynomials $p_1,...,p_\ell$ if and only if it holds for the family of polynomials $-p_1,p_2-p_1,...,p_\ell-p_1$, we can re-index $p_1,...,p_\ell$ and replace $p_1,...,p_\ell$ with $-p_1,p_2-p_1,...,p_\ell-p_1$, if needed, to assume that  without loss of generality $a,b\in\N$ and  $a<b$. The described above choices of the pair $a,b$ (depending on $p_1,p_2$) will be fixed throughout the proof of \cref{6.thm:FailureOfIP*Part2}.\\

By \cref{10'.lem:Rusza(d-1)}, there is a $\beta>0$ such that  for every large enough $m\in\N$, there exists a $\Lambda_m\subseteq \{1,...,m\}$ with  
$$ 
|\Lambda_m|>me^{-\beta\sqrt{\log(m)}}
$$
and such that for every $n_1,n_2,n_3\in \Lambda_m$, the equality 
$$(b-a)n_1+an_2=bn_3$$
implies $n_1=n_2=n_3$.\\
Let 
$$B_{r,m}:=\bigcup_{j\in\Lambda_m}\left(\frac{j}{2m(b+1)}+[0,\frac{1}{8m(b+1)^2})\right).$$
Arguing as in the proof of \cref{5.thm:SmallIntersections}, one can show that whenever $x_1,x_2,x_3\in B_{r,m}$
satisfy $(b-a)x_1+ax_2\equiv bx_3\mod 1$, one has that for some $j\in \Lambda_m$, 
$$x_1,x_2,x_3\in \frac{j}{2m(b+1)}+[0,\frac{1}{8m(b+1)^2}).$$
Note that  for any $y,x\in [0,1)$  with 
$$y,y+bx,y+ax\in B_{r,m}$$
one always has  
$$(b-a)y+a(y+bx)\equiv b(y+ax)\mod 1.$$
It follows, by imitating the proof of   \cref{5.thm:SmallIntersections}, that there is an  $m_0=m_0(r)\in\N$ such that 
\begin{multline}\label{eq:DefnB_rm_0}
\int_{\mathbb T^2} \mathbbm 1_{B_{r,m_0}}(y)\mathbbm 1_{B_{r,m_0}}(y+bx)\mathbbm 1_{B_{r,m_0}}(y+ax)\text{d}\mu(x)\text{d}\mu(y)\leq\frac{|\Lambda_{m_0}|}{[8m_0(b+1)^2]^2}\\
\leq\frac{1}{8}\frac{(m_0e^{-\beta\sqrt{\log(m_0)}})^{r-1}}{[8m_0(b+1)^2]^{r-2}} \frac{|\Lambda_{m_0}|}{[8m_0(b+1)^2]^2}\leq \frac{1}{8}\frac{|\Lambda_{m_0}|^r}{[8m_0(b+1)^2]^r}=\frac{\mu^{r}(B_{r,m_0})}{8},
\end{multline}
 where, as before, $\mu$ denotes the Lebesgue measure. The choice of the value in the last expression of \eqref{eq:DefnB_rm_0} (in particular the concrete value of the denominator) is made in order to facilitate the computations at the end of Step 3.\\
Now let  $A_r=[0,1)\times B_{r,m_0}$. Note that $\mu\otimes\mu(A_r)=\mu(B_{r,m_0})\in (0,1)$ and that our choice of $B_{r,m_0}$ also depends on $p_1$, $p_2$, $a$, and $b$; these facts will be utilized when we define the measure $\mu_r$ in the next step.\\

$\bullet$ \textit{ Step 2: Defining $\mu_r$.}  We will define $\mu_r$ to be a product measure of the form $\sigma_r\times\mu$, where, as before, $\mu$ denotes the Lebesgue measure on $\mathbb T$ and  $\sigma_r$ is a Borel probability measure on $\mathbb T$. The choice of $\sigma_r$ will depend on the relations between  $p_1$ and $p_2$ (and the corresponding choice of $a,b$ described at the beginning of Step 1).\\
For each $N\in\N$, let $\lambda_N$ denote the normalized counting measure on the set $\{\frac{j}{N}\,|\,j=0,...,N-1\}$. 
Note that for any function $f:\mathbb T\rightarrow \{0,1\}$ with finitely many discontinuities and any $y\in\mathbb T$,
\begin{equation}\label{8.eq:VarFormula}
|\int_\mathbb T f(y+ax)f(y+bx)\text{d}\mu(x)-\int_\mathbb T f(y+ax)f(y+bx)\lambda_N(x)|\leq \frac{(a+b)\text{Var}(f)}{2N},
\end{equation}
where $\text{Var}(f)$ denotes the total variation of $f$. Recall that at the beginning of Step 1 we picked $a$ and $b$ depending on  the linear  relation between $p_1$ and $p_2$ and then picked in Step 2 the set $B_{r,m_0}$ accordingly. Applying  \eqref{8.eq:VarFormula} to the function 
$$f(x)=\mathbbm 1_{B_{r,m_0}}(x),$$
we obtain that   for any large enough  $M\in\N$ the following inequality holds for  every $y\in \mathbb T$:
\begin{multline}\label{7.eq:RationalApproximation}
    \left|\int_{\mathbb T}\mathbbm 1_{B_{r,m_0}}(y)\mathbbm 1_{B_{r,m_0}}(y+ax)\mathbbm 1_{B_{r,m_0}}(y+bx)\text{d}\mu(x)\right.\\
    -\left.\int_{\mathbb T}\mathbbm 1_{B_{r,m_0}}(y)\mathbbm 1_{B_{r,m_0}}(y+ax)\mathbbm 1_{B_{r,m_0}}(y+bx)\text{d}\lambda_M(x)\right|<\frac{\mu^r(B_{r,m_0})}{8}\left(1-\mu^2(B_{r,m_0})\right).
\end{multline}
Analogously to our specific choice of the rightmost expression  in \eqref{eq:DefnB_rm_0}, the right-hand side of \eqref{7.eq:RationalApproximation} is selected in order to smooth out the estimates at the end of  Step 3.\\

We will now define the measure $\sigma_r$ with the help of the measure  $\lambda_M$ which appears in formula \eqref{7.eq:RationalApproximation}. We consider two cases: (1)  $p_1$ and $p_2$ are linearly independent and (2)  $bp_1=ap_2$.\\
\textit{ Case 1: } If $p_1,p_2$ are linearly independent,  \cref{6.thm:DistributionResult} guarantees that there exists a Borel probability measure  $\sigma_{r,1}$ on $\mathbb T$ and an increasing sequence  $(n_k)_{k\in\N}$ in $\N$ with the property that for every continuous function $f:\mathbb T^2\rightarrow \R$, 
\begin{equation*}
\mathop{\text{IP-lim}}_{\alpha\in\mathcal F}\int_\mathbb T f(p_1(n_\alpha)x,p_2(n_\alpha)x)\text{d}\sigma_{r,1}(x)=\int_{\mathbb T} f(y,2y)\text{d}\lambda_{M}(y)
=\int_{\mathbb T} f(ay,by)\text{d}\lambda_{M}(y).
\end{equation*}
(Recall that when $p_1$ and $p_2$ are linearly independent, we choose $a=1$ and $b=2$ in Step 1.)\\
\textit{ Case 2:} If $bp_1=ap_2$, we utilize \cref{6.thm:DistributionResult} to pick a Borel probability measure $\nu$ on $\mathbb T$ and an increasing sequence $(n_k)_{k\in\N}$ in $\N$  such that for every continuous $f:\mathbb T\rightarrow \R$,
$$
\mathop{\text{IP-lim}}_{\alpha\in\mathcal F}\int_\mathbb T f(p_1(n_\alpha)x)\text{d}\nu(x)=\int_{\mathbb T} f(y)\text{d}\lambda_{M}(y)
$$
 and  set $\sigma_{r,2}$ to be the unique Borel probability measure on $\mathbb T$ with the property that for every $\gamma\in\Z$,
\begin{equation}\label{7.eq:DefinitionSigma}
\int_\mathbb T e^{2\pi i \gamma x}\text{d}\sigma_{r,2}(x)=\int_\mathbb T e^{2\pi i a\gamma x}\text{d}\nu(x).
\end{equation}
Observe that \eqref{7.eq:DefinitionSigma} implies that for every continuous function $f:\mathbb T^2\rightarrow \R$, 
\begin{multline*}
\mathop{\text{IP-lim}}_{\alpha\in\mathcal F}\int_\mathbb T f(p_1(n_\alpha)x,p_2(n_\alpha)x)\text{d}\sigma_{r,2}(x)=\mathop{\text{IP-lim}}_{\alpha\in\mathcal F}\int_\mathbb T f(ap_1(n_\alpha)x,ap_2(n_\alpha)x)\text{d}\nu(x)\\
=\mathop{\text{IP-lim}}_{\alpha\in\mathcal F}\int_\mathbb T f(ap_1(n_\alpha)x,bp_1(n_\alpha)x)\text{d}\nu(x)=\int_{\mathbb T} f(ay,by)\text{d}\lambda_{M}(y)
\end{multline*}
If Case 1 holds, we  let $\sigma_r=\sigma_{r,1}$. If Case 2 holds, we  let $\sigma_r=\sigma_{r,2}$. Finally, we set $\mu_r=\sigma_r\otimes\mu$.\\

$\bullet$ \textit{ Step 3: An asymptotic property of $\mu_r$.} Consider the $\mathcal F$-indexed sequence   $(\mu_\alpha)_{\alpha\in\mathcal F}$ of Borel probability measures on $\mathbb T^2$ which are defined by their values on the measurable rectangles $A\times B$ as follows 
$$
\mu_{\alpha}(A\times B)=\int_\mathbb T\mathbbm 1_A(p_1(n_\alpha) x)\mathbbm 1_{B}(p_2(n_\alpha)x)\text{d}\sigma_r(x),\,A,B\in \text{Borel}(\mathbb T),
$$
and let $\mu_\infty$ be given by 
$
\mu_{\infty}(A\times B)=\int_\mathbb T\mathbbm 1_A(a x)\mathbbm 1_{B}(b x)\text{d}\lambda_M(x)
$. By our choice of $\sigma_r$, 
\begin{equation}\label{8.eq:weak*Convergence}
\mathop{\text{IP-lim}}_{\alpha\in\mathcal F}\mu_\alpha=\mu_\infty
\end{equation}
in the weak$^*$ topology. Note that the boundary of $B_{r,m_0}$ is a finite set and, so, $$\mu_\infty(\partial([B_{r,m_0}-y]\times [B_{r,m_0}-y]))\leq\mu_\infty(\partial([B_{r,m_0}-y])\times\mathbb T)+\mu_\infty(\mathbb T\times\partial([B_{r,m_0}-y])) =0$$
for all but finitely many $y\in\mathbb T$. Thus, by \eqref{8.eq:weak*Convergence}, 
\begin{multline*}
\mathop{\text{IP-lim}}_{\alpha\in\mathcal F}\int_{\mathbb T}\mathbbm 1_{B_{r,m_0}}(y+p_1(n_\alpha)x)\mathbbm 1_{B_{r,m_0}}(y+p_2(n_\alpha)x)\text{d}\sigma_r(x)\\
=\int_{\mathbb T}\mathbbm 1_{B_{r,m_0}}(y+ax)\mathbbm 1_{B_{r,m_0}}(y+bx)\text{d}\lambda_M(x)
\end{multline*}
for $\mu$-a.e. $y\in \mathbb T$. (Here we are using the fact that if $\mu_\infty(\partial A)=0$ for some set $A$, then the convergence \eqref{8.eq:weak*Convergence} implies that $\mathop{\text{IP-lim}}_{\alpha\in\mathcal F}\mu_\alpha(A)=\mu_\infty(A)$. This fact can be deduced from the definition of IP-limits and Remark 3(iv) in \cite[p. 149]{waltersIntroduction}, for example). 
The Lebesgue dominated convergence theorem now implies that 
\begin{multline*}
\mathop{\text{IP-lim}}_{\alpha\in\mathcal F}\int_{\mathbb T}\int_{\mathbb T}\mathbbm 1_{B_{r,m_0}}(y)\mathbbm 1_{B_{r,m_0}}(y+p_1(n_\alpha)x)\mathbbm 1_{B_{r,m_0}}(y+p_2(n_\alpha)x)\text{d}\sigma_r(x)\text{d}\mu(y)\\
=\int_{\mathbb T}\mathbbm 1_{B_{r,m_0}}(y)\left(\mathop{\text{IP-lim}}_{\alpha\in\mathcal F}\int_{\mathbb T}\mathbbm 1_{B_{r,m_0}}(y+p_1(n_\alpha)x)\mathbbm 1_{B_{r,m_0}}(y+p_2(n_\alpha)x)\text{d}\sigma_r(x)\right)\text{d}\mu(y)\\
=\int_{\mathbb T}\int_{\mathbb T}\mathbbm 1_{B_{r,m_0}}(y)\mathbbm 1_{B_{r,m_0}}(y+ax)\mathbbm 1_{B_{r,m_0}}(y+bx)\text{d}\lambda_M(x)\text{d}\mu(y).
\end{multline*}
Let us now choose a large enough $k_0\in\N$ so  that for every $\alpha\in \mathcal F$ with $k_0<\min \alpha$,
\begin{multline}\label{8.eq:ConvergenceBound}
\int_{\mathbb T}\int_{\mathbb T}\mathbbm 1_{B_{r,m_0}}(y)\mathbbm 1_{B_{r,m_0}}(y+p_1(n_\alpha)x)\mathbbm 1_{B_{r,m_0}}(y+p_2(n_\alpha)x)\text{d}\sigma_r(x)\text{d}\mu(y)\\<\int_{\mathbb T}\int_{\mathbb T}\mathbbm 1_{B_{r,m_0}}(y)\mathbbm 1_{B_{r,m_0}}(y+ax)\mathbbm 1_{B_{r,m_0}}(y+bx)\text{d}\lambda_M(x)\text{d}\mu(y)+\frac{\mu^{r+2}(B_{r,m_0})}{8},
\end{multline}
By formula \eqref{7.eq:RationalApproximation}, the expression on the right of \eqref{8.eq:ConvergenceBound} is bounded above by 
\begin{multline}\label{8.eq:UpperBoundUsingLambda_M}
    \int_{\mathbb T}\int_{\mathbb T}\mathbbm 1_{B_{r,m_0}}(y)\mathbbm 1_{B_{r,m_0}}(y+ax)\mathbbm 1_{B_{r,m_0}}(y+bx)\text{d}\mu(x)\text{d}\mu(y)\\
+\frac{\mu^r(B_{r,m_0})}{8}\left(1-\mu^2(B_{r,m_0})\right)+\frac{\mu^{r+2}(B_{r,m_0})}{8}.
\end{multline}
Thus, by formula \eqref{eq:DefnB_rm_0} we obtain 
\begin{multline*}
\int_{\mathbb T}\int_{\mathbb T}\mathbbm 1_{B_{r,m_0}}(y)\mathbbm 1_{B_{r,m_0}}(y+p_1(n_\alpha)x)\mathbbm 1_{B_{r,m_0}}(y+p_2(n_\alpha)x)\text{d}\sigma_r(x)\text{d}\mu(y)\\
<\int_{\mathbb T^2}\mathbbm 1_{B_{r,m_0}}(y)\mathbbm 1_{B_{r,m_0}}(y+ax)\mathbbm 1_{B_{r,m_0}}(y+bx)\text{d}\mu(x)\text{d}\mu(y)+\frac{\mu^r(B_{r,m_0})}{8}
\leq\frac{\mu^r(B_{r,m_0})}{4}.
\end{multline*}
To conclude the proof note that $T_r$ is an invertible $\mu_r$-preserving transformation and let  $$\Gamma_r:=\{n_{\alpha}\,|\,\min \alpha>k_0\}$$
be the  IP set generated by $(n_k)_{k=k_0+1}^\infty$. Since $\mu_r(A_r)=\mu(B_{r,m_0})$, we have  that  for any $\alpha\in\mathcal F$ with $\min \alpha>k_0$,
\begin{multline}
 \mu_r(A_r\cap T_r^{-p_1(n_\alpha)}A_r\cap \cdots\cap T_r^{-p_\ell(n_\alpha)}A_r)\leq \mu_r(A_r\cap T_r^{-p_1(n_\alpha)}A_r\cap T_r^{-p_2(n_\alpha)}A_r)\\
 =\int_{\mathbb T}\int_{\mathbb T}\mathbbm 1_{B_{r,m_0}}(y)\mathbbm 1_{B_{r,m_0}}(y+p_1(n_\alpha)x)\mathbbm 1_{B_{r,m_0}}(y+p_2(n_\alpha)x)\text{d}\sigma_r(x)\text{d}\mu(y)<\frac{\mu_r^r(A_r)}{4},
 \end{multline}
  proving that \eqref{6.eq:LargeIntersectionSet} holds. We are done.\hfill \qedsymbol
%%%%%%%%%%%%%%%%%%%%%%%%%%%%%%%%%%%%%%%%%%%%%%%%%%%%%%%%%%%%%%%%%%%%%%%%%%%%%%%%%%%%%%%%%%%%%%%%%%%%%%%%%%%%%%%%%%%%%%%%%%%%%
\section{An approximate uniform inverse of the Furstenberg correspondence principle ---the proof of \cref{0.Cor:CombinatorialFailureOfIP}}\label{Sec:9}
The aim of this section is to prove \cref{0.Cor:CombinatorialFailureOfIP}. A key ingredient in our argument is a result of independent interest, which can be regarded as an approximate version of the inverse Furstenberg correspondence principle, and will be established in the general setting of countable abelian groups.  \\
Let $(G,+)$ be a countable abelian group. A sequence $(\Phi_N)_{N\in\N}$ of non-empty, finite  subsets  of $G$ is called a F{\o}lner sequence if for any $g\in G$,
$$
\lim_{N\rightarrow\infty}\frac{|\Phi_N\cap(\Phi_N-g)|}{|\Phi_N|}=1.
$$
We say that a  set $E\subseteq G$ has uniform density $\delta\geq 0$  and write $d_{{\rm U}}(E)=\delta$ if for every F{\o}lner sequence $(\Phi_N)_{N\in\N}$ in $G$ one has
 $$\delta=d_{(\Phi_N)}(E):=\lim_{N\rightarrow\infty}\frac{|E\cap \Phi_N|}{|\Phi_N|}.$$
 Here is now  our "approximate" variant of the inverse Furstenberg correspondence principle. Note that formula \eqref{3.eq:UniformLimitWithError} involves sets having uniform density but, unlike other variants of the inverse Furstenberg correspondence principle (see 
  Theorem 1.2 in \cite{SohailRobinVanDerCorputSets} and Theorem 1.6 in \cite{SaulRMInverseFurstenberg}), the price to pay is the appearance of the error term $k\epsilon$. 
\begin{thm}\label{3.Lem:epsilonInverse} Let $(G,+)$ be a countable abelian group, let $(X,\mathcal A,\mu)$  be a standard Lebesgue space with no atoms (i.e. it is measure-theoretically isomorphic to $[0,1]$ equipped with the Lebesgue measure), and let $(T_g)_{g\in G}$ denote a $\mu$-preserving $G$-action. Suppose that $(T_g)_{g\in G}$ is ergodic. For any $A\in\mathcal A$ with $\mu(A)>0$  and any $\epsilon>0$, there is a set $E\subseteq G$ such that for any $k\in\N$ and any  $g_1,...,g_k\in G$, the quantity $d_{{\rm U}}(\bigcap_{j=1}^k(E-g_j))$ is well-defined and 
\begin{equation}\label{3.eq:UniformLimitWithError}
\left|d_{{\rm U}}(\bigcap_{j=1}^k(E-g_j))-\mu(\bigcap_{j=1}^k T_{-g_j}A)\right|<k\epsilon.
\end{equation}
\end{thm}
In Subsection \ref{Sec:9.1} we will prove \cref{0.Cor:CombinatorialFailureOfIP}. In Subsection \ref{Sec:9.2}, we prove 
\cref{3.Lem:epsilonInverse}.
%%%%%%%%%%%%%%%%%%%%%%%%%%%%%%%%%%%%%%%%%%%%%%%%%%%%%%%%%%%%%%%%%%%%%%%%
\subsection{The proof of \cref{0.Cor:CombinatorialFailureOfIP}}\label{Sec:9.1}
Our goal in this subsection is to prove \cref{0.Cor:CombinatorialFailureOfIP} which we reformulate here for reader's convenience (see \cref{9.Cor:CombinatorialFailureOfIP} below).
\begin{cor}\label{9.Cor:CombinatorialFailureOfIP}
    Let $d,D,\ell\in\N$  be such that $D>\ell>1$. Let $p_1,...,p_\ell\in\Z[x_1,...,x_d]$ be  non-constant polynomials with zero constant term and 
    let ${\bf v_1},...,{\bf v_\ell}\in\Z^D$ be non-zero vectors for which  the  maps $\vec n\mapsto p_j(\vec n){\bf v_j}$, $j\in\{1,...,\ell\}$, are pairwise distinct.
    There exist $\epsilon>0$ and a set $E\subseteq \Z^D$ with $d_{{\rm U}}(E)>0$ such that the set 
    \begin{equation}\label{9.eq:InverseFurstenberg}
\{\vec n\in\Z^d\,|\,d_{{\rm U}}(E\cap (E-p_1(\vec n){\bf v_1})\cap\cdots\cap (E-p_\ell(\vec n){\bf v_\ell}))>(d_{{\rm U}}(E))^{\ell+1}-\epsilon\}
    \end{equation}
    is not {\rm IP$^*$}.  
\end{cor}
\begin{proof}
Because the non-constant polynomials $p_1,...,p_\ell$ have zero constant term and the maps $\vec n\mapsto p_j(\vec n){\bf v_j}$, $j\in\{1,...,\ell\}$, are pairwise distinct, one can find an $\vec n\in \Z^d$ with the property  that the vectors $p_1(\vec n){\bf v_1},...,p_\ell(\vec n){\bf v_\ell}$ are non-zero and pairwise distinct. 
Thus, we can find a non-zero ${\bf c}=(c_1,...,c_D)\in \Z^D$ with the property that for every $j\in\{1,...,\ell\}$,
    $\langle {\bf c},p_j(\vec n){\bf v_j}\rangle\neq 0$ and such that for any $i\neq j$,  $\langle {\bf c},p_j(\vec n){\bf v_j}-p_i(\vec n){\bf v_i}\rangle\neq 0$. 
    It follows that the maps
    \begin{equation}\label{9.eq:Defnq_j}
    q_j(\vec x):=\langle {\bf c},{\bf v_j}\rangle p_j(\vec x),\,j=1,...,\ell,
    \end{equation}
    are non-constant, essentially distinct polynomials with zero constant term. Notice that since $\ell<D$, one can also find a non-zero vector ${\bf w}=(w_1,...,w_D)\in \Z^D$ with the property that $\langle {\bf w},{\bf v_j}\rangle=0$ for each $j\in\{1,...,\ell\}$.
    Replacing ${\bf c}$ with $\langle {\bf w},{\bf w}\rangle{\bf c}-\langle {\bf c},{\bf w}\rangle {\bf w}$, if needed, we can assume  without loss of generality that  $\langle {\bf c},{\bf w}\rangle=0$.\\
    
    By the proof of \cref{6.thm:FailureOfIP*Part2}, there exist an atomless, standard Lebesgue space $(Y,\mathcal B,\nu)$, an invertible $\nu$-preserving transformation $S$, a set $B\in \mathcal B$ with $\nu(B)>0$ and an $\epsilon>0$, for which the set 
    $$
\mathcal R:=\{\vec n\in\Z^d\,|\,\nu(B\cap S^{- q_1(\vec n)}B\cap \cdots\cap S^{-q_\ell(\vec n)}B)> \nu^{\ell+1}(B)-\epsilon\}
    $$
    is not IP$^*$. Let $X=Y^\Z$, let $\mathcal A$ denote its product $\sigma$-algebra, and let $\mu$ be the product measure on $X$. We define the maps $\tilde S,R:X\rightarrow X$ by $[R\omega](n)=\omega(n+1)$ and $[\tilde S\omega](n)=S(\omega(n))$. Observe that $\tilde S$ and $R$ are commuting, $\mu$-preserving transformations. For each $\vec a=(a_1,...,a_D)\in\Z^D$, we define the $\mu$-preserving transformation 
    \begin{equation}\label{9.eq:DefnS}
    T_{\vec a}=\tilde S^{\sum_{j=1}^Dc_ja_j}R^{\sum_{j=1}^Dw_ja_j}.
    \end{equation}
     Note now that because of our choice of ${\bf c}$,  $\langle {\bf c},{\bf w}\rangle=0$, one has that 
    $$T_{ {\bf w}}= \tilde S^{\langle {\bf c},{\bf w}\rangle }R^{\langle {\bf w},{\bf w}\rangle }= R^{\langle {\bf w},{\bf w}\rangle }.$$
It follows that, since $R$ is a shift map, $T_{{\bf w}}$ is ergodic. Thus, the action  $(T_{\vec a})_{\vec a\in\Z^D}$ has an ergodic element which implies that  it is an ergodic $\Z^D$-action.\\
Let $A:=\{\omega \in X\,|\,\omega(0)\in B\}$ and note that $\mu(A)=\nu(B)>0$. Since $(Y,\mathcal B,\nu)$ is an atomless standard Lebesgue space, $(X,\mathcal A,\mu)$ also is and, so,  by \cref{3.Lem:epsilonInverse}, there exists a set $E\subseteq \Z^D$ with the property  that for any $k\in\N$ and any $\vec a_1,...,\vec a_k\in\Z^D$, 
    \begin{equation}\label{9.eq:translateInequality}
\left|d_{{\rm U}}(\bigcap_{j=1}^k(E-\vec a_j))-\mu(\bigcap_{j=1}^k T_{-\vec a_j}A)\right|<k\frac{\epsilon}{3(\ell+1)}.
    \end{equation}
By taking $k=1$ and $\vec a_1=\vec 0$ in \eqref{9.eq:translateInequality}, we see that  formula \eqref{eq:IntegralZeroIdentity} implies that 
$$|d_{{\rm U}}^{\ell+1}(E)-\mu^{\ell+1}(A)|<(\ell+1)\frac{\epsilon}{3(\ell+1)}=\frac{\epsilon}{3}.$$
Also, by formula  \eqref{9.eq:translateInequality}, we have that for any $\vec n\in \Z^d$,
$$|d_{{\rm U}}(E\cap (E-p_1(\vec n){\bf v_1})\cap\cdots\cap (E-p_\ell(\vec n){\bf v_\ell}))-\mu(A\cap T_{-p_1(\vec n){\bf v_1}}A\cap \cdots\cap T_{-p_\ell(n){\bf v_\ell}}A)|<\frac{\epsilon}{3}.$$
It follows that if $\vec n\in\Z^d$ is such that  
$$d_{{\rm U}}(E\cap (E-p_1(\vec n){\bf v_1})\cap\cdots\cap (E-p_\ell(\vec n){\bf v_\ell}))> d_{{\rm U}}^{\ell+1}(E)-\frac{\epsilon}{3},$$
then 
\begin{multline*}
    \mu(A\cap T_{-p_1(\vec n){\bf v_1}}A\cap \cdots\cap T_{-p_\ell(\vec n){\bf v_\ell}}A)\\
    >
    d_{{\rm U}}(E\cap (E-p_1(\vec n){\bf v_1})\cap\cdots\cap (E-p_\ell(\vec n){\bf v_\ell}))-\frac{\epsilon}{3}
   > d_{{\rm U}}^{\ell+1}(E)-2\frac{\epsilon}{3}>\mu^{\ell+1}(A)-\epsilon.
\end{multline*}
Noting that formulas \eqref{9.eq:Defnq_j} and \eqref{9.eq:DefnS} imply that for every $j\in\{1,...,\ell\}$ and every $\vec n\in\Z^d$,
$$
T_{-p_j(\vec n){\bf v_j}}=\tilde S^{-\langle {\bf c},{\bf v_j}\rangle p_j(\vec n)}R^{-\langle {\bf w},{\bf v_j}\rangle p_j(\vec n)}=\tilde S^{-\langle {\bf c},{\bf v_j}\rangle p_j(\vec n)},
$$
we obtain  from the definition of $A$ that 
$$
\mu(A\cap T_{-p_1(\vec n){\bf v_1}}A\cap \cdots\cap T_{-p_\ell(\vec n){\bf v_\ell}}A)=
\nu(B\cap S^{- q_1(\vec n)}B\cap \cdots\cap S^{-q_\ell(\vec n)}B)
$$
and, so, 
\begin{multline*}
\mathcal S:=\{\vec n\in\Z^d\,|\, d_{{\rm U}}(E\cap (E-p_1(\vec n){\bf v_1})\cap\cdots\cap (E-p_\ell(\vec n){\bf v_\ell}))>(d_{{\rm U}}(E))^{\ell+1}-\frac{\epsilon}{3}\}\\
\subseteq \{\vec n\in \Z^d\,|\,\mu(A\cap T_{-p_1(\vec n){\bf v_1}}A\cap \cdots\cap T_{-p_\ell(\vec n){\bf v_\ell}}A)>\mu^{\ell+1}(A)-\epsilon\}\\
= \{\vec n\in\Z^d\,|\,\nu(B\cap S^{- q_1(\vec n)}B\cap \cdots\cap S^{-q_\ell(\vec n)}B)
> \nu^{\ell+1}(B)-\epsilon\}= \mathcal R.
\end{multline*}
So, since $\mathcal R$ is not IP$^*$ and $\mathcal S\subseteq \mathcal R$, we have that $\mathcal S$ is not IP$^*$. We are done.
\end{proof}
\begin{rem}\label{0.WMPrinciple}
Let ${\bf e_1}:=(1,0)$, ${\bf e_2}:=(0,1)$, and pick two $\Q$-linearly independent vectors ${\bf v},{\bf w}\in\Z^2$ with the property that for some $a_1,b_1,a_2,b_2\in\Z$, 
$${\bf e_j}=a_j{\bf v}+b_j{\bf w},\,j\in\{1,2\}.$$
Theorem 1.5 in \cite{FraKu2025ergodic} shows that for any non-constant $p\in\Z[x]$ with $p(0)=0$, any $E\subseteq \Z^2$ with $d^*(E)>0$, and any $\epsilon>0$, the set 
$$
\{n\in\Z\,|\,d^*(E\cap (E-p(n){\bf v})\cap (E-p(n){\bf w}))>(d^*(E))^4-\epsilon\}
$$
is syndetic. We remark that the linear independence of ${\bf v}$ and ${\bf w}$ cannot be omitted.\\
Indeed, \cite[Theorem 2.1]{BHKNilSystems2005} provides an example of an atomless standard Lebesgue space $(Y,\mathcal B,\nu)$, an invertible  $\nu$-preserving transformation $S$, and  a set $B\in\mathcal B$ with $\nu(B)>0$ for which 
$$
\{n\in\Z\,|\,\nu(B\cap S^{-n}B\cap S^{-2n}B)\geq\frac{\nu^{4}(B)}{2}\}=\{0\}.
$$
Let ${\bf v}\in\Z^2$ be a non-zero vector. Arguing as in the proof of \cref{0.Cor:CombinatorialFailureOfIP}, one can find a suitable ergodic probability preserving systems $(X,\mathcal A,\mu,(T_{\vec n})_{\vec n\in\Z^2})$ and  a set $A\in\mathcal A$ with $\mu(A)=\nu(B)>0$  to which one can apply \cref{3.Lem:epsilonInverse} to obtain  a subset $E\subseteq \Z^2$ with $d_{{\rm U}}(E)>0$ and an $\epsilon>0$ for which the set 
    $$\{ n\in\Z\,|\,d_{{\rm U}}(E\cap (E-n{\bf v})\cap (E-2n{\bf v}))>(d_{{\rm U}}(E))^{4}-\epsilon\}=\{0\}.$$
Thus the syndeticity conclusion in \cite[Theorem 1.5]{FraKu2025ergodic} fails for the linearly dependent pair ${\bf v},{2\bf v}$.
\end{rem}
%%%%%%%%%%%%%%%%%%%%%%%%%%%%%%%%%%%%%%%%%%%%%%%%%%%%%%%%%%%%%%%%%%%%%%%
\subsection{The proof of \cref{3.Lem:epsilonInverse}}\label{Sec:9.2}
The proof of \cref{3.Lem:epsilonInverse} makes use of a generalization of the classical Jewett-Krieger Theorem to countable abelian groups, see \cite{Weis1985StricltyErgodic}. Before formulating this result, we need to introduce some relevant definitions.\\
Given a countable abelian group $(G,+)$  and a probability space $(X,\mathcal A,\mu)$ where $X$ is a compact metric space and $\mathcal A=\text{Borel}(X)$, we say that the measure preserving system $(X,\mathcal A,\mu,(T_g)_{g\in G})$ is \textit{strictly ergodic} if (1) $(T_g)_{g\in G}$ is a $\mu$-preserving $G$ action by homeomorphisms of $X$, (2) $\mu$ is the only $(T_g)_{g\in G}$-invariant Borel probability measure in $X$, and (3) for every non-empty open set $U\subseteq X$, $\mu(U)>0$. Observe that the conditions (1),(2),(3) imply that (a) $(T_g)_{g\in G}$ is minimal, meaning that for each $x\in X$, $\{T_gx\,|\,g\in G\}$ is dense in $X$, and (b) for every continuous $f:X\rightarrow \R$, any $x\in X$, and any F{\o}lner sequence $(\Phi_N)_{N\in\N}$ in $G$,
\begin{equation}\label{3.eq:UniformLimitForC(X)}
\lim_{N\rightarrow\infty}\frac{1}{|\Phi_N|}\sum_{g\in \Phi_N}f(T_gx)=\int_X f\text{d}\mu.
\end{equation}
We remark that by a standard approximation argument (see \cite[page 149]{waltersIntroduction}, for example), \eqref{3.eq:UniformLimitForC(X)} implies that for every $V\in\mathcal A$ with $\mu(\partial V)=0$ and   every F{\o}lner sequence $(\Phi_N)_{N\in\N}$ in $G$, one has 
\begin{equation}\label{3.eq:UniformLimitForMeasurable}
\lim_{N\rightarrow\infty}\frac{1}{|\Phi_N|}\sum_{g\in \Phi_N}\mathbbm 1_V(T_gx)=\mu(V),\,\text{ for every }x\in X.
\end{equation}
We will utilize \eqref{3.eq:UniformLimitForMeasurable} in the proof of \cref{3.Lem:epsilonInverse}.
\begin{thm}[Jewett-Krieger Theorem for countable abelian groups, \cite{Weis1985StricltyErgodic}]\label{3.Thm:AbelianJewettKrieger}
    Let $(G,+)$ be a countable abelian group and let $(Y,\mathcal B,\nu)$  be a standard Lebesgue space with no atoms. Suppose that $(S_g)_{g\in G}$ is an ergodic $\nu$-preserving $G$-action. Then, $(Y,\mathcal B,\nu, (S_g)_{g\in G})$ is measure-theoretically isomorphic to a strictly ergodic measure preserving system $(X,\mathcal A,\mu,(T_g)_{g\in G})$.
\end{thm}
\begin{proof}[Proof of \cref{3.Lem:epsilonInverse}] The main argument in our proof of \cref{3.Lem:epsilonInverse} makes use  of the ideas utilized in the proof of Theorem 6.4 in \cite{bergelson2016NonPWSyndetic}.
Let $A\in\mathcal A$ be such that $\mu(A)>0$ and fix $\epsilon>0$. By \cref{3.Thm:AbelianJewettKrieger}, we  assume without loss of generality that $(X,\mathcal A,\mu,(T_g)_{g\in G})$ is strictly ergodic and let $d:X\times X\rightarrow [0,\infty)$ denote the metric on $X$.\\
Our strategy of proving \cref{3.Lem:epsilonInverse} is to first approximate $\mathbbm 1_A$ in the $L^1(\mu)$-norm by a function of the form $\mathbbm 1_{\bigcup_{j=1}^t U_j}$, where $U_1,...,U_t$ is a finite collection of open sets with the additional property that for each $j\in \{1,...,t\}$, $\mu(\partial U_j)=0$. We will then utilize this approximation in combination with \eqref{3.eq:UniformLimitForMeasurable} to define the set $E$.\\
$\bullet$ \textit{ Finding $U_1,...,U_t$.} By the inner regularity of $\mu$, there exists a compact set $K\subseteq A$ such that $\mu(A\setminus K)<\epsilon/2$. For every $r>0$ and every $x\in X$, let $B_r(x)=\{y\in X\,|\,d(y,x)<r\}$. Because $\lim_{r\searrow 0}\mu(\bigcup_{x\in K}B_r(x))=\mu(K)$ and because $K$ is compact we can find $x_1,...,x_t\in K$ and $\sigma,\delta>0$ such that for any $r\in(\sigma,\sigma+\delta)$, $\mu(A\triangle\bigcup_{j=1}^tB_r(x_j))<\epsilon$.
Noting that the set $(\sigma,\sigma+\delta)$ is uncountable and that for each $j\in\{1,...,t\}$ and each $r\in(\sigma,\sigma+\delta)$, $\partial B_r(x_j)\subseteq\{y\in X\,|\,d(y,x_j)=r\}$, we see that there is an $r_0\in (\sigma,\sigma+\delta)$ with the property that for each $j\in\{1,...,t\}$, $\mu(\partial B_{r_0}(x_j))=0$.\\
$\bullet$ \textit{ Defining the set $E$.} Set $V=\bigcup_{j=1}^tB_{r_0}(x_j)$ and pick $x_0\in X$ arbitrarily. We let  
$$E:=\{g\in G\,|\,T_gx_0\in V\}.$$
$\bullet$ \textit{ Proving that \eqref{3.eq:UniformLimitWithError} holds.} Let $k\in\N$ and $g_1,...,g_k\in G$ be arbitrary. We will first show that $$d_{{\rm U}}(\bigcap_{s=1}^k(E-g_s))=\mu(\bigcap_{s=1}^kT_{-g_s}V).$$ 
To do this, let $(\Phi_N)_{N\in\N}$ be a F{\o}lner sequence in $G$. It suffices to show that 
$$\lim_{N\rightarrow\infty}\frac{|\bigcap_{s=1}^k(E-g_s)\cap \Phi_N|}{|\Phi_N|}=\mu(\bigcap_{s=1}^kT_{-g_s}V).$$
Since $T_{-g_1},...,T_{-g_k}$ are homeomorphisms, one has that for every $s\in\{1,...,k\}$,
$$T_{-g_s}\partial V=\partial(T_{-g_s}V)\subseteq \bigcup_{j=1}^t\partial(T_{-g_s} B_{r_0}(x_j))=T_{-g_s}(\bigcup_{j=1}^t(\partial B_{r_0}(x_j)).$$
 Thus, $\mu(\partial(\bigcap_{s=1}^k T_{-g_s}V))\leq \mu(\bigcup_{s=1}^k T_{-g_s}\partial V)) =0$. By \eqref{3.eq:UniformLimitForMeasurable}, we now obtain,
 \begin{multline*}
     \lim_{N\rightarrow\infty}\frac{|\bigcap_{s=1}^k(E-g_s)\cap \Phi_N|}{|\Phi_N|}=\lim_{N\rightarrow\infty}\frac{1}{|\Phi_N|}\sum_{g\in \Phi_N}\prod_{s=1}^k\mathbbm 1_E(g+g_s)\\
     =\lim_{N\rightarrow\infty}\frac{1}{|\Phi_N|}\sum_{g\in \Phi_N}\prod_{s=1}^k\mathbbm 1_{T_{-g_s}V}(T_gx_0)=\lim_{N\rightarrow\infty}\frac{1}{|\Phi_N|}\sum_{g\in \Phi_N}\mathbbm 1_{\bigcap_{s=1}^kT_{-g_s}V}(T_gx_0)=\mu(\bigcap_{s=1}^kT_{-g_s}V),
 \end{multline*}
 proving the claim. To prove formula \eqref{3.eq:UniformLimitWithError}, note that  
 \begin{multline*}
     |d_{{\rm U}}(\bigcap_{s=1}^k(E-g_s))- \mu(\bigcap_{s=1}^k T_{-g_s}A)|\\
     \leq |d_{{\rm U}}(\bigcap_{s=1}^k(E-g_s))- \mu(\bigcap_{s=1}^k T_{-g_s}V)|
    + \mu(\bigcap_{s=1}^k T_{-g_s}A\triangle \bigcap_{s=1}^k T_{-g_s}V)\\
    =0+  \mu(\bigcap_{s=1}^k T_{-g_s}A\triangle \bigcap_{s=1}^k T_{-g_s}V)
    =\int_X|\prod_{s=1}^k\mathbbm 1_{T_{-g_s}A}-\prod_{s=1}^k\mathbbm 1_{T_{-g_s}V}|\text{d}\mu<k\epsilon,
 \end{multline*}
 where the last inequality in the expression above follows from formula \eqref{eq:IntegralZeroIdentity}.
\end{proof}
\begin{rem}
    By suitably adapting the formulation of \cref{3.Lem:epsilonInverse} and invoking Rosenthal's Jewett-Krieger theorem for amenable groups \cite{RosenthalAmenableJewettKrieger} (see also \cite[Theorem 2]{Weiss2025JewettKrieger}), one can extend \cref{3.Lem:epsilonInverse} to the setting of amenable groups.
\end{rem}
%%%%%%%%%%%%%%%%%%%%%%%%%%%%%%%%%%%%%%%J
%%%%%%%%%%%%%%%%%%%%%%%%%%%%%%%%%%%%%%%%%%%%%%%%%

%%%%%%%%%%%%%%%%%%%%%%%%%%%%%%%%%%%%%%%%%%%%%%%%%%%%%%%%%%%%%%%%%%%%%%%%%%%
\appendix
\counterwithin{thm}{section}
\renewcommand{\thethm}{\Alph{section}\arabic{thm}}
\renewcommand{\thesubsection}{\Alph{section}.\arabic{subsection}}
\renewcommand{\thesection}{\Alph{section}}%   A, B, C, ...
%  \renewcommand{\thesubsection}{\thesection.\arabic{subsection}
%%%%%%%%%%%%%%%%%%%%%%%%%%%%%%%%%%%%%%%%%%%%%%%%%%%%%%%%%%%%%%%%%%%%%%%%%%%%
\section{A consequence of the Density Polynomial Hales-Jewett conjecture}\label{Sec:A}
 In this appendix we show that the Density Polynomial Hales-Jewett (DPHJ) conjecture \cite[p. 56]{ERTaU}, which we formulate for convenience of the reader in Subsection \ref{A.sec.DPHJConjecture} below, implies the following  result dealing with polynomial multiple recurrence. This result was discussed at the end of the Introduction and forms a conditional refinement of the corresponding restricted version of the  IP-polynomial Szemer{\'e}di theorem to the so called "polynomial corners". Note that we make no assumption about the linear independence of $p_1,...,p_\ell\in\Z[x_1,...,x_d]$. 
\begin{thm}\label{8.thm:ConditionalSzemeredi}
     Suppose that the Density Polynomial Hales-Jewett conjecture holds. Let $D,\ell\in\N$
     and let $\delta\in (0,1)$.  There is a constant $c=c(\ell,D,\delta)>0$ and an $r=r(\ell,D,\delta)\in\N$ with the property that for any 
     probability space  $(X,\mathcal A,\mu)$, any  invertible and commuting $\mu$-preserving transformations  $T_1,...,T_\ell$, any non-constant polynomials $p_1,...,p_\ell\in\Z[x_1,...,x_d]$ with zero constant term and $\max_{1\leq j\leq \ell}\deg_{x_1,...,x_d}p_j\leq D$,
     and any $A\in\mathcal A$ with $\mu(A)\geq \delta$, the set 
     \begin{equation}
\{\vec n\in \Z^d\,|\,\mu(A\cap T_1^{-p_1(\vec n)}A\cap\cdots\cap T_\ell^{-p_\ell(\vec n)}A)\geq c(\ell,D,\delta)\}
     \end{equation}
     is \rm{IP$_r^*$}.
 \end{thm}
We will derive \cref{8.thm:ConditionalSzemeredi} from \cref{8.ConditionalSzemeredi}  which we will prove in Subsection \ref{Sec:A3} and which is a more general statement concerning the so-called \textit{VIP-systems} introduced in \cite{BFM} (see formula \eqref{7.eq:VIPDefn1} for the definition).
We note that, should the Density Polynomial Hales-Jewett conjecture be confirmed, \cref{8.ConditionalSzemeredi} would yield a refined version of the IP-polynomial Szemer{\'e}di theorem in its full generality.
 %%%%%%%%%%%%%%%%%%%%%%%%%%%%%%%%%%%%%%%%%%%%%
 \subsection{The Density Polynomial Hales-Jewett conjecture}\label{A.sec.DPHJConjecture}
We begin by stating the Density Polynomial Hales-Jewett  conjecture as it was originally formulated in \cite[p. 56]{ERTaU}.\\
For any  $q,D,a\in\N$ we let $\mathcal M_{q,D,a}$ denote the set of all $q$-tuples of subsets of $\{1,...,a\}^D$. More formally,  
 $$\mathcal M_{q,D,a}:=\{(\alpha_1,...,\alpha_q)\,|\,\forall i\in\{1,...,q\},\,\alpha_i\subseteq \{1,...,a\}^D\}.$$
 \begin{namedthm*}{Density Polynomial Hales-Jewett conjecture (DPHJ)} Let $q,D\in\N$ and let $\delta>0$. There exists a positive integer 
 \begin{equation}\label{A.eq:DPHJConstant}
 C=C(q,D,\delta)
 \end{equation}
 such that for any $N\in\N$ with $N\geq C$ and any $\mathcal S\subseteq \mathcal M_{q,D,N}$ with $\frac{|\mathcal S|}{|\mathcal M_{q,D,N}|}>\delta$, there exists a non-empty $\gamma\subseteq \{1,...,N\}$ and an $(\alpha_1,...,\alpha_q)\in \mathcal S$ such that: 
 \begin{enumerate}[(i)]
\item For every $j\in\{1,...,q\}$, $\alpha_j\cap \gamma^D=\emptyset$.
\item One has
\begin{equation*}
(\alpha_1\cup \gamma^D,\alpha_2,...,\alpha_q),(\alpha_1,\alpha_2\cup \gamma^D,...,\alpha_q),\dots,
(\alpha_1,...,\alpha_{q-1}\cup\gamma^D,\alpha_q), (\alpha_1,...,\alpha_{q-1},\alpha_q\cup\gamma^D)\in\mathcal S.
\end{equation*}
\end{enumerate}
 \end{namedthm*}

 %%%%%%%%%%%%%%%%%%%%%%%%%%%%%%%%%%
 \subsection{Background on VIP-systems}
Let $(G,+)$ be an abelian  group and let $\mathcal F$ denote the family of all finite, non-empty subsets of $\N$. An $\mathcal F$-sequence $(g_\alpha)_{\alpha\in\mathcal F}$ in $G$  is called a \textit{VIP-system} if there is a $t\in\N$ with the property that  for any disjoint $\alpha_0,...,\alpha_t\in\mathcal F$,
\begin{equation}\label{7.eq:VIPDefn1}
\sum_{k=1}^{t+1}\sum_{0\leq i(1)<\cdots< i(k)\leq t}(-1)^kg_{\alpha_{i(1)}
\cup\cdots\cup \alpha_{i(k)}}=0.
\end{equation}
We say that the VIP-system $(v_\alpha)_{\alpha\in\mathcal F}$ has degree $D$ if $D$ is the least natural number $t$ for which \eqref{7.eq:VIPDefn1} holds. Our goal in this subsection is to prove the following lemma.  Recall that for each $D\in\N$, $\mathcal F_\emptyset(\N^D)$ denotes the family of all  finite subsets of $\N^D$ (so, in particular, $\emptyset\in \mathcal F_\emptyset (\N^D)$). 
\begin{lem}[Cf. formula for $\varphi_{v,\eta}$ on page 66 of \cite{BerLeibPolyHJ}]\label{7.lem:ShiftToSetPolynomial}
    Let  $D\in\N$ and let $(G,+)$ be an abelian group. For any VIP-system $(g_\alpha)_{\alpha\in\mathcal F}$ in $G$ of degree at most $D$ there is a function $\eta:\mathcal F_\emptyset (\N^D)\rightarrow G$ with the following two properties:
    \begin{enumerate}
       \item [(P.1)] For any $\alpha\in\mathcal F$,    $$\eta(\alpha^D)=\eta(\underbrace{\alpha\times\cdots\times\alpha}_{D\text{ times}})=g_\alpha.$$
         \item [(P.2)] For any $A,B\in\mathcal F_\emptyset (\N^D)$ with $A\cap B=\emptyset$, $\eta(A\cup B)=\eta(A)+\eta(B)$.
    \end{enumerate}
\end{lem}
The proof of \cref{7.lem:ShiftToSetPolynomial} makes use of the following result from \cite{McCutcheonInfinitaryPVW}. For each $t\in\N$ and any non-empty set $S$, we will let $S^{(t)}$ denote the class of all non-empty, $t$-element subsets of $S$. 
\begin{prop}[Proposition 2.5 in \cite{McCutcheonInfinitaryPVW}]\label{7.prop:LevelDecomposition}
    Let $D\in\N$ and let $(G,+)$ be an abelian group. An $\mathcal F$-indexed sequence $(g_\alpha)_{\alpha\in\mathcal F}$ in $G$ is a VIP-system of degree at most $D$ if and only if there are  functions $\eta_t:\N^{(t)}\rightarrow G$, $t\in\{1,...,D\}$, with the property that for any $\alpha\in\mathcal F$,
$$g_\alpha=\sum_{t=1}^{\min\{|\alpha|,D\}}\sum_{\gamma\in\mathcal \alpha^{(t)}}\eta_t(\gamma).$$
\end{prop}
\begin{proof}[Proof of \cref{7.lem:ShiftToSetPolynomial}]
    Let $\eta_1,...,\eta_D$ be as in \cref{7.prop:LevelDecomposition}. We define $\eta:\mathcal F_\emptyset (\N^D)\rightarrow G$ as follows.
For any $A\in\mathcal F_\emptyset (\N^D)$, $A\neq \emptyset$,
$$\eta(A):=\sum_{t=1}^D\sum_{\substack{(i_1,...,i_D)\in A\\i_r< i_s\text{ for }r<s\leq t}}\left(\eta_t(\{i_1,...,i_t\})\prod_{j=t}^D\mathbbm 1_{\{i_t\}}(i_j)\right)$$
and set $\eta(\emptyset)=0$.
It follows that for any $\alpha\in\mathcal F$,   
\begin{equation*}
\eta(\alpha^D)=\sum_{t=1}^D\sum_{\substack{(i_1,...,i_D)\in \alpha^D\\i_r< i_s\text{ for }r<s\leq t}}\left(\eta_t(\{i_1,...,i_t\})\prod_{j=t}^D\mathbbm 1_{\{i_t\}}(i_j)\right)
=\sum_{t=1}^{\min\{|\alpha|,D\}}\sum_{\gamma\in\alpha^{(t)}}\eta_t(\gamma),
\end{equation*}
which by \cref{7.prop:LevelDecomposition} implies that (P.1) holds.\\
The fact that  (P.2) holds is an immediate consequence of the definition of $\eta$. We are done.
\end{proof}

%%%%%%%%%%%%%%%%%%%%%%%%%%%%%%%%%%%%%%%%%%%%%%%%%%%%%  
\subsection{A conditional VIP Szemer{\'e}di theorem and the proof of \cref{8.thm:ConditionalSzemeredi}}\label{Sec:A3}
Invoking the inclusion-exclusion principle, it is not hard to see that for any $t\in \N$ and any sequences $(n_k^{(1)})_{k\in\N}$,...,$(n_k^{(t)})_{k\in\N}$ in $\Z$, the map 
$$\alpha\mapsto n_\alpha^{(1)}\cdots n_\alpha^{(t)}$$ 
is a VIP-system of degree at most $t$. It follows that for any polynomial $p\in\Z[x_1,...,x_d]$  with zero constant term and any sequence $(\vec n_k)_{k\in\N}$ in $\Z^d$, one has that the VIP-system  $p(\vec n_\alpha):=p(\sum_{j\in\alpha}\vec n_j)$ has degree at most  $D=\deg_{x_1,...,x_d}p$ \cite[Proof of Theorem 1.8]{BFM},\cite[p. 1001]{BHM}.  
Thus, \cref{8.thm:ConditionalSzemeredi} follows from the following more general result for VIP-systems.
 \begin{thm}\label{8.ConditionalSzemeredi}
Let $D,\ell\in\N$ and let $\delta>0$. Suppose that the Density Polynomial Hales-Jewett conjecture holds and put  
$$r:=C(\ell,D,\delta/2)\in\N\text{ and }  c:=\frac{\delta}{2^{\ell r^D+r+1}}>0,$$
where $C(\ell,D,\delta/2)$ is as defined in \eqref{A.eq:DPHJConstant}.\\
Then, for any probability space $(X,\mathcal A,\mu)$, any countable abelian group $(G,+)$, any $\mu$-preserving and commuting $G$-actions $(T^{(1)}_g)_{g\in G}$,...,$(T^{(\ell)}_g)_{g\in G}$, any VIP-systems $v_{1},...,v_{\ell}:\mathcal F\rightarrow G$ of degree at most $D$, any $A\in\mathcal A$ with $\mu(A)\geq \delta$, and any $\alpha\in\mathcal F$ with $|\alpha|\geq r$, there 
is  a non-empty $\gamma\subseteq \alpha$ such that 
 \begin{equation}\label{7.eq:GammaREturn}
\mu(A\cap T_{-v_1(\gamma)}^{(1)}A\cap\cdots\cap  T_{-v_\ell(\gamma)}^{(\ell)}A)\geq c.
     \end{equation}
 \end{thm}
 \begin{proof}
    We imitate the proof of Theorem 6.15 in \cite{berMcCuIPPolySzemeredi}. By \cref{7.lem:ShiftToSetPolynomial}, for each  $j\in\{1,...,\ell\}$, there is a map $\eta_j:\mathcal F_\emptyset (\N^D)\rightarrow G$ with the properties that for every $\alpha\in\mathcal F$, $v_{j}(\alpha)=\eta_j(\alpha^D)$ and for any disjoint $A,B\in \mathcal F_\emptyset(\N^D)$, $\eta_j(A\cup B)=\eta_j(A)+\eta_j(B)$. 
    For any $\alpha_1,...,\alpha_\ell\in\mathcal F_\emptyset (\N^D)$, let
     $$\varphi(\alpha_1,...,\alpha_\ell):=\prod_{j=1}^\ell T_{- \eta_j(\alpha_j)}^{(j)}.$$
Without loss of generality assume that $|\alpha|=r$ and let 
$$\mathcal M=\{(\beta_1,...,\beta_\ell)\,|\,\beta_1,...,\beta_\ell\subseteq \alpha^D\}.$$
We define the function $f:X\rightarrow [0,1]$, by
$$f(x)=\frac{1}{|\mathcal M|}\sum_{(\alpha_1,...,\alpha_\ell)\in\mathcal M} \mathbbm 1_{\varphi(\alpha_1,...,\alpha_\ell)A}(x).$$
For each $x\in X$, we set 
$$S_x=\{(\alpha_1,...,\alpha_\ell)\in\mathcal M\,|\,\mathbbm 1_{\varphi(\alpha_1,...,\alpha_\ell)A}(x)=1\}.$$
Note that $|S_x|=|\mathcal M|f(x)$.\\
 Let 
$$Y=\{x\in X\,|\,f(x)> \frac{
\mu(A)}{2}\}.$$
Because $\int_Xf(x)\text{d}
\mu(x)=\mu(A)$ and $0\leq f(x)\leq 1$ for each $x\in X$, we must have 
$$
\mu(Y)\geq \frac{\mu(A)}{2}\geq \frac{\delta}{2}.
$$
It follows that for every $x\in Y$, $\frac{|S_x|}{|\mathcal M|}> \frac{\delta}{2}$ and, so, by the DPHJ conjecture, there exist $\alpha_{1,x},...,\alpha_{\ell,x}\subseteq \alpha^D$ and a non-empty $\gamma_x\subseteq \alpha$ such that for every $j\in\{1,...,\ell\}$, $\alpha_{j,x}\cap \gamma_x^D=\emptyset$  and 
\begin{equation}\label{8.eq:WildCard}
(\alpha_{1,x},...,\alpha_{\ell,x}),(\alpha_{1,x}\cup\gamma_x^D,...,\alpha_{\ell,x}),(\alpha_{1,x},\alpha_{2,x}\cup\gamma_x^D,...,\alpha_{\ell,x}),...,(\alpha_{1,x},...,\alpha_{\ell,x}\cup\gamma_x^D)\in\mathcal S_x.
\end{equation}
 Because there are no more than  $2^{\ell r^D}$ choices for $(\alpha_{1,x},...,\alpha_{\ell,x})$ and no more than $2^r$ choices for $\gamma_x\subseteq \alpha$, we can find $(\alpha_1,...,\alpha_\ell)\in \mathcal M$ and a non-empty $\gamma\subseteq \alpha$ with $\gamma^D$ disjoint from  $\alpha_1,...,\alpha_\ell$ for which
 \begin{multline*}
     \mu(A\cap T_{-v_1(\gamma)}^{(1)}A\cap \cdots\cap T_{-v_\ell(\gamma)}^{(\ell)}A)
     =\mu(A\cap T_{-\eta_1(\gamma^D)}^{(1)}A\cap \cdots\cap T_{-\eta_\ell(\gamma^D)}^{(\ell)}A)\\
    =\mu(\varphi(\alpha_1,...,\alpha_\ell)A\cap \varphi(\alpha_1\cup\gamma^D,....,\alpha_\ell)A\cap \cdots\cap \varphi(\alpha_1,...,\alpha_\ell\cup\gamma^D)A)\\
     \geq \mu(\{x\in X\,|\, (\alpha_1,...,\alpha_\ell),(\alpha_1\cup \gamma^D,...,\alpha_\ell),...,(\alpha_1,...,\alpha_\ell\cup \gamma^D)\in S_x\})\geq \frac{\delta}{2^{\ell r^D+r+1}}.
 \end{multline*} 
 We are done. 
 \end{proof}

%%%%%%%%%%%%%%%%%%%%%%%%%%%%%%%%%%%%%%%%%%%%%%%%%%%%%%%%%%%%%%%%
%%%%%%%%%%%%%%%%%%%%%%%%%%%%%%%%%%%%%%%%%%%%%%%%%%%%%%%%%%%%%%%%%%%%%%%%%%%%%%%%%%%%%%%%%%%%%%%%%%%%%%%%%%%%%%%%
  \bibliography{Bib.bib}
\bibliographystyle{amsplain}
\vspace{0.5em}
\noindent
{\small Vitaly Bergelson\\
\textsc{Department of Mathematics, The Ohio State University, Columbus, OH 43210, USA}\par\nopagebreak
\noindent
\href{mailto:vitaly@math.ohio-state.edu}
{\texttt{vitaly@math.ohio-state.edu}}
\vspace{0.5em}\\
\noindent
Rigoberto Zelada\\
\textsc{Mathematics Institute. University of Warwick, Coventry, CV4 7AL, UK}\par\nopagebreak
\noindent
\href{mailto:rzelada@umd.edu}
{\texttt{rigozelada@gmail.com}}}
%%%%%%%%%%%%%%%%%%%%
\end{document}